\documentclass{ims9x6}
\usepackage{amssymb,mathrsfs,graphicx,subfigure, enumerate}

\makeindex


\newcommand{\Un} {\mathbf{U}(n)}
\newcommand{\Um} {\mathbf{U}(m)}

\renewcommand{\d}{\textup{d}}

\newcommand{\dt}{\textup{d}t}
\newcommand{\tF}{\textup{F}}

\newcommand{\kp}{\kappa}
\newcommand{\dg}{\dagger}

\newcommand{\bbr}{\mathbb R}
\newcommand{\bbc}{\mathbb C}

\newcommand{\bbs}{\mathbb S}
\newcommand{\bbz}{\mathbb Z}

\def\mi {{\mathrm i}}
\newcommand{\p}{\partial}

\newcommand{\ba}{\begin{aligned}}
\newcommand{\ea}{\end{aligned}}

\newcommand{\be}{\begin{equation}}
\newcommand{\ee}{\end{equation}}

\newcommand{\bea}{\begin{eqnarray}}
\newcommand{\eea}{\end{eqnarray}}

\newcommand{\lela}{\langle}
\newcommand{\rira}{\rangle}
\newcommand{\fl}{\frac}

\begin{document}

\setcounter{page}{1}  

\chapter{COLLECTIVE DYNAMICS OF LOHE TYPE AGGREGATION MODELS}


\markboth{Seung-Yeal Ha and Dohyun Kim}{Collective dynamics of Lohe type aggregation models} 

%

\author{Seung-Yeal Ha}  

\address{Department of Mathematical Sciences and Research Institute of Mathematics\\
Seoul National University, Seoul 08826, and \\
Korea Institute for Advanced Study, \\
Hoegiro 85, Seoul 02455, Republic of Korea\\
syha@snu.ac.kr} 

\author{Dohyun Kim}  

\address{School of Mathematics, Statistics and Data Science\\
Sungshin Women's University\\
Seoul 02844, Republic of Korea \\
dohyunkim@sungshin.ac.kr} 
%
%

%
%

\begin{abstract}
In this paper, we review state-of-the-art results on the collective behaviors for Lohe type first-order aggregation models. Collective behaviors of classical and quantum  many-body systems have received lots of attention from diverse scientific disciplines such as applied mathematics, control theory in engineering, nonlinear dynamics of statistical physics, etc. To model such collective dynamics, several phenomenological models were proposed in literature and their emergent dynamics were extensively studied in recent years. Among them, we present two Lohe type models: {\it the Lohe tensor(LT) model} and {\it the Schr\"odinger-Lohe(SL) model}, and present several sufficient conditions in unified frameworks via the Lyapunov functional approach for state diameters and dynamical systems theory approach for two-point correlation functions. We also present several numerical simulation results for the SL model. 
\end{abstract}

\section{Introduction}
\setcounter{equation}{0}
The purpose of this paper is to continue a recent review \cite{H-K-P-Z}  on the collective behaviors of classical and quantum synchronization (or aggregation) models.  From the beginning of this century, collective behaviors of many-body systems have received a lot of  attention from various scientific disciplines, e.g., synchronization and swarming behaviors in biology \cite{B-B,Pe,T-B-L,T-B,Wi1,Wi2}, decentralized control in engineering \cite{L-P-L-S,P-L-S,Re}, non-convex stochastic optimization algorithms in machine learning community \cite{C-C-T-T,C-J-L-Z,F-H-P-S1,F-H-P-S2,K-T,K-E,P-T-T-M}, etc. Mathematical modeling of such collective behaviors was begun by several pioneers, Winfree \cite{Wi2}, Kuramoto \cite{Ku2} and Vicsek \cite{Vi} whose works have established milestones of collective dynamics as major research subjects in applied mathematics, control theory, statistical physics, etc. Since then, several phenomenological models have been proposed, to name a few, the Cucker-Smale model \cite{C-S}, the swarm sphere model \cite{C-C-H,C-H1,C-H2,C-H4,H-K-L-N,O1}, matrix aggregation models \cite{B-C-S,D-F-M-T,D,G-H,H-K2,Lo2,Lo3}, etc. We also refer the reader to survey articles \cite{A-B,A-B-F,D-B1,H-K-P-Z,M-T,P-R-K,St,VZ,Wi2} for a brief introduction to collective dynamics. 

In what follows, we are mainly interested in the Lohe type aggregation models (the LT model and the SL model).  The LT model is a finite-dimensional aggregation model on the space of tensors with the same rank and size, and it encompasses previously introduced first-order aggregation models, whereas the SL model is an infinite-dimensional toy model describing synchronous behaviors in a quantum regime. In fact, Lohe first focused on the similarity between classical and quantum synchronizations and proposed a merely phenomenological model to capture common properties between two emergent behaviors in different regimes. However, when we focus on the dissimilarity between classical and quantum systems, one encounters several limitations of the SL model due to its quantum nature, namely \textit{entanglement} which is a genuine quantum feature and is not observed in the classical world. We here omit detailed descriptions on the relation between quantum synchronization and entanglement which is beyond our scope. 

In this paper, we briefly review state-of-the-art results on the collective dynamics of the aforementioned two Lohe type aggregation models from a universal  platform for  collective behaviors, Lyapunov functional approach and dynamical systems theory approach. 

The rest of the paper is organized as follows. In Section \ref{sec:2}, we introduce two Lohe type aggregation models, namely the LT model and the SL model. In Section \ref{sec:3}, we review state-of-the-art results on the emergent dynamics of the LT model for homogeneous and heterogeneous ensembles by providing several sufficient frameworks for the complete state aggregation and practical aggregation. In Section \ref{sec:4}, we review parallel results for the SL model compared to the LT model in Section \ref{sec:3}. Finally, Section \ref{sec:5} is devoted to a brief summary and discussion on some remaining interesting issues for a future direction. 
  
 \vspace{1cm}
 
\noindent {\bf Notation}: Throughout the paper, we will see several models. As long as there are no confusion, we use handy acronyms for such models:
\begin{align*}
\begin{aligned}
& \mbox{LT:~ Lohe tensor}, \quad   \mbox{SL:~Schr\"{o}dinger-Lohe}, \quad \mbox{LM:~ Lohe matrix,}  \\
&  \mbox{LHS:~ Lohe hermitian sphere}, \quad   \mbox{SDS:~swarm double sphere}, \\
&  \mbox{SDM:~swarm double matrix}. 
\end{aligned}
\end{align*}
Moreover, we use $| \cdot |$ to denote $\ell^2$-norm of vectors in $\bbr^d$ or $\bbc^d$, where  $\| \cdot \|$ represents $L^2$-norm.  

\section{Preliminaries} \label{sec:2} 
\setcounter{equation}{0}
In this section, we briefly introduce two first-order aggregation models whose emergent dynamics will be discussed in the following two sections, namely {\it ``the Lohe tensor model''} and {\it ``the Schr\"odinger-Lohe model''}, separately.

\subsection{The Lohe tensor model} \label{sec:2.1}
To set up the stage, we begin with basic terminologies on tensors. A complex rank-$m$ tensor can be visualized as a  multi-dimensional array of complex numbers with multi-indices. The {\it ``rank}'' (or  {\it ``order}'') of a tensor is the number of indices, say a rank-$m$ tensor with size $d_1 \times \cdots \times d_m$ is an element of ${\mathbb C}^{d_1 \times \cdots \times d_m}$. For example, scalars, vectors and matrices correspond to rank-0, 1 and 2 tensors, respectively.  

Let $T$ be a rank-$m$ tensor with a size $d_1 \times \cdots \times d_m$.  Then, we denote $(\alpha_1, \cdots, \alpha_m)$-th component of  $T$ by $[T]_{\alpha_1 \cdots \alpha_m}$, and we also set $\overline{T}$ by the rank-$m$ tensor whose components are simply the complex conjugate of the corresponding elements in $T$:
\[ [\overline{T}]_{\alpha_1 \cdots \alpha_m} :=\overline{[T]_{\alpha_1 \cdots \alpha_m}}, \quad 1 \leq \alpha_i \leq d_i,~~1 \leq i \leq m. \]
In other words, each component of $\overline T$ is defined as the complex conjugate of the corresponding element of $T$.  Let ${\mathcal T}_m(\bbc; d_1 \times\cdots\times d_m)$ be the collection of all rank-$m$ tensors with size $d_1 \times\cdots\times d_m$. For notational simplicity, we set
\[  {\mathbb C}^{d_1 \times \cdots \times d_m}:= {\mathcal T}_m(\bbc; d_1 \times\cdots\times d_m). \]
Then, it is a complex vector space. Several well-known first-order aggregation models, for instance, the Kuramoto model \cite{Ku2}, the swarm sphere model \cite{O1} and the Lohe matrix model \cite{Lo3} can be regarded as aggregation models on the subsets of  $\bbr, \bbc^d$ and $\bbc^{d \times d}$, respectively. Furthermore, we also introduce the following handy notation:~for $T \in \bbc^{d_1 \times \cdots \times d_m}$ and $A \in \bbc^{d_1 \times \cdots \times d_m \times d_1 \times \cdots \times d_m}$, we set
\begin{align*}
\begin{aligned}
& [T]_{\alpha_{*}}:=[T]_{\alpha_{1}\alpha_{2}\cdots\alpha_{m}}, \quad [T]_{\alpha_{*0}}:=[T]_{\alpha_{10}\alpha_{20}\cdots\alpha_{m0}},  \quad  [T]_{\alpha_{*1}}:=[T]_{\alpha_{11}\alpha_{21}\cdots\alpha_{m1}}, \\
&  [T]_{\alpha_{*i_*}}:=[T]_{\alpha_{1i_1}\alpha_{2i_2}\cdots\alpha_{mi_m}}, \quad [T]_{\alpha_{*(1-i_*)}}:=[T]_{\alpha_{1(1-i_1)}\alpha_{2(1-i_2)}\cdots\alpha_{m(1-i_m)}}, \\
&  [A]_{\alpha_*\beta_*}:=[A]_{\alpha_{1}\alpha_{2}\cdots\alpha_{m}\beta_1\beta_2\cdots\beta_{m}}.
\end{aligned}
\end{align*}
Moreover, we can associate inner product $\langle \cdot, \cdot \rangle_\tF$, namely {\it ``Frobenius inner product''} and its induced norm $\| \cdot \|_\tF$ on $ \bbc^{d_1 \times \cdots \times d_m}$:~for $T, S \in  \bbc^{d_1 \times \cdots \times d_m}$, 
\[ \langle T, S \rangle_\tF := \sum_{\alpha_* \in \prod_{i=1}^{m} \{1, \cdots, d_i\} } [{\bar T}]_{\alpha_*} [S]_{\alpha_*}, \quad \|T \|_\tF^2 := \langle T, T \rangle_\tF. \]
Let  $A_j$ be a {\it block skew-hermitian} rank-$2m$ tensor with size $(d_1 \times\cdots\times d_m) \times (d_1 \times \cdots\times d_m)$ such that
\begin{equation*}
[\bar A_j]_{\alpha_{*0} \alpha_{*1}} = -[A_j]_{\alpha_{*1} \alpha_{*0}}.
\end{equation*}
In other words, if two blocks with the first $m$-indices are interchanged with the rest $m$-indices, then its sign is changed. 

Now, we are ready to introduce the LT model on the finite ensemble $\{T_j \}_{j=1}^{N} \subset \bbc^{d_1 \times \cdots \times d_m}$: 
\begin{align} \label{LT}
\begin{aligned} 
&\dot{[T_j]}_{\alpha_{*0}} = [A_j]_{\alpha_{*0}\alpha_{*1}}[T_j]_{\alpha_{*1}} \\
&\hspace{0.5cm}+ \sum_{i_* \in \{0, 1\}^m}\kappa_{i_*} \Big([T_c]_{\alpha_{*i_*}}\bar{[T_i]}_{\alpha_{*1}}[T_i]_{\alpha_{*(1-i_*)}}-[T_i]_{\alpha_{*i_*}}\bar{[T_c]}_{\alpha_{*1}}[T_i]_{\alpha_{*(1-i_*)}} \Big), 
\end{aligned}
\end{align}
where $\kappa_{i_*}$'s are coupling strengths, $T_c := \frac{1}{N} \sum_{k=1}^{N} T_k$ is the f average of $\{T_j\}_{j1}^N$, and we used the Einstein summation convention for repeated indices in the R.H.S. of \eqref{LT}.  Although the LT model \eqref{LT} looks so complicated, interaction terms inside the parenthesis of \eqref{LT} are designed to include interactions for the Kuramoto model, the swarm sphere model and the Lohe matrix model, and they certainly have ``gain term$-$loss term''  structure and are cubic in nature. These careful designs of interactions  lead to the existence of a constant of motion. 
\begin{proposition} \label{P2.1}
\emph{\cite{H-P7}}
Let $\{ T_j \}$ be a solution to \eqref{LT}.  Then, for each $j = 1, \cdots, N$, $\|T_j \|_\tF$ is a first integral for \eqref{LT}:
\[  \|T_j(t) \|_\tF = \|T_j(0) \|_\tF, \quad t \geq 0.\]
\end{proposition}
\subsection{The Schr\"odinger-Lohe model} \label{sec:2.2}
Consider a network consisting of $N$ nodes denoted by $1, \cdots, N,$ and we assume that the network topology is registered by an adjacent weighted matrix $(a_{ij})$. 

For each $j = 1, \cdots, N$, let $\psi_j=\psi_j(x,t)$ be the wave-function of the  quantum subsystem lying on the $j$-th node. Then, we assume that  the spatial-temporal dynamics of $\psi_j$ is governed by the Cauchy problem to the SL model:  for $(x,t)\in \bbr^d\times \bbr_+$, 
\begin{equation} \label{S-L}
\begin{cases}
\displaystyle \mi \p_t \psi_j = - \frac{1}{2} \Delta \psi_j + V_j\psi_j + \frac{\mi\kp}{2N}\sum_{k=1}^N a_{jk} \left( \psi_k - \frac{ \langle \psi_j,\psi_k\rangle}{\langle \psi_j,\psi_j\rangle} \psi_j \right), \\
\displaystyle \psi_j(x,0) = \psi_j^0(x), \quad \|\psi_j^0 \| = 1,  \quad j=1,\cdots,N,
\end{cases}
\end{equation}
where $\kappa$ is a coupling strength, and we have taken mass and  normalized Planck's constant to be unity for the simplicity of presentation.  Like the classical Schr\"odinger equation, system \eqref{S-L} satisfies $L^2$-conservation.
\begin{proposition}
Let $\{ \psi_j \}$ be a solution to \eqref{S-L}. Then, $\|\psi_j(t)\|$ is a first integral:
\[ \|\psi_j(t)\| = \| \psi_j^0 \|, \quad t \geq 0, \quad 1 \leq j \leq N. \]
\end{proposition}
Then, $L^2$-conservation and standard energy estimates yield a global existence of unique weak solutions to \eqref{S-L}.
\begin{theorem}  \label{T2.1}
\emph{\cite{A-M,H-H}}
Suppose initial data satisfy $\psi_j^0 \in L^2(\bbr^d)$ for each $j= 1, \cdots, N$. Then, the Cauchy problem   \eqref{S-L} admits a unique global-in-time weak solution satisfying  
\[
\psi_j \in C([0,\infty); L^2(\bbr^d)), \quad j = 1, \cdots, N.
\]
Moreover, if we assume $\psi_j^0 \in H^1(\bbr^d)$, then the corresponding global weak solution satisfies $\psi_j \in C([0,\infty);H^1(\bbr^d))$. 
\end{theorem}

\vspace{0.5cm}

Before we close this section, we show that the SL model \eqref{S-L} reduces to the Kuramoto model  as a special case. For this, we assume
\[ V_{j}(x) =  :\nu_j: \mbox{constant}, \qquad \psi_j(x,t) =: e^{-{\mathrm i} \theta_j(t)}, \quad (x, t) \in \bbr^d \times \bbr_+. \]
We substitute the above ansatz into the SL model \eqref{S-L} to get
\[ {\dot \theta}_j \psi_j = \nu_j \psi_j + \frac{{\mathrm i} \kappa}{N} \sum_{k=1}^{N} a_{jk} \Big(\psi_k - e^{-{\mathrm i}(\theta_j - \theta_k)} \psi_j  \Big). \]
Then, we multiply  $\overline{\psi}_j$ to the above relation, use $|\psi_j(t)|^2 =1$, and compare the real part of the resulting relation to obtain the Kuramoto model \cite{Ku2}:
\[  {\dot \theta}_j = \nu_j + \frac{{\bar \kappa}}{N} \sum_{k=1}^{N} a_{jk} \sin (\theta_k - \theta_j), \quad {\bar \kappa} := 2 \kappa.   \]

\vspace{0.5cm}

In the following two sections, we review emergent dynamics of the LT model and the SL model, separately. Most presented results in the following section will be provided without detailed proofs, but we may discuss brief ideas or key ingredients to give a feeling to see how proofs go.  

\section{The Lohe tensor model} \label{sec:3} 
\setcounter{equation}{0}
In this section, we review the emergent dynamics of the LT model and two explicit low-rank LT models which can be related to the swarm sphere model and the Lohe matrix model. 

\subsection{Emergent dynamics} \label{sec:3.1}
 In this subsection, we review emergent dynamics of the Lohe tensor model \eqref{LT}. First, we recall two concepts of aggregations (or synchronizations) as follows.
\begin{definition} \label{D3.1}
\emph{\cite{H-P1,H-P5}}
Let $\{T_j \}$ be a finite ensemble of rank-m tensors whose dynamics is governed by \eqref{LT}. 
\begin{enumerate}
\item The ensemble exhibits   complete state aggregation (synchronization) if relative states tend to zero asymptotically:
\[ \lim_{t\rightarrow \infty}\max_{1\leq i,j\leq N } \| T_i(t)-T_j(t) \|_\tF = 0. \]
\item The ensemble exhibits practical aggregation (synchronization) if magnitudes of relative states can be controlled by the principal coupling strength $\kappa_{0\cdots0}$ as follows:
\[  \lim_{\kappa_{i_*} \to \infty}  \limsup_{t\rightarrow\infty} \max_{1\leq i,j\leq N } \| T_i(t)-T_j(t) \|_\tF =0,  \]
for some $i_* \in \{0, 1 \}^m$.
\end{enumerate}
\end{definition}
\begin{remark} The jargon ``{\it synchronization''} is  often used in control theory and physics communities instead of aggregation. In fact, synchronization represents an adjustment of rhythms in oscillatory systems. In contrast, our systems under consideration might not be oscillatory. Thus, the authors feel more comfortable to use aggregation instead of synchronization. 
\end{remark}

For a given state ensemble $\{T_j \}$ and free flow ensemble $\{A_j \}$, we define diameters for both ensembles:
\[
\mathcal D(T) := \max_{1\leq i,j\leq N} \|T_i  - T_j \|_\tF, \quad \mathcal D(A) := \max_{1\leq i,j\leq N} \|A_i  - A_j \|_\tF.
\]
Note that $\mathcal D(T)$ is time-varying and Lipschitz continuous. Thus, it is differentiable a.e. and $\mathcal D(A)$ is constant, since $A_j$ is a constant tensor. \newline

Let $\{T_j \}$ be a solution to \eqref{LT} with  $ \|T_j \|_\tF = 1.$ Then, after tedious and delicate analysis, one can derive a differential inequality for  $\mathcal D(T)$ (see Proposition 4.2 in \cite{H-P7}): for a.e. $t>0$, 
\begin{equation} \label{C-1}   
\left|{\d\over{\dt}} {\mathcal D}(T)+\kappa_0 {\mathcal D}(T) \right| 
\leq 2\kappa_0 {\mathcal D}(T)^2+ 2 \hat{\kappa}_0 \|T_c^{0}\|_\tF {\mathcal D}(T)+ {\mathcal D}(A),
\end{equation}
where $T_c^{0}, \kappa_0$ and ${\hat \kappa}_0$ are defined as follows:
\[ T_c^{0} := \frac{1}{N} \sum_{j=1}^{N} T_j^{0}, \qquad \kappa_0 := \kappa_{0\cdots 0}, \qquad  \hat{\kappa}_0 := \sum_{i_*\neq (0,\cdots, 0)}\kappa_{i_*} . \]
Depending on whether the ensemble $\{T_j \}$ has the same free flows or heterogeneous free flows, we have the following two cases:
\[ {\mathcal D}(A) = 0 : \mbox{homogeneous ensemble}, \quad {\mathcal D}(A) > 0 :  \mbox{heterogeneous ensemble}. \]
Then, the emergent dynamics of \eqref{LT} can be summarized as follows. 
\begin{theorem} \label{T3.1} \cite{H-P7}
The following assertions hold.
\begin{enumerate}
\item
(Emergence of complete state aggregation):~Suppose system parameters and   initial data $\{T_j^0 \}$ satisfy
\begin{align}
\begin{aligned} \label{C-2}
& {\mathcal D}(A) = 0, \quad  \|T_j^{0}\|_\tF = 1, \quad  j = 1, \cdots, N, \quad \kappa_0 > 0, \\
&  {\hat \kappa}_{0}  < \frac{\kappa_{0}}{2 \|T_c^{0}\|_\tF},\quad 0< {\mathcal D}(T^{0})<\frac{\kappa_{0}- 2{\hat \kappa}_0 \|T_c^{0}\|_\tF}{2\kappa_0},
\end{aligned}
\end{align}
and let $\{T_j \}$ be a global solution to \eqref{LT}. Then, there exist positive constants $C_0$ and $C_1$ depending on $\kappa_{i_*}$ and $\{ T_j ^{0} \}$ such that 
\[
C_0 e^{-\left(\kappa_{0}+ 2 \hat{\kappa}_0  \|T_c^{0}\|_\tF \right)t} \leq {{\mathcal D}(T(t))} \leq C_1 e^{-\left(\kappa_{0}- 2\hat{\kappa}_0 \|T_c^{0}\|_\tF  \right)t}, \quad t \geq 0.
\]
\item
(Emergence of practical aggregation):~Suppose system parameters and initial data satisfy
\begin{align}
\begin{aligned} \label{C-2-1}
& \|T_j^{0}\|_\tF = 1, \quad  j = 1, \cdots, N, \quad \kappa_0 > 0, \\
& 0\leq {\mathcal D}(T^{0})\leq\eta_2,\quad 0 < {\mathcal D}(A)< \frac{|\kappa_0- 2\hat{\kappa}_0 \|T_c^{0}\|_\tF |^2}{8 \kappa_0},
\end{aligned}
\end{align}
where $\eta_2$ appearing in \eqref{C-2-1} is the largest positive root of the following quadratic equation:
\[
-2\kappa_0 x^2+(\kappa_{0}- 2\hat{\kappa}_0  \|T_c^{0}\|_\tF )x = {\mathcal D}(A).
\]
Let $\{T_j\}$ be a global solution to system \eqref{LT}. Then, practical aggregation emerges asymptotically:
\begin{equation} \label{C-2-2}
\lim_{ \kappa_0 \to \infty} \limsup_{t\rightarrow\infty} {\mathcal D}(T(t))=0.
\end{equation}
\end{enumerate}
\end{theorem}
\begin{proof}
For a detailed proof, we refer the reader  to \cite{H-P1}. Below, we instead provide a brief sketch on the key ingredient of proofs.  \newline

\noindent (i)~Suppose that ${\mathcal D}(A) = 0$. \newline

\noindent $\bullet$~Case A (Upper bound estimate):~It follows from \eqref{C-1} that 
\begin{equation} \label{C-4}
{\d\over{\dt}} {\mathcal D}(T)\leq  {\mathcal D}(T) \Big[ 2\kappa_0 {\mathcal D}(T)-(\kappa_{0}- 2\hat{\kappa}_0  \|T_c^{0}\|_\tF) \Big ], \quad \mbox{a.e.~$t > 0$}.
\end{equation}
Under the assumptions \eqref{C-2},   coefficients appearing in the R.H.S. of \eqref{C-4} satisfy
\[ 2\kappa_0 >0, \quad \kappa_{0}- 2\hat{\kappa}_0  \|T_c^{0}\|_\tF > 0. \]
Now, we directly solve the differential inequality \eqref{C-4} to derive  desired upper bound estimates. \newline

\noindent $\bullet$~Case B (Lower bound estimate): Again, it follows from \eqref{C-1} that 
\begin{equation} \label{C-5}
\frac{\d}{\dt} {\mathcal D}(T) \geq {\mathcal D}(T) \Big[ 2\kappa_0 {\mathcal D}(T) - (\kappa_0 + 2 {\hat \kappa}_0 \| T^{0}_c \|_\tF) \Big], \quad \mbox{a.e.}~t > 0. 
\end{equation}
Similar to Case A, we integrate \eqref{C-5} to find the desired lower bound estimate.  \newline

\noindent (ii)~~Suppose that ${\mathcal D}(A) > 0$. Then, it follows from \eqref{C-1} that 
\begin{equation} \label{F-2}
{\d\over{\dt}} {\mathcal D}(T)\leq 2\kappa_0 {\mathcal D}(T)^2-(\kappa_{0}- 2\hat{\kappa}_0  \|T_c^{0}\|_\tF) {\mathcal D}(T)+ {\mathcal D}(A), \quad \mbox{a.e.~$t > 0$}.
\end{equation}
In order to use a comparison principle, we introduce a quadratic function $f$ defined by  
\begin{equation} \label{F-3}
 f(x) :=-2\kappa_0 x^2+(\kappa_{0}- 2\hat{\kappa}_0  \|T_c^{0}\|_\tF )x.
\end{equation} 
Since we are interested in the regime $\kappa_0 \to \infty$, the term $\kappa_0-2\hat{\kappa}_0  \|T_c^{0}\|$ will be positive. Thus, it follows from \eqref{F-2} and \eqref{F-3} that 
\begin{equation*} \label{F-4}
\frac{\d}{\dt} {\mathcal D}(T)\leq{}{\mathcal D}(A)-f({\mathcal D}(T)), \quad \mbox{a.e.}~t > 0.
\end{equation*}
Considering the geometry of the graph of $f$, one can see that the quadratic equation $f(x)=D(A)$ has two positive roots $\eta_1$ and $\eta_2$ satisfying  
\[ 0<\eta_1< \frac{\kappa_0- 2\hat{\kappa}_0 \|T_c^{0}\|_\tF }{4\kappa_0} <\eta_2< \frac{\kappa_0- 2\hat{\kappa}_0 \|T_c^{0}\|_\tF }{2\kappa_0}. \]
Moreover, one can claim: (see Lemma 5.3 in \cite{H-P7}) 
\begin{enumerate}
\item
${\d\over{\dt}} {\mathcal D}(T(t))\leq{0}$ almost every $t>0$ when ${\mathcal D}(T(t))\in[\eta_1, \eta_2]$.
\vspace{0.1cm}
\item
$ \mathcal{S}(\eta_2):=\{ {\mathcal D}(T(t))< \eta_2 \}$ is a positively invariant set for the LT flow generated by \eqref{LT}.
\vspace{0.1cm}
\item
There exist $t_e\geq0$ such that 
\[ {\mathcal D}(T(t)) < \eta_1, \quad t \geq t_e. \]
\end{enumerate}
In fact, the smaller positive root $\eta_1$ can be calculated explicitly:
\[ \eta_1 = \frac{\mathcal D(A)}{\kappa_0- 2\hat{\kappa}_0 \|T_c^{0}\|_\tF } \left( \frac{2}{1+\sqrt{1-  \frac{8\kappa_0 \mathcal D(A)}{\kappa_0- 2\hat{\kappa}_0 \|T_c^{0}\|_\tF}}} \right). \]
Then, it follows from the claim above that 
\[ \limsup_{t\rightarrow\infty}\mathcal D(T(t) )\leq\eta_1={\mathcal D(A)\over{\kappa_0 - 2\hat{\kappa}_0 \|T_c^{0}\|_\tF }}{2\over{1+\sqrt{1-8 \kappa_0 \mathcal D(A)\over{\kappa_0-2 \hat{\kappa}_0 \|T_c^{0}\|_\tF }}}}.
\]
By letting $\kappa_0 \to \infty$, one has the desired estimate \eqref{C-2-2}. 
\end{proof}
\subsection{Low-rank LT models} \label{sec:3.2}
In the previous subsection, we have reviewed the emergent dynamics of the LT model in a general setting. In this subsection, we study two low-rank LT models which can be derived from the LT model on $\bbc^{d_1 \times d_2}$ and $\bbc^d$, respectively.

\subsubsection{The generalized Lohe matrix model} \label{sec:3.2.1}
In this part, we first derive a matrix aggregation model on $\bbc^{d_1 \times d_2}$ that can be reduced from the LT model. In the case of a square matrix $d_1 = d_2$, the Lohe matrix model serves as a first-order aggregation model on the subset of $\bbc^{d_1 \times d_2}$ which is a unitary group. Thus, the LT model can provide an aggregation model on the space of nonsquare matrices. More precisely, consider the Lohe tensor model \eqref{LT} with $m=2$:
\begin{align}
\begin{aligned} \label{C-6}
\dot{T}_j &= A_j T_j + \kappa_{00}(\mathrm{tr}(T_j^\dagger  T_j)T_c-\mathrm{tr}(T_c^\dagger  T_j)T_j) + + \kappa_{11}\mathrm{tr}(T_j^\dagger  T_c-T_c^\dagger  T_j)T_j \\
&\hspace{0.2cm} + \kappa_{10}(T_j T_j^\dagger  T_c-T_j T_c^\dagger  T_j) +  \kappa_{01}(T_cT_j^\dagger  T_j -T_jT_c^\dagger  T_j),
\end{aligned}
\end{align}
where $T_j^{\dagger}$ is a hermitian conjugate of $T_j$ and the free flow term $A_jT_j$ is defined as a rank-2 tensor via tensor contraction between a rank-4 tensor and a rank-2 tensor:
\[
[T_j]^{\dagger}_{\alpha \beta} = [{\overline T_j}]_{\beta \alpha}, \quad [A_jT_j]_{\alpha\beta}=[A_j]_{\alpha\beta\gamma\delta}[T_j]_{\gamma\delta}.
\]
For simplicity, we set
\[ \kappa_{00}=\kappa_{11}=0, \quad \kappa_1 := \kappa_{01}, \quad \kappa_2 := \kappa_{10}. \]
Then, system \eqref{C-6} reduces to the following simplified model, namely {\it ``the generalized Lohe matrix model''} \cite{H-P3}:
\begin{equation}
\begin{cases} \label{C-7}
\vspace{0.3cm} \displaystyle {\dot T}_j =A_j T_j +\kappa_{1}(T_cT_j^\dagger  T_j -T_j T_c^\dagger  T_j)+\kappa_{2}(T_j T_j^\dagger  T_c-T_j T_c^\dagger  T_j), \quad t >0, \\
\displaystyle T_j(0) =T_j^0 \in \bbc^{d_1 \times d_2}, \quad j = 1, \cdots, N,
\end{cases}
\end{equation}
where  $\kappa_{1}$ and $\kappa_{2}$ are nonnegative coupling strengths, and we have used Einstein summation convention. 

Next, we define a functional measuring deviations from the centroid of configuration for \eqref{C-7}:
\[ {\mathcal V}[T(t)] :=\frac{1}{N}\sum_{k=1}^N\|T_k(t)-T_c(t)\|_\tF^2 = 1 - \|T_c(t)\|_\tF^2, \quad t \geq 0. \]
Then, we show that $\mathcal V[T]$ converges to  a nonnegative constant $\mathcal V_\infty$.

\begin{theorem}\label{T3.2}
\emph{\cite{H-P3}}
Let $\{ T_j \}$ be a global solution to \eqref{C-7} with $\|T_j^0 \|_\tF = 1$. Then, the following assertions hold. 
\begin{enumerate}
\item
There exists a nonnegative constant ${\mathcal V}_\infty$ such that 
\[ \lim_{t \to \infty } {\mathcal V}[T] = {\mathcal V}_\infty.  \]
\item
The orbital derivative of ${\mathcal V}(T)$ tends to zero asymptotically: 
\begin{align*}
&\lim_{t\rightarrow\infty} \frac{\d}{\dt} {\mathcal V}[T(t)] = 0 \quad \textup{and} \\  &\lim_{t\rightarrow\infty} \Big( \|T_j T_c^\dagger -T_cT_j^\dagger \|_\tF + \|T_j^\dagger T_c-T_c^\dagger T_j \|_\tF \Big) = 0. 
\end{align*}
\end{enumerate}
\end{theorem}

\vspace{0.2cm}

In what follows, we consider {\it ``the reduced Lohe matrix model}'' which corresponds to the generalized Lohe matrix model \eqref{C-7} with $\kappa_2=0$:  \begin{align}\label{C-8}
\begin{cases}
{\dot T}_j =A_j T_j + \kappa_{1}(T_cT_j^\dagger  T_j -T_j T_c^\dagger  T_j), \quad t > 0, \\
T_j(0)=T_j^0, \quad j = 1, \cdots, N,
\end{cases}
\end{align}
subject to the initial conditions:
\begin{equation*} \label{C-9}
T_i^{0,\dagger} T_i^0 = T_j^{0,\dagger} T_j^0, \quad 1 \leq i, j \leq N. 
\end{equation*}
Let $\{T_j\}$ be a solution to \eqref{C-8}  with a specific natural frequency tensor $A_j$: 
\begin{equation*} 
[A_j]_{\alpha\beta\gamma\delta} :=[B_j]_{\alpha\gamma}\delta_{\beta\delta},
\end{equation*}
where $B_j$ is a rank-2 tensor. Then, by singular value decomposition of $T_j(t)$, one has
\[
T_j(t)=U_j(t)\Sigma_j V_j^\dagger .
\] 
Here, $\Sigma_j$ and $V_j$ are time-independent constant matrices, whereas $U_j=U_j(t)$ is  a time-dependent unitary matrix satisfying
\begin{align}\label{C-11}
\begin{cases}
\dot{U}_j= B_j U_j + \kappa_{1}(U_cD-U_j D^\dagger U_c^\dagger U_j), \quad t > 0, \\
U_j(0)=U_j^0, \quad j = 1, \cdots, N,
\end{cases}
\end{align}
where $D$ is a diagonal matrix and $B_jU_j$ is a usual matrix multiplication between $B_j$ and $U_j$.  Moreover, if   complete state aggregation occurs for \eqref{C-8}, then it also occurs for \eqref{C-11}, and vice versa. By algebraic manipulation, one can find a differential inequality for the diameter of $\{U_j\}$:
 \[
 \mathcal D(U) := \max_{1\leq i,j \leq N} \|U_i - U_j\|_\tF,\quad \mathcal D(B) := \max_{1\leq i,j\leq N} \|B_i - B_j\|_\tF.
 \]
For notational simplicity, we denote 
\begin{align*}
\begin{aligned}
& D :=\mathrm{diag}(\lambda_1^2, \cdots, \lambda_{d_1}^2), \quad 
\langle \lambda^2 \rangle :=\frac{1}{d_1}(\lambda_1^2+\lambda_2^2+\cdots+\lambda_{d_1}^2), \\
& \Delta(\lambda^2):=\max_{1\leq k\leq d_1}|\lambda_k^2- \langle \lambda^2 \rangle |, \quad  \mathcal{A} :=\langle \lambda^2 \rangle +\Delta(\lambda^2),\quad \mathcal{B} :=\langle \lambda^2 \rangle -\Delta(\lambda^2).
\end{aligned}
\end{align*}
Now, we are ready to provide the emergent dynamics of \eqref{C-11} as follows.
 \begin{theorem}\label{T3.3}
 \emph{\cite{H-P3}}
The following assertions hold. 
\begin{enumerate}
\item
(Complete state aggregation):~Suppose system parameters and initial data satisfy
\[ {\mathcal D}(B) = 0, \quad  {\mathcal A} > 0, \quad {\mathcal B} > 0, \quad   \kappa_1 > 0, \quad  {\mathcal D}(U^0)\leq \sqrt{\frac{2\mathcal{B}}{\mathcal{A}}}. \]
Then for any solution $\{U_j\}$ to  \eqref{C-11}, we have 
\[
\lim_{t\to\infty} \mathcal D(U) =0.
\]
Moreover, the convergence rate is exponential. 
\vspace{0.1cm}
\item
(Practical aggregation):~Suppose system parameters and initial data satisfy
\[  {\mathcal A} > 0, \quad {\mathcal B} > 0, \quad \kappa_1 > {\mathcal D}(B)\cdot\sqrt{\frac{27\mathcal{A}}{32\mathcal{B}^3}}>0,\quad {\mathcal D}(U^0)<\alpha_2, \]
where $\alpha_2$ is a largest positive root of $g(x)= \mathcal{A}x^3-2\mathcal{B}x+\frac{{\mathcal D}(B)}{\kappa_1} = 0$, and let $\{U_j\}$ be a solution to \eqref{F-4}. Then, one has practical aggregation:
\[
\lim_{\kappa_1\rightarrow\infty}\limsup_{t\rightarrow\infty} {\mathcal D}(U)=0.
\]
\end{enumerate}
\end{theorem}
\begin{proof} The first statement for the complete state aggregation is based on the following differential inequalities:
\[ -2\kappa_1 \mathcal{A} {\mathcal D}(U)+\kappa_1 \mathcal{A} {\mathcal D}(U)^3 \leq\frac{\d}{\dt} {\mathcal D}(U) \leq-2\kappa_1 \mathcal{B} {\mathcal D}(U)+\kappa_1 \mathcal{A} {\mathcal D}(U)^3,
\]
where $\mathcal A$ and $\mathcal B$ are constants determined by the diagonal matrix $D$. This implies the desired upper and lower bound estimates for ${\mathcal D}(U)$. In contrast, the proof of the second statement will be done by similar arguments as in Theorem \ref{T3.1}. For details, we refer the reader to \cite{H-P3}. 
\end{proof}

\subsubsection{The Lohe hermitian sphere model} \label{sec:3.2.2} 
In this part, we consider a reduction of the LT model on the hermitian sphere $\{z\in \bbc^d: |z|=1\}$. In this case, the LT model reduces to 
\begin{equation} \label{C-13}
\dot{z}_j= \Omega_j z_j +\kappa_{0} (z_c\langle z_j, z_j \rangle-z_j \langle z_c. z_j \rangle )+\kappa_1(\langle{z_j, z_c}\rangle- \langle z_c z_j\rangle) z_j,
\end{equation}
where $\langle z, w \rangle$ and $\Omega_j$ are standard inner product in $\bbc^d$ and a skew-hermitian $d \times d$ matrix satisfying
\[ \langle z, w \rangle = [{\bar z}]_{\alpha} [w]_{\alpha}, \quad  \Omega_j^\dagger  = -\Omega_j. \]
Here we used the Einstein summation convention. We refer the reader to \cite{B-H-H-P,B-H-P,H-H-P} for the emergent dynamics of \eqref{C-13}. 

Note that for a real-valued rank-1 tensor $z_j \in \bbr^d$, the coupling terms involving   $\kappa_1$ become identically zero thanks to the symmetry of inner product in ${\mathbb R}^d$. Hence, we can recover the swarm sphere model on $\bbs^{d-1}$:
\begin{equation}\label{C-14}
\dot{x}_j=\Omega_j x_j+\kappa_0 \Big (\langle{x_j, x_j}\rangle x_c-\langle{x_c, x_j}\rangle x_j \Big ).
\end{equation}
The emergent dynamics of \eqref{C-14} has been extensively studied in a series of papers
 \cite{C-C-H,C-H1,C-H2,C-H3,C-H4,J-C,Lo2,Lo3,M-T-G,T-M,Zhu}. 
In what follows, to investigate the nonlinear effect on the collective behaviors of \eqref{C-13} due to two coupling terms, we consider for a while the following two special cases:
 \[ \textup{(i)} ~~\Omega_j = \Omega, \quad  \kappa_1 = 0;  \qquad \textup{(ii)}~~ \Omega_j = \Omega,   \quad \kappa_0 = 0,\quad j=1,\cdots,N. \]
Then, the corresponding models for each case read as follows:
 \begin{align} \label{C-15}
 \begin{aligned}
&\textup{(i)}~~ \dot{z}_j = \Omega z_j + \kappa_0(\langle{z_j, z_j}\rangle z_c-\langle{z_c, z_j}\rangle z_j). \\
&\textup{(ii)}~~\dot{z}_j = \Omega z_j + \kappa_1(\langle{z_j, z_c}\rangle-\langle{z_c, z_j}\rangle)z_j.
\end{aligned}
\end{align}
From the models in \eqref{C-15}, we define a functional  for $\{z_j\}$:
\[
\mathcal D(Z):= \max_{1\leq i,j\leq N} |1-\langle z_i,z_j\rangle|
\]
In the following proposition, we summarize  results on the emergent dynamics of the models in \eqref{C-15} without proofs.

 \begin{proposition} \label{P3.1}
 \emph{\cite{H-P6}}
 The following assertions hold.
 \begin{enumerate}
 \item
Suppose system parameters and initial data satisfy
\[
\kappa_0 > 0, \quad  |z_j^{0}  |=1,\quad \max_{i\neq j}|1-\langle{z_i^{0}, z_j^{0}}\rangle|<\frac12,
\]
and let $ \{ z_j \}$ be a global solution to $\eqref{C-15}_1$ with the initial data $\{z_j^0\}$. Then, there exists a positive constant $\Lambda$ depending on the initial data such that 
\[ {\mathcal D}(Z(t)) \leq {\mathcal D}(Z^{0}) e^{- \kappa_0 \Lambda t}, \quad t \geq 0. \]
\item
Suppose system parameters and initial data satisfy
\[
\kappa_1 > 0, \quad |z_j^{0}  |=1,
\]
and let  $\{ z_j \}$ be a global solution to $\eqref{C-15}_2$ with the  initial data $\{z_j^{0} \}$. Then, there exist a  time-dependent phase function $\theta_j=\theta_j(t)$ such that  
\[ z_j(t)= e^{{\mathrm i} \theta_j(t)} z^{0}_j, \quad j = 1, \cdots, N, \]
and $\theta_j$ is a solution to the Kuramoto-type model with frustrations:
\begin{equation} \label{C-16}
\begin{cases}
\displaystyle \dot{\theta}_j=\frac{2 \kappa_1}{N}\sum_{k=1}^N R_{jk}^{0} \sin(\theta_k-\theta_j+\alpha_{jk}^{0}), \quad t > 0, \\
\displaystyle \theta_j(0)=0,\quad j=1,\cdots,N,
\end{cases}
\end{equation}
where $R_{jk}^{0}$ and $\alpha_{jk}^{0}$ are determined by initial data:
\[  \langle{z_j^{0}, z_k^{0}} \rangle=R^{0}_{jk}e^{\mathrm{i} \alpha^0_{jk}}.  \]
\end{enumerate}
\end{proposition}
\begin{remark}
\textup{(i)}~~In Proposition \ref{P3.1}(2), $R_{jk}^{0}$ and $\alpha^0_{jk}$ satisfy symmetry and anti-symmetry properties, respectively:
\[   R_{jk}^{0} = R_{kj}^{0}, \qquad \alpha^0_{jk} = -\alpha^0_{kj}, \qquad j,k = 1, \cdots, N. \]
\textup{(ii)}~~System \eqref{C-16} can be rewritten as a gradient flow with the following potential $V$:
  \[ \dot{\Theta}=-\nabla_{\Theta} V[\Theta], \quad t > 0, \quad  V[\Theta] := \frac{\kappa_1}{N}\sum_{i, j = 1}^{N} R^{0}_{ij} \Big(1- \cos(\theta_i-\theta_j+\alpha^0_{ji}) \Big). \]
\end{remark}
We now return to the full model \eqref{C-13} with the same free flows. Note that the term  involving  $\kappa_0$  corresponds to the swarm sphere model and the term with $\kappa_1$ describes the complex nature of underlying phase space. For an ensemble $\{z_j \}$, we define the norm of the centroid:
\[ \rho (t):= \left  | \frac{1}{N} \sum_{j=1}^{N} z_j(t) \right  |. \]

\begin{theorem} \label{T3.4}
 \emph{\cite{H-P6}} Suppose system parameters and initial data satisfy
\begin{equation} \label{Z-10}
 \Omega_j = \Omega, \quad j = 1, \cdots, N, \quad 0< \kappa_1 <  \frac{1}{4} \kappa_0, \quad \rho^{0} > \frac{N-2}{N},  
 \end{equation}
and let $\{z_j \}$ be a solution to \eqref{C-13}. Then for each $i, j = 1, \cdots, N$, two-point correlation function $\langle z_i, z_j \rangle$ converges to 1 exponentially fast, i.e., complete state aggregation emerges asymptotically. 
\end{theorem}
\begin{proof} We give a brief sketch for a proof. Details can be found in Theorem 4.1 \cite{H-P6}.  First, note that 
\[   |z_i - z_j  |^2 =  |z_i |^2 +  |z_j |^2 - 2 \mbox{Re} \langle z_i, z_j \rangle = 2 (1 - \mbox{Re} \langle z_i, z_j \rangle)  \leq 2 |1 - \langle z_i, z_j \rangle|. \]
Thus, once we can show that $ |1 - \langle z_i, z_j \rangle|$ tends to zero exponentially fast, then it directly follows that ${\mathcal D}(Z)$ tends to zero exponentially fast. Thus, we introduce a functional
\[ \mathcal{L}(Z) :=\max_{1\leq i, j \leq N}|1-\langle{z_i, z_j}\rangle|^2. \]
By detailed and straightforward calculation, one can derive differential inequality for ${\mathcal L}(Z)$:
\[
\frac{\d}{\dt}\mathcal{L}(Z)\leq-\kappa_0\mathcal{L}(Z)\left(\mathrm{Re}(\langle{z_{i_0}+ z_{j_0}, z_c}\rangle)-\frac{4\kappa_1}{\kappa_0}\right),
\]
where $i_0$ and $j_0$ are extremal indices such that 
\[ \mathcal{L}(Z) =:|1-\langle{z_{i_0}, z_{j_0}}\rangle|^2. \]
On the other hand, one can show that under the assumption \eqref{Z-10} on $\rho^0$, the quantity $\langle z_i, z_c \rangle$ tends to 1 asymptotically. Again, by the assumption on the coupling strengths, there exist positive constants $T$ and $\varepsilon$ such that 
\[
\mathrm{Re}(\langle{z(t)_{i_0}+z(t)_{j_0}, z_c}\rangle)-\frac{4\kappa_1}{\kappa_0}>\varepsilon, \quad \mbox{$t > T$}.
\]
This yields
\[ \frac{\d}{\dt}\mathcal{L}(Z)\leq-\kappa_0 \varepsilon \mathcal{L}(Z), \quad t > T. \]
Hence, one gets the exponential decay of ${\mathcal L}(Z)$. 
\end{proof}

Before we close this subsection, we discuss the swarm double sphere (SDS) model on $\bbs^{d_1-1} \times \bbs^{d_2-1}$ which was recently introduced by  Lohe \cite{Lo1}:
\begin{align}\label{SDS}
\begin{cases}
\dot{u}_i=\Omega_iu_i+\displaystyle\frac{\kappa}{N}\sum_{j=1}^N\langle v_i, v_j\rangle (u_j-\langle u_i, u_j\rangle u_i) ,\quad t>0, \\
\dot{v}_i=\Lambda_i v_i+\displaystyle\frac{\kappa}{N}\sum_{j=1}^N \langle u_i, u_j\rangle(v_j-\langle v_i, v_j\rangle v_i),\\
(u_i, v_i)(0)=(u_i^0, v_i^0) \in\bbs^{d_1-1} \times \bbs^{d_2-1},\quad 1\leq i \leq N, 
\end{cases}
\end{align}
where $\Omega_i \in {\mathbb R}^{d_1 \times d_1}$ and $\Lambda_i \in  {\mathbb R}^{d_2 \times d_2}$ are real skew-symmetric matrices, respectively, and $\kp$ denotes the (uniform) nonnegative coupling strength. For homogeneous zero free flows
\[ \Omega_i = O_d, \quad \Lambda_i = O_d, \quad i = 1, \cdots, N, \]
system \eqref{SDS} can be represented as a gradient flow. More precisely, we set an analytical potential ${\mathcal E}_s$ as 
\begin{equation*} \label{C-16-1} 
\mathcal{E}_s(U, V) := 1-\frac{1}{N^2}\sum_{i, j=1}^N\langle u_i,u_j\rangle \langle v_i, v_j\rangle.
\end{equation*}
Then,  system \eqref{SDS} can recast as a gradient system on the compact state space $(\bbs^{d_1-1} \times \bbs^{d_2-1})^N$:
\begin{equation}
\begin{cases} \label{C-16-2}
\vspace{0.3cm} \displaystyle {\dot u}_i =-\frac{N\kappa}{2} {\mathbb P}_{T_{u_i}\bbs^{d_1-1}} \Big( \nabla_{u_i}\mathcal{E}_s(U, V) \Big), \\
\displaystyle {\dot v}_i =-\frac{N\kappa}{2}  {\mathbb P}_{T_{v_i}\bbs^{d_2-1}} \Big( \nabla_{v_i}\mathcal{E}_s(U, V)\Big),
\end{cases}
\end{equation}
where   projection operators onto the tangent spaces of $\bbs^{d_1-1}$ and $\bbs^{d_2-1}$ at $u_i$ and $v_i$, respectively, are defined by the  following explicit formula: for $w_1 \in \bbr^{d_1}$ and $w_2 \in \bbr^{d_2}$, 
\begin{equation*} \label{C-16-3}
{\mathbb P}_{T_{u_i}\bbs^{d_1-1}} (w_1) := w_1 - \langle w_1, u_i \rangle u_i, \quad {\mathbb P}_{T_{v_i}\bbs^{d_2-1}} (w_2) := w_2 - \langle w_2, v_i \rangle v_i.
\end{equation*}
By the  standard convergence result on a gradient system with analytical potential on a compact space, one can derive the following convergence result for all initial data. 
\begin{proposition}\label{P3.2}
The following assertions hold. 
\begin{enumerate}
\item
Let $\{(u_i, v_i)\}$ be a solution to \eqref{C-16-2}. Then, there exists a constant asymptotic state $(U^{\infty}, V^{\infty}) \in (\bbs^{d_1 - 1})^{N} \times (\bbs^{d_2 - 1})^{N}$ such that 
\[
\lim_{t\to\infty} (U(t), V(t)) = (U^\infty, V^\infty). 
\]
\item
Suppose initial data satisfy 
\begin{equation} \label{C-16-3}
\min_{1\leq i, j\leq N }\langle u_i^0, u_j^0\rangle>0,\quad \min_{1\leq i, j\leq N}\langle v_i^0, v_j^0\rangle>0,
\end{equation}
and let $\{(u_i, v_i)\}$ be a solution to   system \eqref{C-16-2}.  Then, one has   complete state aggregation:
\[
\lim_{t\to\infty} \max_{1\leq i,j\leq N}  |u_i(t)-u_j(t)  |=0,\quad \lim_{t\to\infty}  \max_{1\leq i,j\leq N}   |v_i(t)-v_j(t) |=0.
\]
\end{enumerate}
\end{proposition}
\begin{remark} 
1. We give some comments on the results in Proposition \ref{P3.2}. In the first statement, the asymptotic state $(U^\infty, V^{\infty})$ may depend on initial data. Thus, the result does not tell us whether the complete state aggregation occurs or not as it is.\newline

\noindent 2. The second result says that relative states $u_i - u_j$ and $v_i - v_j$ tend to zero asymptotically, but if we combine   both results (1) and (2), we can show that the initial data satisfying \eqref{C-16-3} lead to complete state aggregation. 
\end{remark}

\vspace{0.5cm}

So far, we have considered sufficient conditions for the emergence of complete state aggregation and practical aggregation. However, this emergent dynamics does not tell us on the solution structure of the LT model. In the following subsection, we consider a special set of solutions, namely tensor product states which can be expressed as tensor products of lower rank tensors.

\subsection{Tensor product states} \label{sec:3.3}
In this subsection, we review tensor product states for the LT model which can be written as a tensor product of rank-1 tensors or rank-2 tensors.  In the following definition, we provide concepts of two special tensor product states.
\begin{definition} \label{D3.2}
\emph{\cite{H-K-P1,H-K-P2}}
\begin{enumerate}
\item  Let $\{T_i\}$ be a ``completely separable state"  if it is a solution to \eqref{LT}, and it is the  tensor product of only rank-1 tensors with unit modulus: for $1\leq i \leq N$ and $1\leq k \leq m$, 
 \begin{equation*} \label{C-17}
 T_i = u^1_i \otimes u_i^2 \otimes \cdots \otimes u_i^m, \quad u^k_i \in \bbc^{d_k}, \quad  |u^k_i | = 1, 
 \end{equation*}
 where $ |\cdot |$ is the standard $\ell^2$-norm in $\bbc^d$. 
 
 \vspace{0.2cm}
 
\item Let $\{T_i \}$ be a ``quadratically separable state"  if it is a solution to \eqref{LT}. and it is the tensor product of only rank-2 tensors (or matrices) with unit Frobenius norm: for $1\leq i \leq N$ and $1\leq k \leq m$,
\begin{equation*}
T_i = U_i^1 \otimes U_i^2 \otimes \cdots \otimes U_i^m, \quad U_i^k \in \bbc^{d_1^k \times d_2^k},\quad \|U_i^k\|_\tF = 1,  
\end{equation*} 
where $\|\cdot\|_\tF$ is the Frobenius norm induced by Frobenius inner product. 
\end{enumerate} 
\end{definition}
\subsubsection{Completely separable state} \label{sec:3.3.1}
To motivate our discussion, we begin with rank-2   tensors that can be decomposed into two rank-1 tensors for all time. Then, since its extension to the case of a rank-$m$ tensor will be straightforward, it suffices to focus on rank-2 tensors, and we refer the reader to \cite{H-K-P1,H-K-P2} for details.  In next proposition, we show that \eqref{C-7} and \eqref{SDS} are equivalent in the following sense. 
 \begin{proposition} \label{P3.3} 
 \emph{\cite{H-K-P2}} The following assertions hold.
 \begin{enumerate}
 \item
 (Construction of a completely separable state):~ Suppose $\{(u_i,v_i) \}$ is a global solution to \eqref{SDS}. Then, a real rank-2 tensor  $T_i$ defined by $T_i  :=u_i \otimes v_i$ is a completely separable state to  \eqref{C-7} with a well-prepared free flow tensor $A_i$ and coupling strengths:
 \begin{equation*} \label{C-18}
 A_i T_i:= \Omega_i T_i + T_i \Lambda_i^\top, \quad \kp_1 = \kp_2 =:\kp.
 \end{equation*}
 \item
(Propagation of complete separability):~Suppose $T_i$ is a solution to \eqref{C-7}  with completely separable initial data:
 \[
 T_i^0 =: u_i^0 \otimes v_i^0, \quad 1 \leq i, j \leq N,
 \]
 for real rank-1 tensors $u_i^0 \in \bbs^{d_1-1}$ and $v_i^0 \in \bbs^{d_2-1}$. Then, there exist two unit vectors $u_i=u_i(t)$ and $v=v_i(t)$ such that 
 \[  T_i(t) = u_i(t) \otimes v_i(t), \quad t>0, \]
 where $(u_i,v_i)$ is a solution to \eqref{SDS} with $(u_i,v_i)(0) = (u_i^0,v_i^0)$.
 \end{enumerate}
  \end{proposition}
Thus, in order the investigate the emergent dynamics of some classes of solutions to \eqref{C-7}, it suffices to study the dynamics of \eqref{SDS}. 
\begin{theorem} \label{T3.5}
\emph{\cite{H-K-P2}}
Let $\{T_i = u_i \otimes v_i \}$ be a completely separable state to \eqref{C-7} with the initial data $\{ u_i^0 \otimes v_i^0 \}$ satisfying 
\[
\min_{1 \leq i,j\leq N } \langle u_i^0,u_j^0\rangle>0 \quad\textup{and}\quad   \min_{1 \leq i,j\leq N } \langle v_i^0,v_j^0\rangle>0. 
\]
Then, we have complete state aggregation:
\[ \lim_{t \to \infty} \| T_i(t) - T_j(t) \|_\tF = 0, \quad 1 \leq i, j \leq N. \]
\end{theorem}
\begin{proof}
Note that 
\[ T_i - T_j = u_i \otimes v_i - u_j \otimes v_j = u_i v_i^{\top} - u_j v_j^{\top} = (u_i - u_j) v_i^{\top}  + u_j (v_i^{\top} - v_j^{\top}). \]
This yields
\begin{align*}
\begin{aligned} 
\| T_i - T_j  \|_\tF &\leq \|(u_i - u_j) v_i^{\top} \|_\tF + \|  u_j (v_i^{\top} - v_j^{\top}) \|_\tF \\
&\leq \| u_i - u_j \|_\tF \cdot \| v_i^{\top} \|_\tF + \| u_j \|_\tF \cdot \| v_i^{\top} - v_j^{\top} \|_\tF \\
&= |u_i - u_j | \cdot  |v_i| + |u_j | \cdot |v_i - v_j | =  |u_i - u_j | +  |v_i - v_j |.
\end{aligned}
\end{align*}
By the result of Proposition \ref{P3.3},  one has complete state aggregation of \eqref{C-7} under suitable conditions on initial data and coupling strengths, and the desired estimates follow.
\end{proof}
\begin{remark} 
Extension to rank-m tensors of the results in Theorem \ref{T3.5} can be found in Section 6 and Section 7 of \cite{H-K-P2}. 
\end{remark}

\subsubsection{Quadratically separable state} \label{sec:3.3.2}
Similar to the previous part, we consider only a finite ensemble of rank-4 tensors that can be decomposed into  a tensor product of two rank-2 tensors, and extension to rank-$2m$ tensors will be straightforward. 

Now, we introduce the {\it swarm double matrix (SDM) model} induced from the LT model whose elements have rank-4 with a specific condition on natural frequencies $B_j$ and $C_j$ in the same spirit of the SDS model \eqref{SDS}: 
\begin{align} \label{C-19}
\begin{cases}
 \dot{U}_j=B_jU_j+\displaystyle\frac{\kappa_{1}}{N}\sum_{k=1}^N\left(
\langle V_j, V_k\rangle_\tF~U_kU_j^\dagger U_j
-\langle V_k, V_j\rangle_\tF~U_jU_k^\dagger U_j\right)\\
 \hspace{2cm}+\displaystyle\frac{\kappa_{2}}{N}\sum_{k=1}^N\left(
\langle V_j, V_k\rangle_\tF~U_jU_j^\dagger U_k 
-\langle V_k, V_j\rangle_\tF~U_jU_k^\dagger U_j\right),\\
 \dot{V}_j=C_jV_j+\displaystyle\frac{\kappa_{1}}{N}\sum_{k=1}^N\left(
\langle U_j, U_k\rangle_\tF~V_kV_j^\dagger V_j
-\langle U_k, U_j\rangle_\tF~V_jV_k^\dagger V_j\right)\\
 \hspace{2cm}+\displaystyle\frac{\kappa_{2}}{N}\sum_{k=1}^N\left(
\langle U_j, U_k\rangle_\tF~V_jV_j^\dagger V_k
-\langle U_k, U_j\rangle_\tF~V_jV_k^\dagger V_j\right),\\
\end{cases}
\end{align}
where $B_j\in\bbc^{d_1\times d_2\times d_1\times d_2}$ and $C_j\in \bbc^{d_3\times d_4\times d_3\times d_4}$  are block skew-hermitian rank-4 tensors, respectively. Similar to the SDS model in Section \ref{sec:3.2.2}, the SDM model \eqref{C-19} with homogeneous zero free flows can be rewritten as a gradient system for the following  analytical potential:
\begin{equation*} \label{C-19-1}
\mathcal{E}_m(U, V) :=1-\frac{1}{N^2}\sum_{i, j=1}^N \langle U_i, U_j\rangle_\tF  \langle V_i, V_j\rangle_\tF.
\end{equation*}
We refer the reader to \cite{H-K-P1} for details. \newline

Next, we consider a special case for the SDM model:
 \[ d_1 = d_2 = n, \quad d_3 = d_4 = m, \quad U_j^0 \in {\mathbf U}(n),  \quad V_j^0\in\mathbf{U}(m), \]
 where $\mathbf{U}(n)$ and $\mathbf{U}(m)$ denote $n\times n$ and $m \times m$ unitary groups, respectively. In this case, one can easily show that unitary properties of $U_j$ and $V_j$ are propagated along \eqref{C-19}:
 \begin{equation}  \label{C-20}
 U_j(t) \in {\mathbf U}(n),  \quad V_j(t) \in\mathbf{U}(m), \quad j = 1,  \cdots, N,\quad t > 0.
  \end{equation}
Note that natural frequency tensors  $B_j$ and $C_j$ are rank-4 tensors satisfying block skew-hermitian properties. In order to give a meaning of Hamiltonian, we associate two hermitian matrices, namely, $H_j \in \bbc^{n\times n}$ and $G_j \in \bbc^{m \times m}$:
\begin{equation} \label{C-21}
[B_j]_{\alpha_1\beta_1\alpha_2\beta_2}=:[-\mathrm{i} H_j]_{\alpha_1\alpha_2}\delta_{\beta_1\beta_2},\quad [C_j]_{\gamma_1\delta_1\gamma_2\delta_2}=:[-\mathrm{i}G_j]_{\gamma_1\gamma_2}\delta_{\delta_1\delta_2}.
\end{equation}
Under the setting \eqref{C-20} and \eqref{C-21}, system \eqref{C-19} reduces to the model on $\Un\times \Um$:
\begin{align}  \label{C-22}
\begin{cases}
\dot{U}_j=-\mathrm{i} H_j U_j+\displaystyle\frac{\kappa}{N}\sum_{k=1}^N\left(
\langle V_j, V_k\rangle_\tF~U_k
-\langle V_k, V_j\rangle_\tF~U_jU_k^\dagger U_j\right),\\
\dot{V}_j=-\mathrm{i}G_j V_j+\displaystyle\frac{\kappa_{1}}{N}\sum_{k=1}^N\left(
\langle U_j, U_k\rangle_\tF~V_k
-\langle U_k, U_j\rangle_\tF~V_jV_k^\dagger V_j\right),
\end{cases}
\end{align}
where $H_j U_j$ and $G_j V_j$ are now usual matrix products. We recall two concepts of definitions for emergent behaviors.
\begin{definition}  \label{D3.3}
\emph{\cite{H-R}}
Let $(\mathcal U,\mathcal V):=\{U_j,V_j\}_{j=1}^N$ be a solution to \eqref{C-22}.  
\begin{enumerate}
\item
System \eqref{C-22} exhibits complete state aggregation if the following estimate holds:
\begin{equation*} \label{Z-0-3} 
\lim_{t\to\infty} \max_{1\leq i,j\leq N} \Big( \|U_i(t) - U_j(t) \|_\tF + \|V_i(t) - V_j(t)\|_\tF \Big) =0.
\end{equation*} 
\item
System \eqref{C-22} exhibits state-locking if the following relations hold:
\begin{equation*}
\exists~ \lim_{t\to\infty} U_i(t) U_j(t)^\dg \quad \textup{and}\quad \exists~ \lim_{t\to\infty} V_i(t) V_j(t)^\dg.
\end{equation*}
\end{enumerate}
\end{definition}

\vspace{0.2cm}

For the emergent dynamics of \eqref{C-22}, we introduce several functionals measuring the degree of aggregation:
\begin{align}
\begin{aligned} \label{C-22-1}
& \mathcal L(t):=  \mathcal D(\mathcal U(t))  +\mathcal  D(\mathcal V(t)) + \mathcal S(\mathcal U(t) ) +\mathcal  S(\mathcal V(t)) , \\
&\mathcal D(\mathcal U(t)) := \max_{1\leq i,j\leq N} \|U_i(t) - U_j(t)\|_\tF,~~  \mathcal  S(\mathcal U(t)):= \max_{1\leq i,j\leq N} |n- \langle U_i,U_j\rangle_\tF (t)|, \\
&\mathcal D(\mathcal V(t)) := \max_{1\leq i,j\leq N} \|V_i(t) - V_j(t)\|_\tF, ~~ \mathcal S(\mathcal V(t)):= \max_{1\leq i,j\leq N} |m- \langle V_i,V_j\rangle_\tF (t)|. 
\end{aligned}
\end{align}
By using the unitarity of $U_i$ and $V_i$, we see
\[
\|U_i - U_j\|_\tF^2 = 2\textup{Re}(n-\langle U_i,U_j\rangle_\tF),\quad \|V_i -V_j\|_\tF^2 = 2\textup{Re}(m-\langle V_i,V_j\rangle_\tF).
\]
Thus, one has
\[
\mathcal D(\mathcal U)^2 \leq 2 \mathcal S(\mathcal U),\quad \mathcal D(\mathcal V)^2 \leq 2 \mathcal S(\mathcal V).
\]
From the relation above, we observe  
\[  \lim_{t\to\infty} \mathcal L(t) = 0 \quad \Longleftrightarrow \quad \mbox{complete state aggregation}, \]
and for given positive integers $n$ and $m$, we set 
\[ \alpha_{n,m} := \frac{-(12n+27) + \sqrt{ (12n+27)^2 + 48(m-4\sqrt n)(3n+4)  } }{4(3n+4)}. \]
Then, the following emergent estimate can be verified by deriving a suitable dissipative differential inequality for ${\mathcal L}$ defined in \eqref{C-22-1}. 
\begin{theorem} \label{T3.6} 
\emph{\cite{H-K-P1}}
Suppose system parameters and initial data satisfy 
\[ H_j = O_n, \quad G_j = O_m, \quad j = 1, \cdots, N, \quad n\geq m >4\sqrt n, \quad  \mathcal L^0 <\alpha_{n,m}, \]
and let $\{(U_j,V_j)\}$ be a global solution to \eqref{C-22}. Then, complete state aggregation emerges asymptotically. 
 \end{theorem}
 \begin{proof} By tedious calculation, one can derive for a.e. $t>0$, 
\begin{equation} \label{C-23}
\dot  {\mathcal L} \leq   -2\kp (m-4\sqrt n)\mathcal L +  \kp (4n+9)\mathcal L^2 +\kp \left( 2n + \frac 83\right)\mathcal L^3 =:\kappa \mathcal L f(\mathcal L), 
\end{equation}
where $f$ is a quadratic polynomial defined by
\[
f(s):= \left( 2n + \frac83\right) s^2 + (4n+9)s - 2(m-4\sqrt n).
\]
By assumption on $n$ and $m$, the coefficient $- 2(m-4\sqrt n) < 0$ and  $f=0$ admit  a unique positive root $\alpha_{n,m}$. Moreover, one can show that the set $\{ {\mathcal L}(t) < \alpha_{n,m} \}$ is positively invariant under the flow \eqref{C-22}. Thus, by the phase line analysis for \eqref{C-23}, one can see that 
\[ \lim_{t \to \infty} {\mathcal L}(t) =0. \]
We refer the reader to \cite{H-K-P1} for details.
 \end{proof}
 \begin{remark} 
 For heterogeneous free flows, we can also find a sufficient framework leading to state-locking (see Definition \ref{D3.3}) in terms of system parameters with a large coupling strength and initial data in \cite{H-K-P1}.
 \end{remark}
Finally, we consider a special ansatz for a solution $T_i$ with rank-4 to \eqref{LT}  as follows:
\begin{equation*}
T_i(t)\equiv U_i(t)\otimes V_i(t),\quad t>0,
\end{equation*}
where $\{(U_j,V_j)\}$ is a solution to   \eqref{C-19}. Parallel to Proposition \ref{P3.3}, we show the propagation of the  tensor product structure along \eqref{C-19}. 
 \begin{proposition} \label{P3.4} 
 \emph{\cite{H-K-P1}}
 The following assertions hold. 
 \begin{enumerate}
 \item
(Construction of a quadratically separable state):~Suppose $\{(U_i,V_i) \}$ is a solution to \eqref{C-19}. Then, a rank-4 tensor  $T_i$ defined by $T_i  :=U_i \otimes V_i$ is a quadratically separable state to \eqref{LT} with a well-prepared free flow tensor $A_i$ satisfying 
 \begin{align} 
 \begin{aligned} \label{C-26}
[A_j]_{\alpha_1\beta_1\gamma_1\delta_1\alpha_2\beta_2\gamma_2\delta_2} &=[B_j]_{\alpha_1\beta_1\alpha_2\beta_2}\delta_{\gamma_1\gamma_2}\delta_{\delta_1\delta_2} \\
&\hspace{0.5cm}  + [C_j]_{\gamma_1\delta_1\gamma_2\delta_2}\delta_{\alpha_1\alpha_2}\delta_{\beta_1\beta_2}.
\end{aligned}
\end{align}

 \vspace{0.2cm}
 
 \item
 (Propagation of quadratic separability):~Suppose a rank-4 tensor $T_i$ is a solution to \eqref{LT} with \eqref{C-26} and quadratically separable initial data:
 \begin{equation*} \label{C-3}
 T_i^0 =: U_i^0 \otimes V_i^0, \quad 1 \leq i \leq N,
 \end{equation*}
 for  rank-2 tensors $U_i^0 \in \bbc^{d_1\times d_2}$ and $V_i^0 \in \bbc^{d_3 \times d_4}$ with unit Frobenius norms. Then, there exist two matrices with unit Frobenius norms $U_i=U_i(t)$ and $V=V_i(t)$ such that 
 \[  T_i(t) = U_i(t) \otimes V_i(t), \quad t>0, \]
 where $(U_i,V_i)$ is a solution to \eqref{C-19} with $(U_i,V_i)(0) = (U_i^0,V_i^0)$.
 \end{enumerate}
  \end{proposition}
 As we have seen in Theorem \ref{T3.5}, complete state aggregation of a quadratically separable state to the  LT model will be completely determined by  Proposition \ref{P3.4}.\newline 

\section{The Schr\"odinger-Lohe model} \label{sec:4} 
\setcounter{equation}{0}
In this section, we review emergent dynamics, standing wave solutions, and numerical simulations for the SL model on a network.

\subsection{Emergent dynamics} \label{sec:4.1} In this subsection, we study sufficient frameworks leading to the emergent dynamics of the SL model over a network in terms of system parameters and initial data, as we have seen in previous section. First, we recall definitions of aggregation for the SL model. 
\begin{definition} \cite{A-M,C-H3,H-K1,H-H,H-H-K}
Let $\Psi = \{\psi_j\}$ be a global solution to \eqref{S-L}. 
\begin{enumerate}
\item $\Psi$ exhibits complete state aggregation if the following estimate holds:
\[
\lim_{t\to\infty} \max_{1\leq i,j\leq N } \| \psi_i (t) - \psi_j(t) \|  =0.
\]
\item $\Psi$ exhibits practical aggregation if the following estimate holds:
\[
\lim_{\kp\to\infty} \limsup_{t\to\infty} \max_{1\leq i,j\leq N } \| \psi_i (t) - \psi_j(t) \|  =0.
\]
\item $\Psi$ exhibits state-locking if the following relation holds:
\[
\exists \lim_{t\to\infty} \langle \psi_i(t),\psi_j(t) \rangle.
\]
\end{enumerate}
\end{definition}
As we have seen from Section 3, the emergent dynamics of the SL model will be different whether the corresponding linear free flows are homogeneous or heterogeneous. Precisely, we define diameter for external potentials:
\[  {\mathcal D}(V) := \max_{1\leq i,j \leq N} \| V_i - V_j \|_{L^{\infty}}.\]
Then, we have the following two cases:
\[ \mathcal D(V)=0:\mbox{homogeneous ensemble}, \quad  {\mathcal D}(V) > 0 : \mbox{heterogeneous ensemble}.
\]

\subsubsection{All-to-all network} \label{sec:4.1.1}
In this part, we list up sufficient frameworks for the emergent dynamics of \eqref{S-L} without detailed proofs. First, we consider identical one-body potentials over all-to-all network:
\[
V_i = V_j,\quad i,j=1,\cdots,N,\quad \textup{i.e.,}\quad \mathcal D(V) =0 \quad \textup{and} \quad a_{ik} \equiv 1.
\]
For a given ensemble $ \Psi = \{ \psi_j \}$, we set several Lyapunov functionals measuring the degree of aggregation: 
\begin{equation*} \label{D-1}
{\mathcal D}(\Psi) := \max_{1 \leq i,j \leq N} \| \psi_i - \psi_j \|, \quad h_{ij} := \langle \psi_i, \psi_j \rangle,   \quad  \rho :=  \left \| \frac{1}{N} \sum_{k=1}^N \psi_k \right \|,
\end{equation*}
where $\langle \cdot, \cdot \rangle$ and $\| \cdot \|$ are $L^2$-inner product and its associated $L^2$-norm, respectively. 
\begin{theorem} \label{T4.1}
\emph{\cite{A-M,C-C-H,C-H3,H-H,H-H-K2}}
(Homogeneous ensemble)~Suppose system parameters and initial data satisfy
\[   \kappa >0, \quad a_{jk} \equiv 1, \quad  \mathcal D(V) = 0, \quad \| \psi_j^0 \| = 1, \quad 1 \leq j,k \leq N, \]
and let $\Psi= \{ \psi_j \}$ be a global solution to \eqref{S-L} with the initial data $\Psi^0$. Then, the following assertions hold:
\begin{enumerate}
\item
(Emergence of dichotomy): one of the following holds: either complete state aggregation:
\begin{equation*}
\lim_{t\to\infty} \langle \psi_i,\psi_j\rangle =1 \quad \textup{for all $i,j=1,\cdots,N$},
\end{equation*}
or bi-polar aggregation:~ there exists a single index $\ell_0 \in \{1,\cdots,N\}$ such that
\begin{equation*} 
\lim_{t\to\infty} \langle \psi_i ,\psi_j \rangle =  1 ~~ \textup{for $i,j\neq \ell_0$} ~~ \textup{and} ~~ \lim_{t\to\infty} \langle \psi_{\ell_0},\psi_i \rangle = -1 ~~ \textup{for $i\neq \ell_0$.}
\end{equation*}
\item
If initial data satisfy ${\mathcal D}(\Psi^0) <\frac{1}{2}$, then complete state aggregation occurs exponentially fast:
\[ {\mathcal D}(\Psi(t)) \lesssim e^{-\kappa t}, \quad \textup{as $t \geq 0$}. \]
\item
If initial data satisfy $\textup{Re} \langle \psi_i^0,\psi_j^0\rangle >0$ for all $i,j$, then complete state aggregation occurs in $H^1$-framework as well:
\begin{equation*}
\lim_{t\to\infty} \|\psi_i-\psi_j\|_{H^1(\bbr^d)} =0, \quad \textup{for all $i,j=1,\cdots,N$.}
\end{equation*}
\end{enumerate}
\end{theorem}
\begin{proof} Below, we provide some ingredients without technical details. \newline

\noindent (i)~By direct calculation,  we can see that $\rho$ satisfies  
\begin{align} \label{D-1-1}
\begin{aligned}
& \frac{d\rho}{dt}= \kp\left( \rho^2 - \frac1N \sum_{k=1}^N \textup{Re}( \langle \zeta,\psi_k\rangle^2  )  \right), \\
& 1-\rho(t)^2 = \frac{1}{2N^2}\sum_{j,k=1}^N \|\psi_j(t)-\psi_k(t)\|^2.
 \end{aligned}
 \end{align}
Note that complete state aggregation emerges if and only if $\rho(t)$ converges to 1 as $t \to \infty$. By $\eqref{D-1-1}_1$ and boundedness of $\rho$,  the order parameter $\rho$ is non-decreasing in time, and hence it converges to a definite value $\rho_\infty$. After careful analysis of $\rho_\infty$, one can see that the disired dichotomy holds (see \cite{H-H-K}). 

\vspace{0.2cm}

\noindent (ii)~By direct calculation, one can derive  Gronwall's inequality for ${\mathcal D}(\Psi)$ (see \cite{C-H3}):  
\[ {\dot {\mathcal D}}(\Psi) \leq  \kappa \Big(-{\mathcal D}(\Psi) + 2 {\mathcal D}(\Psi)^2 \Big), \quad t > 0.
\]
This yields
\[ {\mathcal D}(\Psi(t)) \leq \frac{{\mathcal D}(\Psi^0)}{{\mathcal D}(\Psi^0) + (1-2 {\mathcal D}(\Psi^0)) e^{\kappa t}}. \]
Note that the condition ${\mathcal D}(\Psi^0) <\frac{1}{2}$ is needed to exclude the finite-time blowup of ${\mathcal D}(\Psi(t))$.

\vspace{0.2cm}

\noindent (iii)~In \cite{A-M,H-H}, the authors provided ``{\it finite dimensional approach''} based on the two-point correlation function $h_{ij}=\langle \psi_i,\psi_j \rangle$ measuring the degree of aggregation. Then, the correlation function $h_{ij}$ with $a_{ik} \equiv 1$  satisfies
\begin{equation*}
\frac{\d}{\dt} h_{ij} =  \frac{\kp}{2N} \sum_{k=1}^N (h_{ik} + h_{kj})(1-h_{ij}), \quad 1 \leq i, j \leq N.
\end{equation*}
By applying dynamical systems theory, we can obtain the desired result. 
\end{proof}

\vspace{0.5cm}

Next, we consider a  heterogenous ensemble with distinct one-body potentials. In the aforementioned work \cite{H-H}, the authors considered the case of a two-oscillator system whose external potentials are assumed to be the vertical translation of a common potential:
\[
V_j(x) =V(x) + \nu_j,\quad x\in \bbr^d,\quad \nu_j \in \bbr,\quad j=1,2.
\]
Then for the system, they exactly find the critical coupling strength for the bifurcation where the system undergoes a transition from the emergence of periodic motion to the existence of equilibrium. More precisely,  
\begin{align*}
&\textup{(i)}~\kp < \nu := |\nu_1 - \nu_2| : \textup{$h(t)$ is a periodic function with the period $\frac{\pi}{\sqrt{\nu^2-\kp^2}}$.} \\
&\textup{(ii)}~\kp = \nu:  \lim_{t\to\infty} h(t) = -\mi. \\
&\textup{(iii)}~\kp>\nu: \lim_{t\to\infty} h(t) = -\mi\frac{\nu}{\kp} + \sqrt{1-\frac{\nu^2}{\kp^2}.} 
\end{align*}
The emergent dynamics of the SL model with $N\geq3$ can be studied similarly as in previous section for some restricted class of initial data and system parameters. 
\begin{proposition}
\emph{\cite{Cho-C-H,H-H-K21}}
(Heterogeneous ensemble). The following assertions hold. 
\begin{enumerate}
\item
Suppose system parameters and initial data satisfy
\[
 \kappa > 54 {\mathcal D}(V) > 0, \quad {\mathcal D}(\Psi^0) < \alpha_2,  
\]
where $\alpha_2$ is a larger positive root of $f(x) := 2x^3-x^2+ \frac{2{\mathcal D}(V)}{\kappa}=0$, and  let $\{\psi_j\}$ be a solution to \eqref{S-L}.  Then, practical aggregation emerges:
\[ \lim_{\kappa \to \infty} \limsup_{t \to \infty} {\mathcal D}(\Psi(t)) = 0. \]

\vspace{0.1cm}

\item
Suppose system parameters and initial data satisfy  
\begin{equation} \label{Z-1}
\kp> 4  {\mathcal D}(\nu) =: 4 \max_{1 \leq i, j \leq N} |\nu_i - \nu_j |, \quad \mathcal D(\Psi^0)^2 < \frac{ \kp + \sqrt{ \kp^2- 4\kp  {\mathcal D}(\nu)}}{\kp},
\end{equation}
and let $\{\psi_j\}$ be a global solution to \eqref{S-L}. Then for each $i,j$, there exist a complex number $\alpha_{ij}^\infty \in \bbc$ such that
\[
\lim_{t\to\infty} \langle \psi_i,\psi_j\rangle(t) = \alpha_{ij}^\infty,  \quad |\alpha^{\infty}_{ij}|\leq1.
\]
In other words, state-locking occurs. Moreover, the convergence rate is exponential. 
\end{enumerate}
\end{proposition}
\begin{proof}

\noindent (i)~ Consider the equation:
\[   f(x) := 2x^3-x^2+ \frac{2{\mathcal D}(V)}{\kappa}=0, \quad x \in [0, \infty), \qquad  \kappa > 54 {\mathcal D}(V). \]
Then, the cubic equation $f = 0$ has a positive local maximum $ \frac{2{\mathcal D}(V)}{K}$ and a negative local minimum
$\frac{2D({\mathcal V})}{K}   -\frac{1}{27}$ at $x = 0$ and $\frac{1}{3}$, respectively. Moreover, it has two positive real roots  $\alpha_1 <  \alpha_2 $:
\[  0 < \alpha_1 < \frac{1}{3} < \alpha_2 < \frac{1}{2}. \]
Clearly, the roots depend continuously on $\kappa$ and ${\mathcal D}(V)$, and
\[ \label{Fact-1}
 \lim_{\kappa \to \infty} \alpha_1 = 0,  \quad  \lim_{\kappa \to \infty} \alpha_2 = \frac{1}{2}.
\]
The flow issued from the  initial data satisfying ${\mathcal D}(\Psi^0) < \alpha_2$ tends to the set $\{ \Psi~:~{\mathcal D}(\Psi) < \alpha_1 \}$ in finite-time. Moreover, this set is positively invariant, i.e.,
\[ {\mathcal D}(\Psi(t)) < \alpha_1, \quad t \gg1. \]
On the other hand one has  $ \lim_{\kappa \to \infty} \alpha_1(\kappa) = 0$. This yields the desired estimate. Detailed argument can be found in \cite{Cho-C-H}. 

\vspace{0.1cm}

\noindent (ii)~In order to show that the limit of $h_{ij} =\langle \psi_i,\psi_j\rangle$ exists, we consider any two solutions to \eqref{S-L} denoted by  $\{\psi_j\}$ and $\{\tilde{\psi}_j\}$, and we write
\[
h_{ij} = \langle \psi_i,\psi_j\rangle, \quad \tilde h_{ij} := \langle \tilde\psi_i,\tilde\psi_j\rangle,
\]
and define the diameter measuring the dissimilarity of two correlation functions:
\[
d(\mathcal H, \tilde {\mathcal H})(t) := \max_{1\leq i,j\leq N } |h_{ij}(t) - \tilde h_{ij}(t)|.
\]
Then, we find a differential inequality for $d(\mathcal H,\tilde{\mathcal H})$:
\[
\frac\d\dt d(\mathcal H,\tilde {\mathcal H}) \leq -\kp (1-   \mathcal D(\Psi)^2) d(\mathcal H,\tilde {\mathcal H}) ,\quad t>0.
\]
Under the assumption \eqref{Z-1}, we show that $d(\mathcal H,\tilde{\mathcal H})$ tends to zero. Once we establish the zero convergence of $d(\mathcal H,\tilde{\mathcal H})$, since our system is autonomous, we can choose $\tilde h_{ij}$ as $\tilde h_{ij}(t)  = h_{ij}(t+T)$ for any $T>0$. By discretizing the time $t\in \bbr_+$ as $n\in \bbz_+$, one can deduce that $\{h_{ij}(n)\}_{n\in \bbz_+}$ becomes a Cauchy sequence in the complete space $\{z\in \bbc:|z|\leq1\}$. Thus, we find the desired complex number $\alpha_{ij}^\infty$.

\end{proof}
\subsubsection{Network structure} \label{sec:4.1.2}
For the interplay between emergent behaviors and network structures, the authors in \cite{H-H-K} considered the following three types of network structures: cooperative, competitive and cooperative-competitive networks depending on their signs:
\begin{align*}
&\textup{(i)}~\textup{Cooperative: $a_{ik}>0$ for all $i,k=1,\cdots,N$.} \\
&\textup{(ii)}~\textup{Competitive: $a_{ik}<0$ for all $i,k=1,\cdots,N$.} \\
&\textup{(iii)}~\textup{Cooperative-competitive: $a_{ik}=(-1)^{i+k}$ for $i<k$ and $a_{ki} =-a_{ik}$. }
\end{align*}

 For the cooperative network $a_{ik}$, all values of $a_{ik}$ have positive values so that we can expect complete state aggregation. Here, we associate statistical quantities for the cooperative network  $\{a_{ik}\}$:
 \begin{align*}
 d(\mathcal A) := \max_{1\leq i,j,k\leq N }  |a_{ik} - a_{jk}|, ~~ a_m^c := \min_{1\leq i \leq N } \frac1N \sum_{k=1}^N a_{ik},~~ a_M := \max_{1\leq i,j\leq N } a_{ij}
 \end{align*}
On the other hand, for the competitive network, all values of $a_{ik}$ are assumed to be negative and they considered the simplest case $a_{ik}\equiv -1$. In this case, each oscillator would exhibit repulsive behaviors. Lastly for the cooperative-competitive network, some of $a_{ik}$ are positive and some are negative. Hence, we expect interesting dynamical patterns other than aggregation, such as periodic orbit or bi-polar aggregation. The arguments above are summarized in the following theorem. 
  \begin{theorem}
  \emph{\cite{H-H-K}}
 Let $\{\psi_j\}$ be a global solution to \eqref{S-L}. 
 \begin{enumerate}
 \item Suppose that $a_{ik}$ and initial data satisfy 
 \begin{equation*}
a_{ik}>0,\quad  d(\mathcal A) < a_m^c,\quad \mathcal D(\Psi^0)^2 <\frac{2(a_m^c - d(\mathcal A))}{a_M}.
 \end{equation*}
 Then, system \eqref{S-L} exhibits complete state aggregation. 
 
 \vspace{0.2cm}
 
 \item 
Suppose that $a_{ik}$ and initial data satisfy 
 \begin{equation*}
a_{ik} \equiv -1,\quad \rho^0 >0.
\end{equation*}
Then, a solution to  \eqref{S-L} tends to the splay state. 

\vspace{0.2cm}

\item Suppose for $N=4$ that $a_{ik}$ and initial data satisfy 
\begin{equation*}
 a_{ik} = (-1)^{i+k},\quad 2 + h_{12}^0 + h_{14}^0 + h_{23}^0 + h_{34}^0 < h_{13}^0 + h_{24}^0.
\end{equation*} 
Then, we have
\begin{equation*}
\lim_{t\to\infty} h_{ij}(t)  = (-1)^{i+j}.
\end{equation*}
 \end{enumerate}
 \end{theorem}

\subsection{Standing wave solution} \label{sec:4.2} 
In this subsection, we consider a specific type of a solution whose shape is invariant under the flow, namely, a standing wave solution. We begin with the ansatz for $\psi_i$: 
\begin{equation} \label{E-1-1}
\psi_j(x,t) =  u(x) e^{-\mi E t} \quad \text{with} \quad \|u\|=1, \quad j = 1, \cdots, N,
\end{equation}
where $E$ is a real number. 
We substitute the ansatz \eqref{E-1-1} into \eqref{S-L} with an identical harmonic potential $V_i(x) = |x|^2$ to  derive
\begin{equation} \label{E-1-2}
-\Delta u + |x|^2 u = Eu, \quad x\in \bbr^d.
\end{equation}
In what follows, we consider the one-dimensional $d=1$ case and   generalization to the multi-dimensional case can be constructed from the tensor product of a one-dimensional solution. For the one-dimensional case, equation \eqref{E-1-2} becomes 
\begin{equation} \label{E-2}
-u_{xx} + x^2u = Eu, \quad x \in \bbr.
\end{equation}
Then, it is well known that the equation \eqref{E-2} has eigenvalues $E_k$ and orthonormal eigenfunctions (or constant multiple of the Hermite functions) $u_k$: 
\begin{equation} \label{E-2-1}
E_k=2k+1, \qquad u_k(x)= \left(  \frac{1}{\sqrt{\pi} 2^k k!} \right)^{1/2} e^{-x^2/2} H_k(x), \quad k = 0,1, \cdots,
\end{equation}
where $H_k$ is the $k$-th Hermite polynomial defined by 
\begin{equation*}
H_k(x)=(-1)^k e^{x^2} \frac{d^k}{dx^k}e^{-x^2}, \quad k = 0, 1, \cdots.
\end{equation*}
We now consider the following Cauchy problem: 
\begin{equation} \label{E-3}
\begin{cases}
\displaystyle \mi u_t = -u_{xx} + x^2 u, \quad x \in \bbr,~~t > 0, \\
\displaystyle u(x,0)=u_0(x) = \sum_{k=0}^\infty a_ku_k(x),
\end{cases}
\end{equation}
where $\{u_k\}$ is defined in \eqref{E-2-1} and $\{ a_n\}$ is a $\ell^2$-sequence of complex numbers. Then, it can be easily seen that the solution $u(x,t)$ of the Cauchy problem \eqref{E-3} is given as follows:
\begin{equation*} 
u(x,t) = \sum_{k=0}^\infty a_ku_k(x) e^{-\mi (2k+1)t} .
\end{equation*}
Next, we study the stability issue for the two types of the standing wave solutions whose existence is guaranteed by the previous argument: 
\begin{align*}
 \textup{(I)}~~ &\psi_j(x,t) =u_k(x) e^{-\mathrm{i} (2k+1)t}  \,\,\textup{ for } j=1, \,2, \cdots, N.  \\
 \textup{(II)}~~&\psi_1(x,t)=- u_k(x) e^{-\mathrm{i} (2k+1)t} \quad \textup{and}\quad \psi_j =  u_k(x) e^{-\mathrm{i} (2k+1)t} ~~ \textup{ for } j\neq1,
\end{align*}
where $u_k$ is given by the formula \eqref{E-2-1}. \newline

Note that the  family  (I) corresponds to the completely aggregated state, whereas the  family  (II) corresponds to   bi-polar state. As in Theorem \ref{T4.1},    bi-polar state is unstable. Below, under the initial condition for which complete state aggregation occurs, the SL model becomes stable.
\begin{theorem} \label{T4.3}
\emph{\cite{H-H-K2}}
The family  $(I)$ is stable in the following sense: for all $\varepsilon>0$, there exists $\delta>0$ such that 
\[
\|\psi^0_j - u_k\|<\delta \quad \Longrightarrow \quad \lim_{t\to\infty} \|\psi_j(t) - u_k e^{-\mi(2k+1)t} \| < \varepsilon.
\]
\end{theorem}

\begin{remark}
Note that we cannot expect an asymptotic stability of  $u_k(x) e^{-\mi (2k+1)t}$. In fact, we consider the following form of solution:
 \begin{equation*}
\psi_j(x,t) = (1-a) u_k(x) e^{-\mi (2k+1)t}+ b u_{k+1}(x) e^{-\mi (2k+3)t} \quad \textup{for } j=1, \,2, \cdots, N,
\end{equation*}
where $ |1-a|^2 + |b|^2=1$. Then,  it is easy to check that the above $\psi_j$ is a solution to \eqref{S-L} and
\begin{align*}
\| \psi_j(x,t)-u_k(x) e^{-\mi (2k+1)t}\|^2= |a|^2+ |b|^2 = 2 \textup{Re} (a).
\end{align*}
\end{remark}
 \subsection{Numeric scheme}
In \cite{B-H-K-T}, the authors consider nonlinearly coupled Schr\"odinger-Lohe type system by employing cubic nonlinearity so that the model would be reduced to the Gross-Pitaevskii equation when the coupling is turned off:
\begin{equation}  \label{G-P-L}
\begin{cases}
\displaystyle \mi \p_t \psi_j = - \frac{1}{2} \Delta \psi_j + V_j\psi_j + \sum_{k=1}^N \beta_{jk} |\psi_k|^2 \psi_j  \\
\displaystyle \hspace{3.7cm}+ \frac{\mi\kp}{2N}\sum_{k=1}^N a_{jk} \left( \psi_k - \frac{ \langle \psi_j,\psi_k\rangle}{\langle \psi_j,\psi_j\rangle} \psi_j \right),\\
\displaystyle \psi_j(x,0) = \psi_j^0(x),\quad (x,t)\in \bbr^d\times \bbr_+, \quad j=1,\cdots,N.
\end{cases}
\end{equation}
Next, we discuss an efficient and  accurate numerical method for discretizing   \eqref{G-P-L}. Several numerical examples will be carried out and compared with corresponding  analytical results shown in previous section.      Due to the external trapping potential $V_j(x)$ ($j=1,\cdots,N$), the wave functions $\psi_j$   ($j=1,\cdots,N$) decay  exponentially fast as $|x|\rightarrow\infty$. Therefore, it suffices to truncate the  problem  \eqref{G-P-L}  into a sufficiently large bounded  domain $\mathcal{D}\subset\mathbb{R}^d$  with periodic boundary condition (BC). The  bounded  domain $\mathcal{D}$ is chosen as a box $[a,b]\times[c,d]\times[e,f]$ in 3D, a rectangle $[a,b]\times[c,d]$ in 2D, and an interval $[a,b]$ in 1D.

\subsubsection{A time splitting Crank-Nicolson spectral method}
First, we begin with the description of   \eqref{G-P-L} combining a time splitting spectral method and the Crank-Nicolson method. Choose $\Delta t>0$ as the time step size and denote time steps $t_n:=n\Delta t$ for $n\ge0$. From time $t=t_n$ to $t=t_{n+1}$,  the GPL is solved in three splitting steps.  One  solves first
\be
\label{eq:TS-L}
\mi \p_t \psi_j = - \frac{1}{2} \Delta \psi_j , \quad x\in\mathcal{D},  \quad j=1,\cdots, N,
\ee
with periodic BC on  the boundary $\partial\mathcal{D}$ for  the time step of length $\Delta t$, then solves
\be
\label{eq:TS-NonL01}
\mi \p_t \psi_j =V_j\psi_j + \sum_{k=1}^N \beta_{jk} |\psi_k|^2 \psi_j,  \quad j=1,\cdots, N,
\ee
 for the same time step, and finally solves
 \be
\label{eq:TS-NonL}
\mi \p_t \psi_j = \frac{\mi\kp}{2N}\sum_{k=1}^N a_{jk} \left( \psi_k - \frac{ \langle \psi_j,\psi_k\rangle}{\langle \psi_j,\psi_j\rangle} \psi_j \right),\quad j=1,\cdots, N,
\ee
 for the same time-step. The linear subproblem \eqref{eq:TS-L} is  discretized in space by the Fourier pseudospectral method and integrated in time analytically in the phase space \cite{B,B-C3,B-C4,B-J-M}. For the nonlinear subproblem \eqref{eq:TS-NonL01}, it conserves $|\psi_k|^2$ pointwise in time, i.e. $|\psi_k(x,t)|^2\equiv |\psi_k(x,t_n)|^2$ for $t_n\le t\le t_{n+1}$ and $k=1,\ldots,N$ \cite{B,B-C3,B-C4,B-J-M}. 
 Thus it collapses
 to a linear subproblem and can be integrated in time analytically \cite{B,B-C3,B-C4,B-J-M}.
 For the nonlinear subproblem \eqref{eq:TS-NonL}, due to the presence of the Lohe term involving $\kappa$, it cannot be integrated analytically  (or explicitly) in the way for the standard GPE \cite{B,B-C3}. Therefore,  we  will apply a Crank-Nicolson  scheme  \cite{B-C40} to further discretize the temporal derivate  of  \eqref{eq:TS-NonL}.

To simplify the presentation, we will only provide the scheme for 1D. Generalization to $d>1$ is straightforward for tensor grids. To this end, we choose  the spatial mesh size  as  $\Delta x=\frac{b-a}{M}$ with $M$ a even positive integer, and let the grid points be
\begin{equation*}
x_\ell=a+ \ell \Delta x, \qquad \ell=0,\cdots,M.
\end{equation*}
For $1\le j\le N$  denote $\psi_{j,\ell}^n$ as the approximation of $\psi_j(x_\ell, t_n)$ ($0\le \ell\le M$) and $\bm{\psi}_j^n$ as the solution vector with component $\psi_{j,\ell}^n$.   Combining the time splitting \eqref{eq:TS-L}--\eqref{eq:TS-NonL} via the Strang splitting  and the  Crank-Nicolson  scheme for \eqref{eq:TS-NonL}, a second order  \textit{Time Splitting Crank-Nicolson Fourier Pseudospectral} (TSCN-FP) method to solve GPL on $\mathcal{D}$ reads as:  
\bea
\nonumber
 \psi^{(1)}_{j,\ell}
 &=&\sum_{p=-M/2}^{M/2-1}e^{-i\Delta t\,\mu_p^2/4}\,\widehat{(\bm{\psi}_j^n)}_p\,e^{i\mu_p(x_\ell-a)},\\[0.5em]
\label{eq:TSCNFP12}
\psi^{(2)}_{j,\ell}
\nonumber
 &=&e^{-i\Delta t\left(V_{j}(x_\ell)+\sum_{k=1}^N \beta_{jk}|\psi^{(1)}_{k,\ell}|^2 \right)/2}\;\psi^{(1)}_{j,\ell},\\[0.5em]
\label{eq:TSCNFP2}
i\frac{\psi^{(3)}_{j,\ell}-\psi^{(2)}_{j,\ell}}{\Delta t}
&=&\textcolor{black}{\frac{{\mathrm i}\kappa}{2N}}\sum_{k=1}^{N}a_{jk}\bigg[\psi_{k,\ell}^{(\fl{5}{2})} -\frac{\big\lela \bm{\psi}_{j}^{(\fl{5}{2})},\bm{\psi}_{k}^{(\fl{5}{2})}\big\rira_{\Delta x} }{\big\lela \bm{\psi}_j^{(\fl{5}{2})}, \bm{\psi}_j^{(\fl{5}{2})}\big\rira_{\Delta x}}\;\psi_{j,\ell}^{(\fl{5}{2})} \bigg], \\[0.5em]
\nonumber
\psi^{(4)}_{j,\ell}
 &=&e^{-i\Delta t\left(V_{j}(x_\ell)+\sum_{k=1}^N \beta_{jk}|\psi^{(3)}_{k,\ell}|^2\right)/2}\;\psi^{(3)}_{j,\ell},\quad
0\le \ell\le M,\\[0.5em]
\nonumber
\psi^{n+1}_{j,\ell}&=&\sum_{p=-M/2}^{M/2-1}e^{-i\Delta t\,\mu_p^2/4}\,\widehat{(\bm{\psi}_j^{(4)})}_p\,e^{i\mu_p(x_\ell-a)},\quad j=1,\cdots,N.
\eea
Here, $\mu_p=\frac{p\pi}{b-a}$, $\widehat{(\bm{\psi}_j^{n})}_p$ and $\widehat{(\bm{\psi}_j^{(4)})}_p$ ($p=-\frac{M}{2},\cdots,\frac{M}{2}$) are the discrete Fourier transform coefficients of the vectors $\bm{\psi}_j^{n}$ and $\bm{\psi}_j^{(4)}$ ($j=1,\cdots, N$), respectively. Moreover,
\begin{equation*}
\psi^{(\fl{5}{2})}_{j,\ell}=:\frac{1}{2}\Big( \psi^{(3)}_{j,\ell}+\psi^{(2)}_{j,\ell}\Big),\qquad  
\big\lela \bm{\psi}_{j}^{(\fl{5}{2})},\bm{\psi}_{k}^{(\fl{5}{2})}\big\rira_{\Delta x}
=:\Delta x \sum_{\ell=0}^{M-1}\psi^{(\fl{5}{2})}_{j,\ell}\,
\bar{\psi}^{(\fl{5}{2})}_{k,\ell}.
\end{equation*}
\textcolor{black}{Although the Crank-Nicolson step \eqref{eq:TSCNFP2} is fully implicit, it can be either  solved efficiently by Krylov subspace iteration method with proper preconditioner \cite{AD1} or the fixed-point iteration  method with a stabilization parameter \cite{B-C-L}. In addition, TSCN-FP is of  spectral accuracy in space and second-order accuracy in time. By following the standard procedure, it is straightforward to show that the TSCN-FP conserve mass of each component  in discrete level, i.e.,
$\|\bm{\psi}_j^{n}\|_{l^2}^2:=\big\lela \bm{\psi}_{j}^{n},\bm{\psi}_{j}^{n}\big\rira_{\Delta x}\equiv \|\bm{\psi}_j^{0}\|_{l^2}^2$
for $n\ge0$ and $j=1,2,\ldots,N$. We omit the details here for brevity.}

 \subsubsection{Numerical Results}

 In this subsection, we apply the   TSCN-FP schemes proposed in the previous subsection to simulate some interesting dynamics.
For our simulation,   we choose
\begin{equation*}
\beta=1,\quad \Delta t=\textcolor{black}{2\times10^{-4}},\quad \mathcal{D}=[-12, 12]^d, \quad d=1,2.
\end{equation*}
 The potentials and initial data are chosen  as follows:
 \begin{equation*}
 V_j(x)=\pi^2\alpha_j^2\ |x|^2,\quad
 \psi_j^0=\sqrt{a_j/\pi}\ e^{-a_j |x-x^{j}_0|^2}.
 \end{equation*}
 Here, $\alpha_j$ and $x^{j}_0$ are real constants to be given later.  In fact, complete state aggregation and practical aggregation estimates do not depend on the form of the initial data and the relative $L^2$-distances of the initial data play a crucial role. However, when we deal with the center-of-mass $x_c$, we used the Gaussian initial data so that they have the symmetric form (see Remark 4.2 in \cite{B-H-K-T}). For the numerical experiment, we introduce the following quantities (see \cite{B-H-K-T} for details):
 \begin{align*}
& R(t) := \textup{Re} \langle \psi_1,\psi_2\rangle(t),\quad \mathcal R_{ijk\ell} := \frac{(1-h_{ij})(1-h_{k\ell})}{(1-h_{i\ell})(1-h_{kj})}, \\
& \mathcal B := (\beta_{ij})_{1\leq i,j\leq N},\quad x_c^j(t) := \int_{\bbr} x|\psi_j(x,t)|^2 \d x, \\
& \mathcal E[\Psi] := \sum_{j=1}^N \int_{\bbr^d}\left[ \frac12 |\nabla\psi_j|^2 + V_j |\psi_j|^2 + \frac12\sum_{k=1}^N \beta_{jk} |\psi_k|^2 |\psi_j|^2  \right] \d x.
 \end{align*}

\begin{example}
\label{eg:1d-case}
{\em
 Here, we consider the two-component system in 1D, i.e., we take $N=2$ and $d=1$ in \eqref{S-L}.  To this end, we
 take $(x_1^0, x_2^0) = (2.5, -5)$ and consider the following two cases:   for $j=1,2,$
  \begin{itemize}
 \item[]{\bf Case 1.}  fix  $\alpha_j=\beta_{j\ell}=1$  ($\ell=1,2$) and  vary $\kappa=\textcolor{black}{0,2,20}$.
  \item[]{\bf Case 2.} fix $\alpha_j=j$,   $\beta_{12}=\beta_{21}=1$, $\beta_{11}=4\beta_{22}=2$ and  vary $\kappa=\textcolor{black}{0,2,10, 20}$.
 \end{itemize}
 Figure \ref{fig:quant_case1} and Figure \ref{fig:quant_case2} depict  the time evolution of the quantity $1-R(t)$ (where $R(t)$  is the real part of the correlation function $h
 _{12}(t)$), the center of mass  $x_c^j(t)$, the component mass $\|\psi_j\|^2$ and the total energy
 $\mathcal{E}(t)$ for {\bf Case 1} and {\bf Case 2}, respectively. From these  figures and
 other numerical experiments not shown here for brevity, we can see the following observations: \newline
 
\noindent (i). For all cases, we observe that the mass is conserved along time. 

\vspace{0.2cm}

\noindent (ii).  If the Lohe coupling is off, i.e., $\kp=0$,  both the mass and energy are  conserved well, and the center of mass  ($x_c^1(t),x_c^2(t))$  are periodic in time with the same period.  In addition, for the identical case, i.e.,  $\mathcal{B}=J_2$ and $V_1(x)=V_2(x)$, $R(t)$ is conserved for identical case. 

\vspace{0.2cm}

\noindent (iii). If the Lohe coupling is on, i.e., $\kp>0$,  the phenomena become complicated.  The energy is no longer conserved, indeed it decays to some value for large $\kappa$ while it oscillates for small $\kappa$. 

\vspace{0.2cm}

\noindent (iv). Moreover, for the identical case, $R(t)$  converges exponentially to 1, which coincides with the theoretical results.  Thus,   complete state aggregation occurs in this case.  After complete state aggregation, 
 $\|\psi_1(x,t)-\psi_2(x,t)\|_\infty$ will converge to zero and the center of mass  $x_c^1(t)$ and $x_c^2(t)$ will become the same and swing periodically along the line
 connecting $-\bar{x}_c^0$ and $\bar{x}_c^0$ (here, $\bar{x}_c^0:=(x_c^1(0)+x_c^2(0))/2$). 
 
 \vspace{0.2cm}
 
\noindent  (v). Furthermore, for the non-identical case, i.e., $\mathcal{B}\ne J_2$ and $V_1(x)\ne V_2(x)$, $R(t)$  does not
 converge to 1, i.e.,   complete state aggregation cannot occur. However, for large $\kappa$,
 $R(t)$ indeed converges to some definite constant $R_\infty<1$. The larger  $\kappa$, the smaller value $1-R_\infty$.  Meanwhile, $|x_c^1(t)-x_c^2(t)|$ also converges to zero, which could be also justified in a similar process as shown in Corollary 4.1 of \cite{B-H-K-T}.
  }

\begin{figure}[h]
\centering
\emph{(a)
\subfigure{\includegraphics[width=0.3\textwidth]{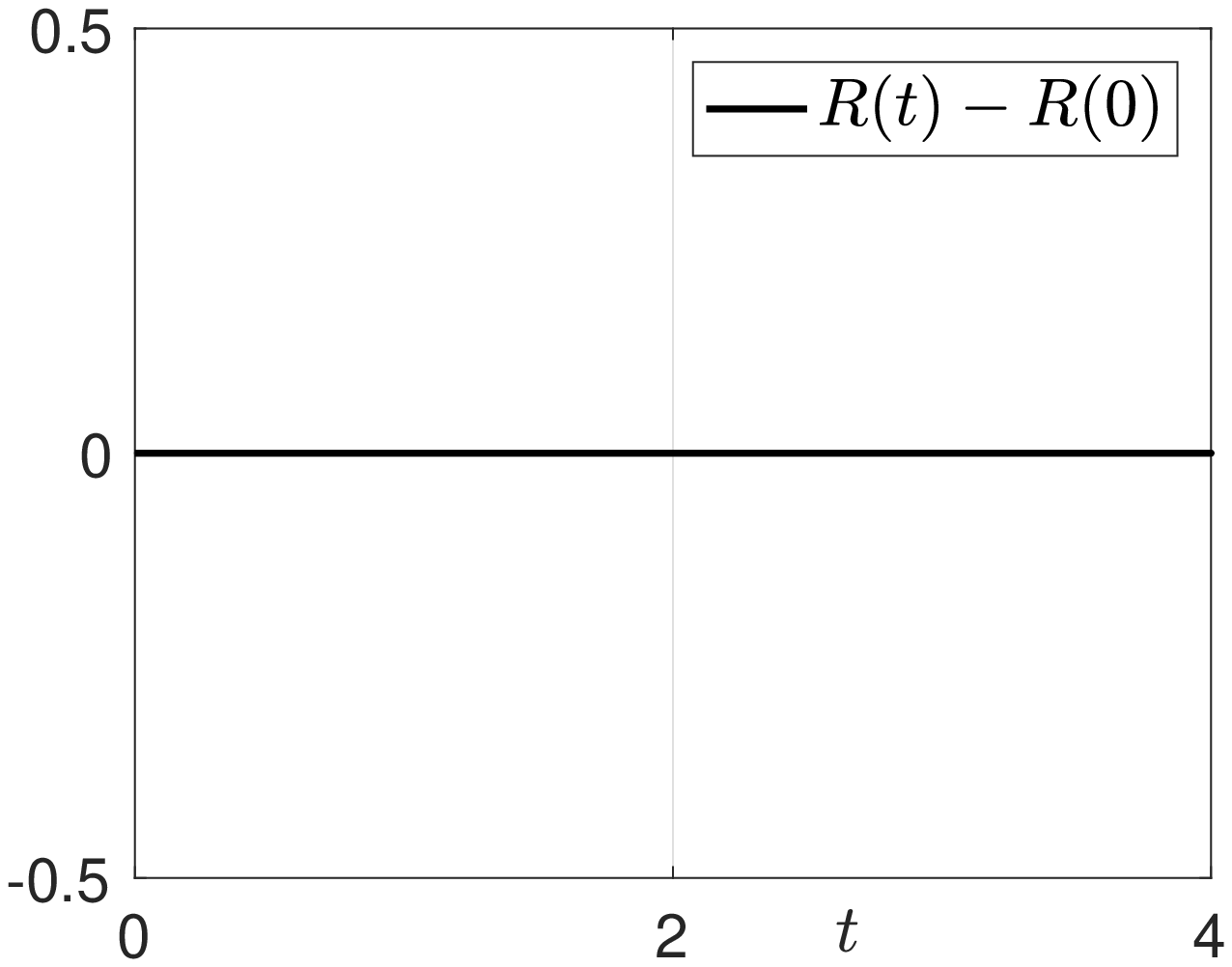}}
\subfigure{\includegraphics[width=0.3\textwidth]{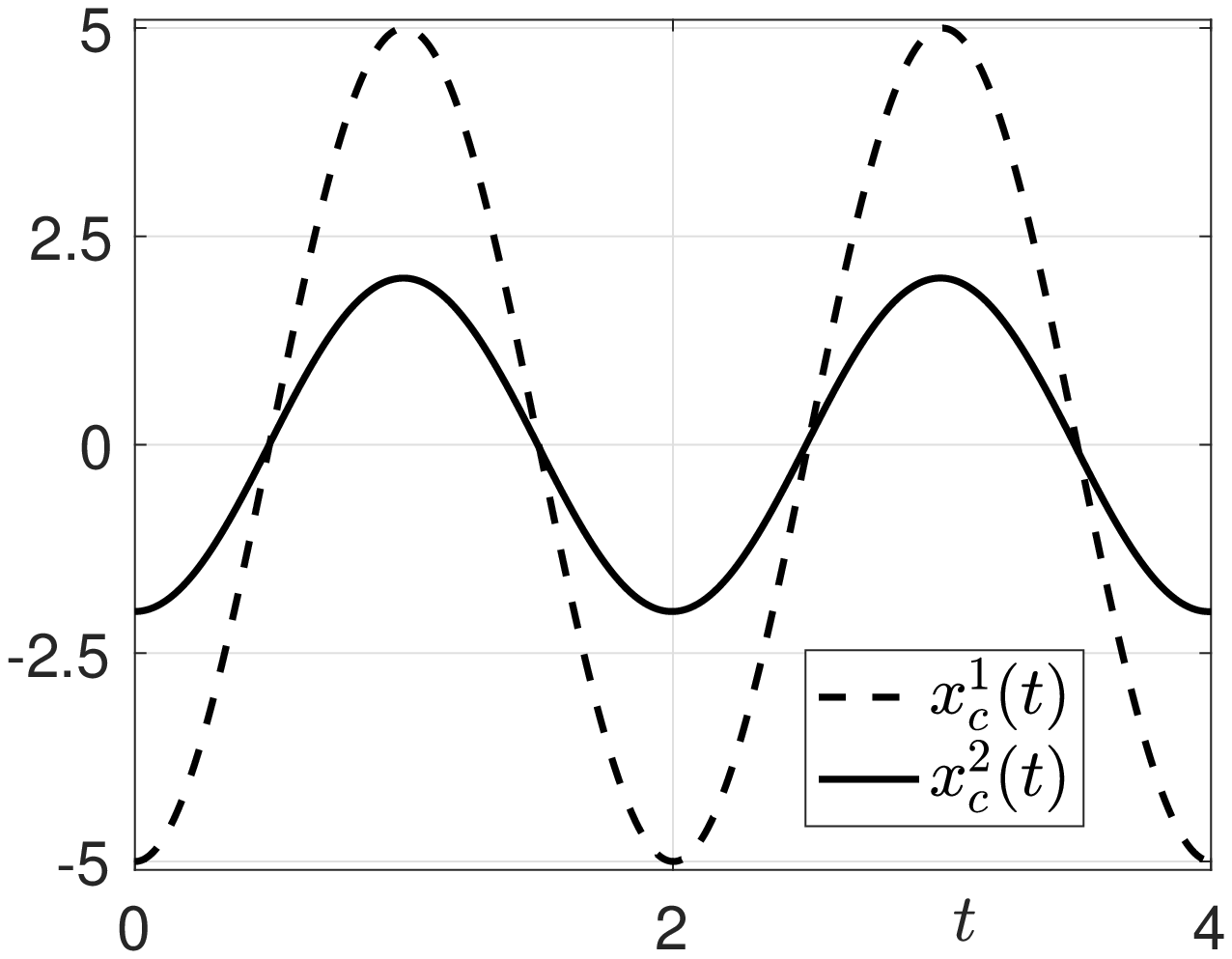}}
\subfigure{\includegraphics[width=0.3\textwidth]{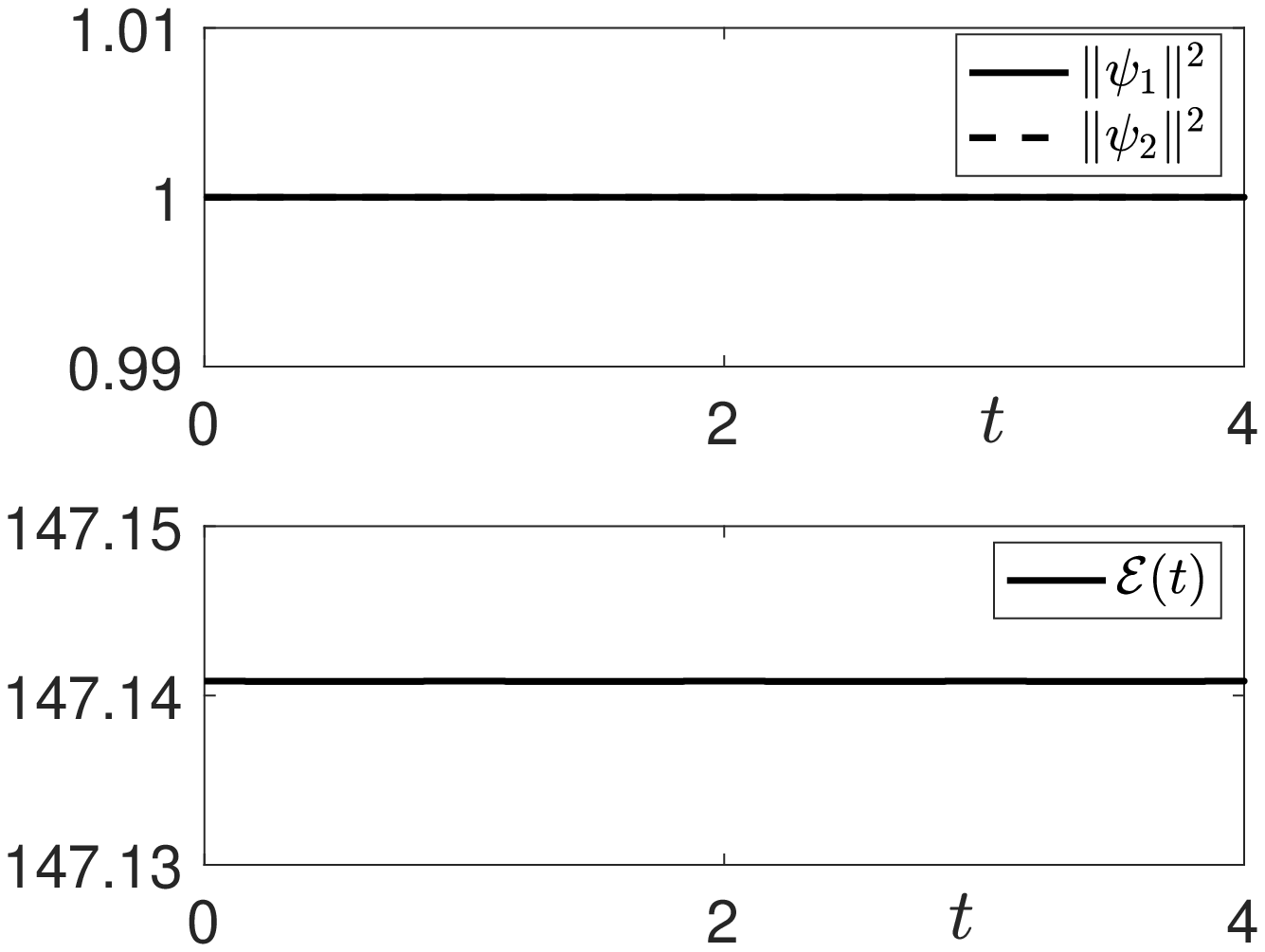}}
}
\centering{\emph{(b)}
\subfigure{\includegraphics[width=0.3\textwidth]{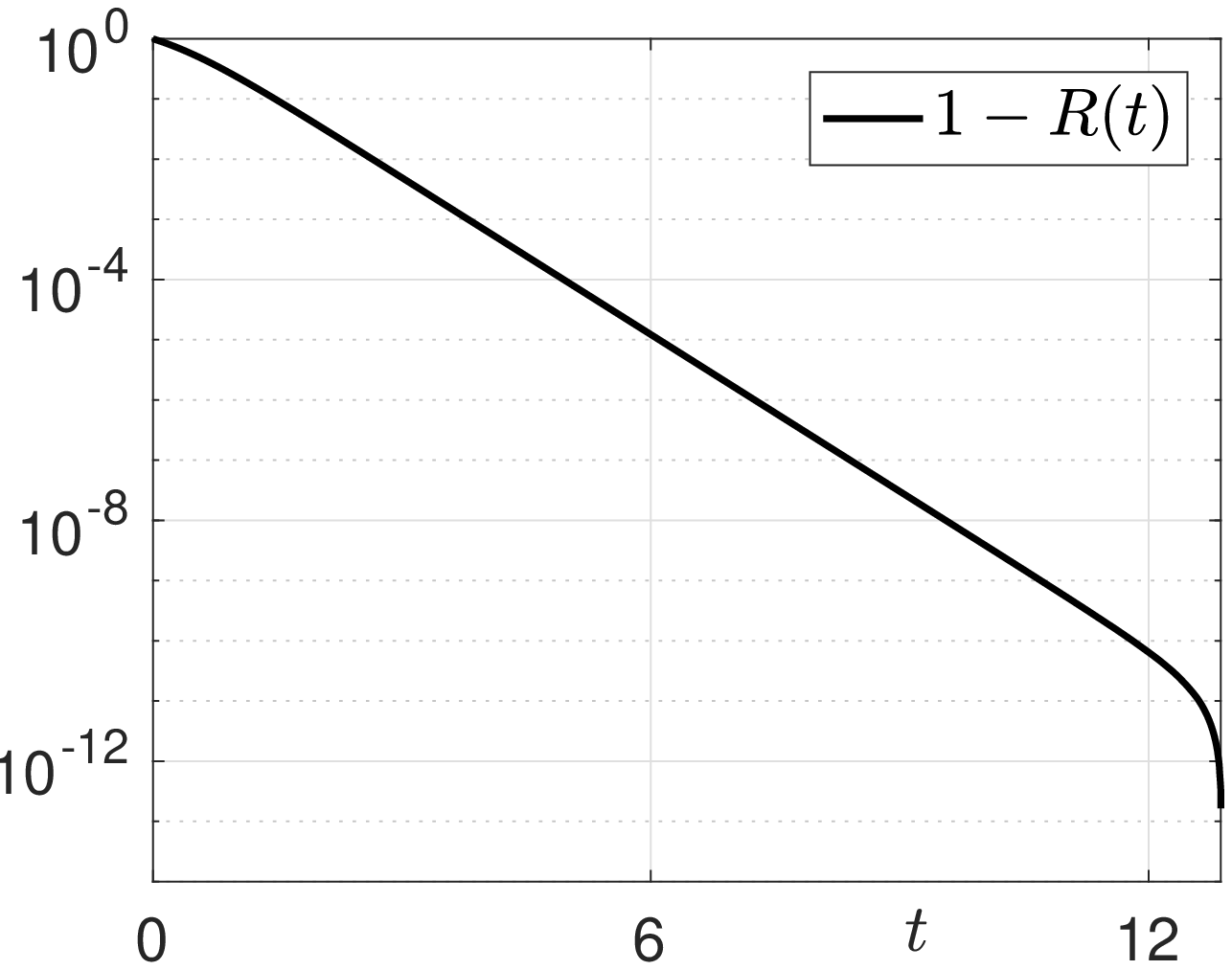}}
\subfigure{\includegraphics[width=0.3\textwidth]{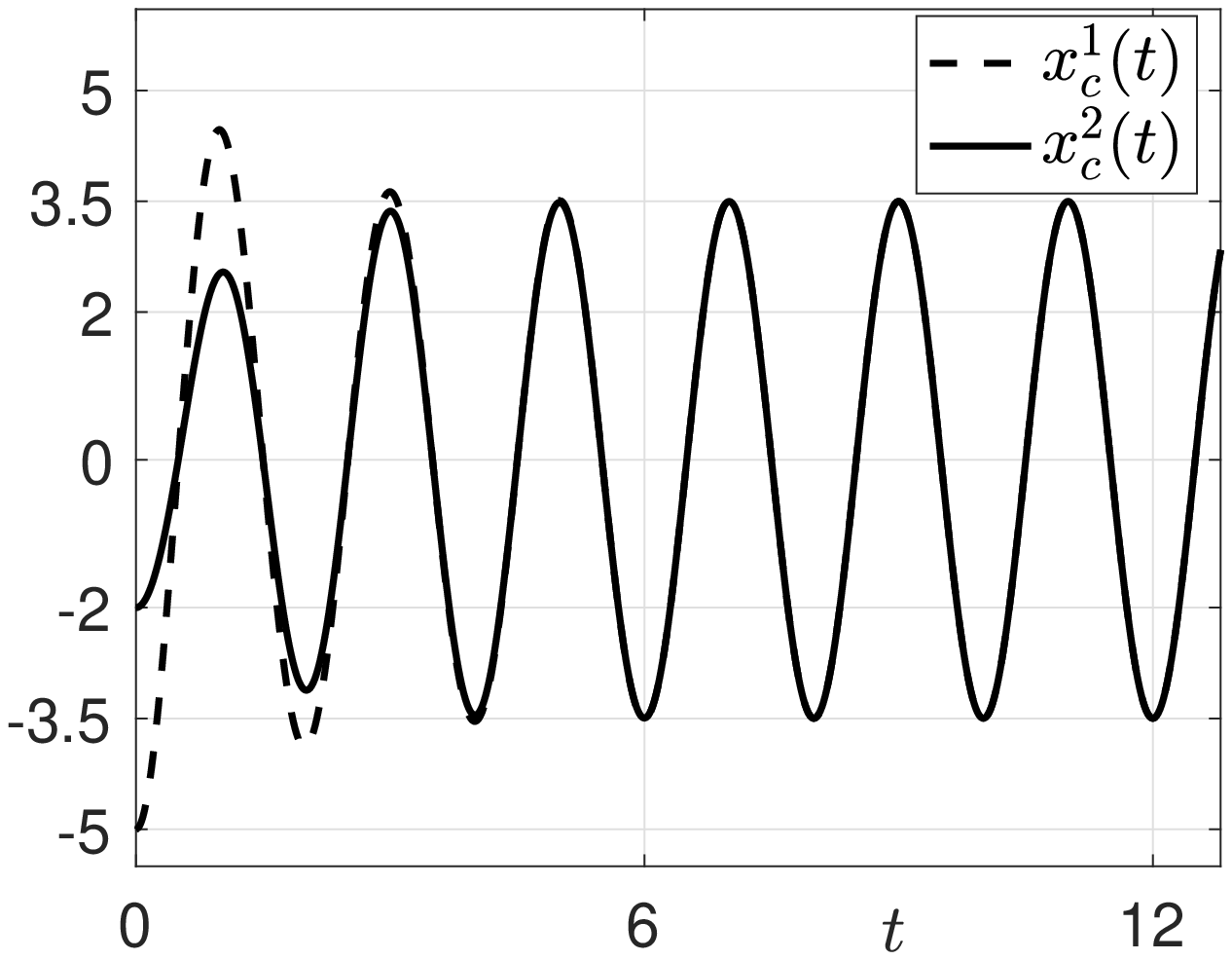}}
\subfigure{\includegraphics[width=0.3\textwidth]{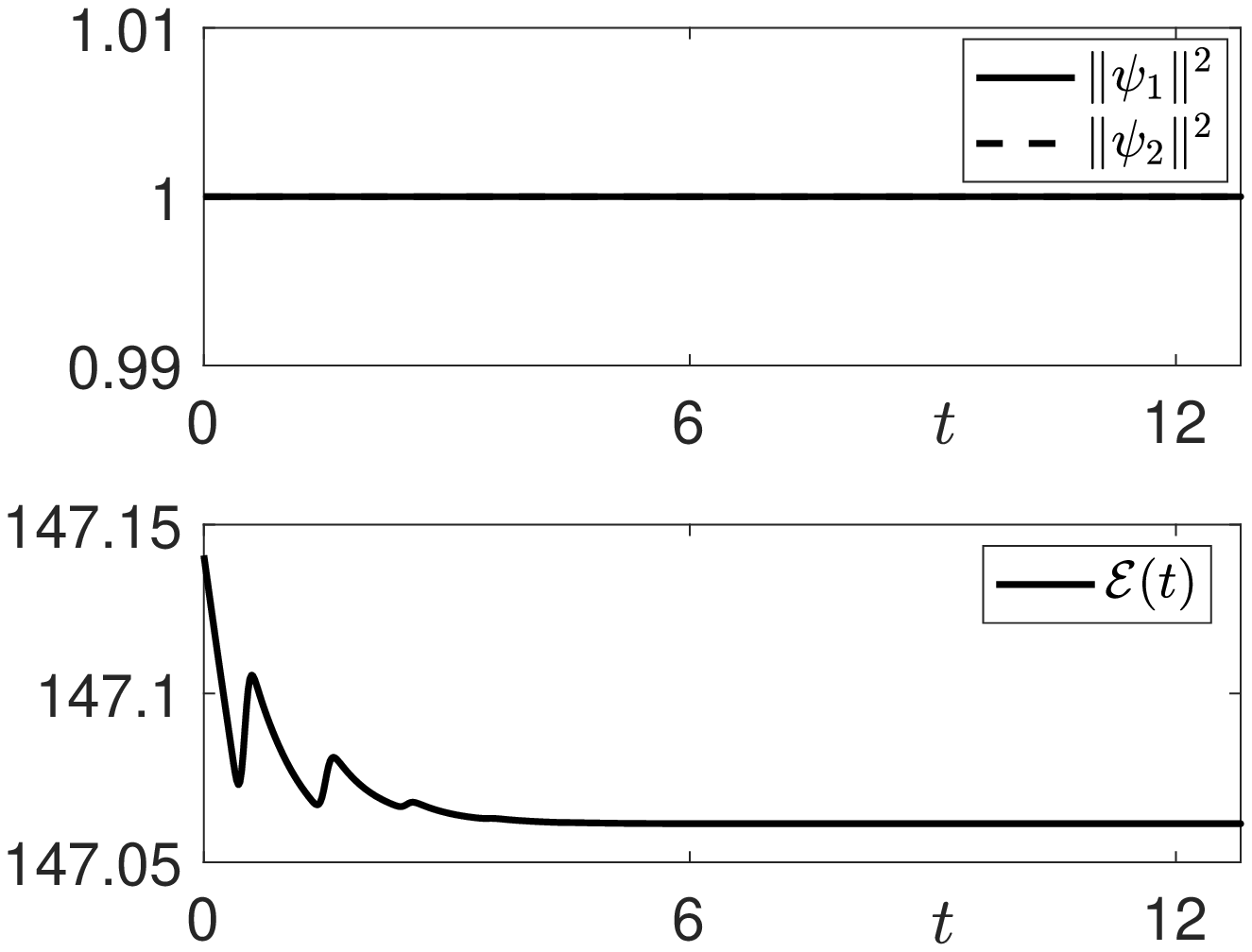}}  }
\centering{\emph{(c)}
\subfigure{\includegraphics[width=0.3\textwidth]{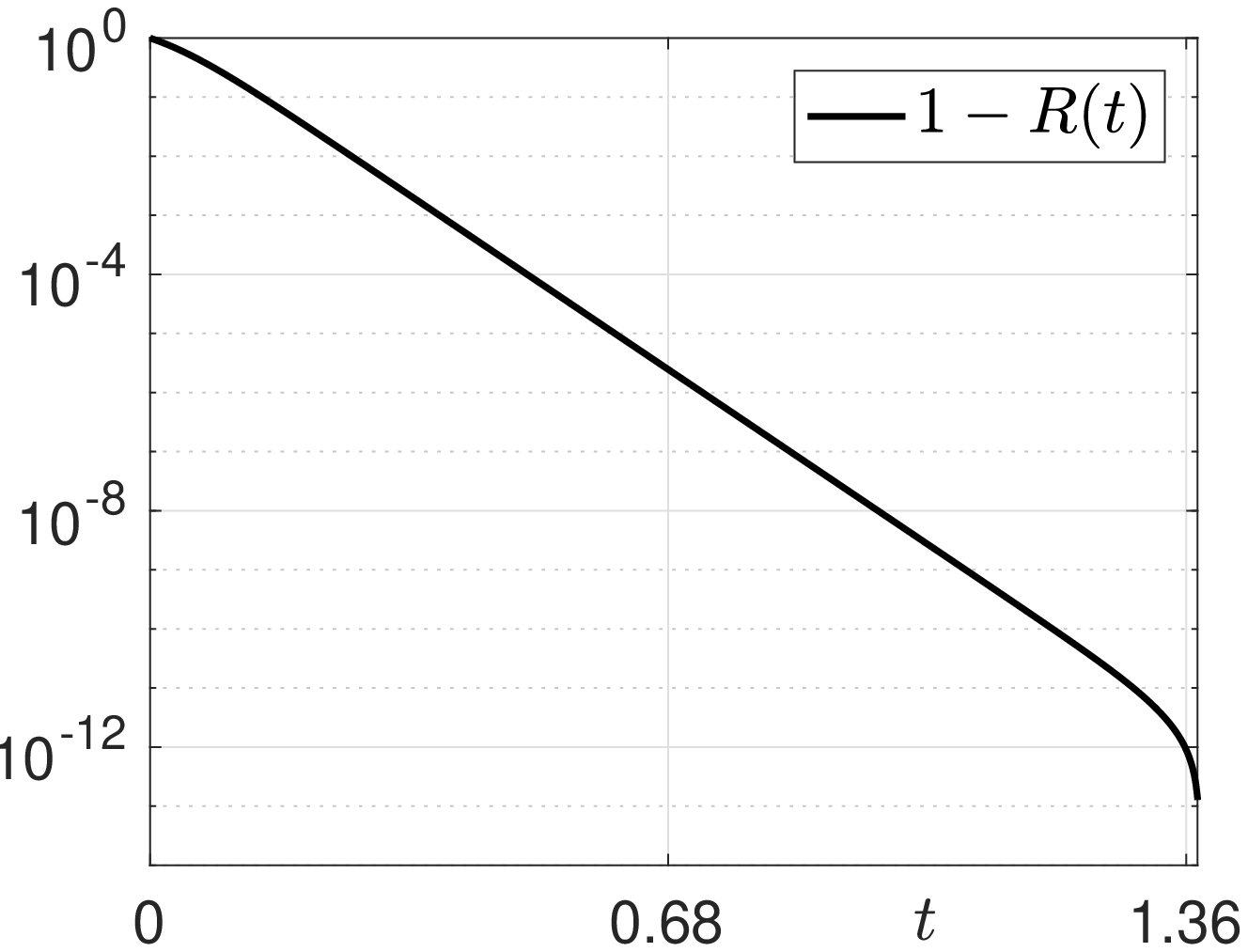}}
\subfigure{\includegraphics[width=0.3\textwidth]{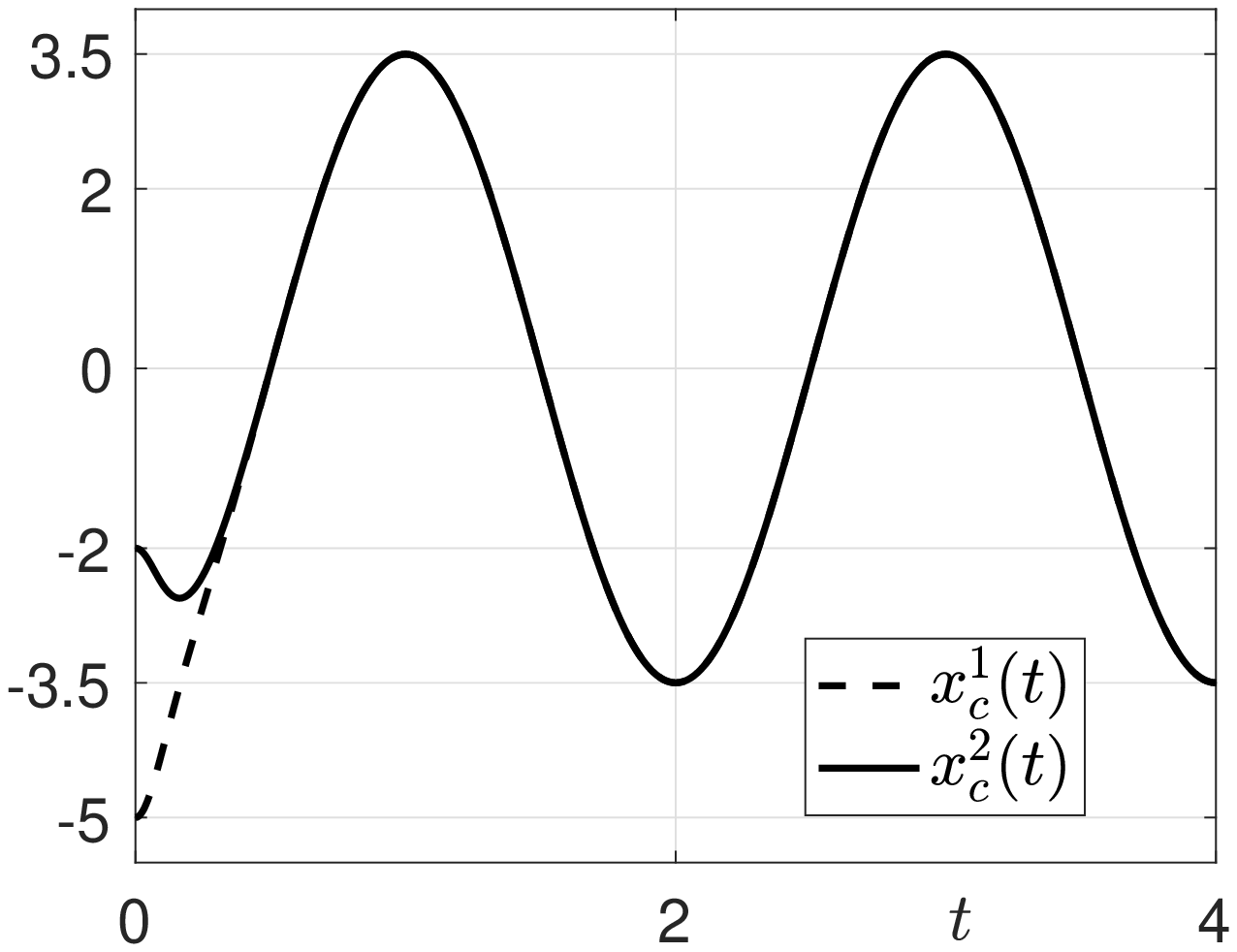}}
\subfigure{\includegraphics[width=0.3\textwidth]{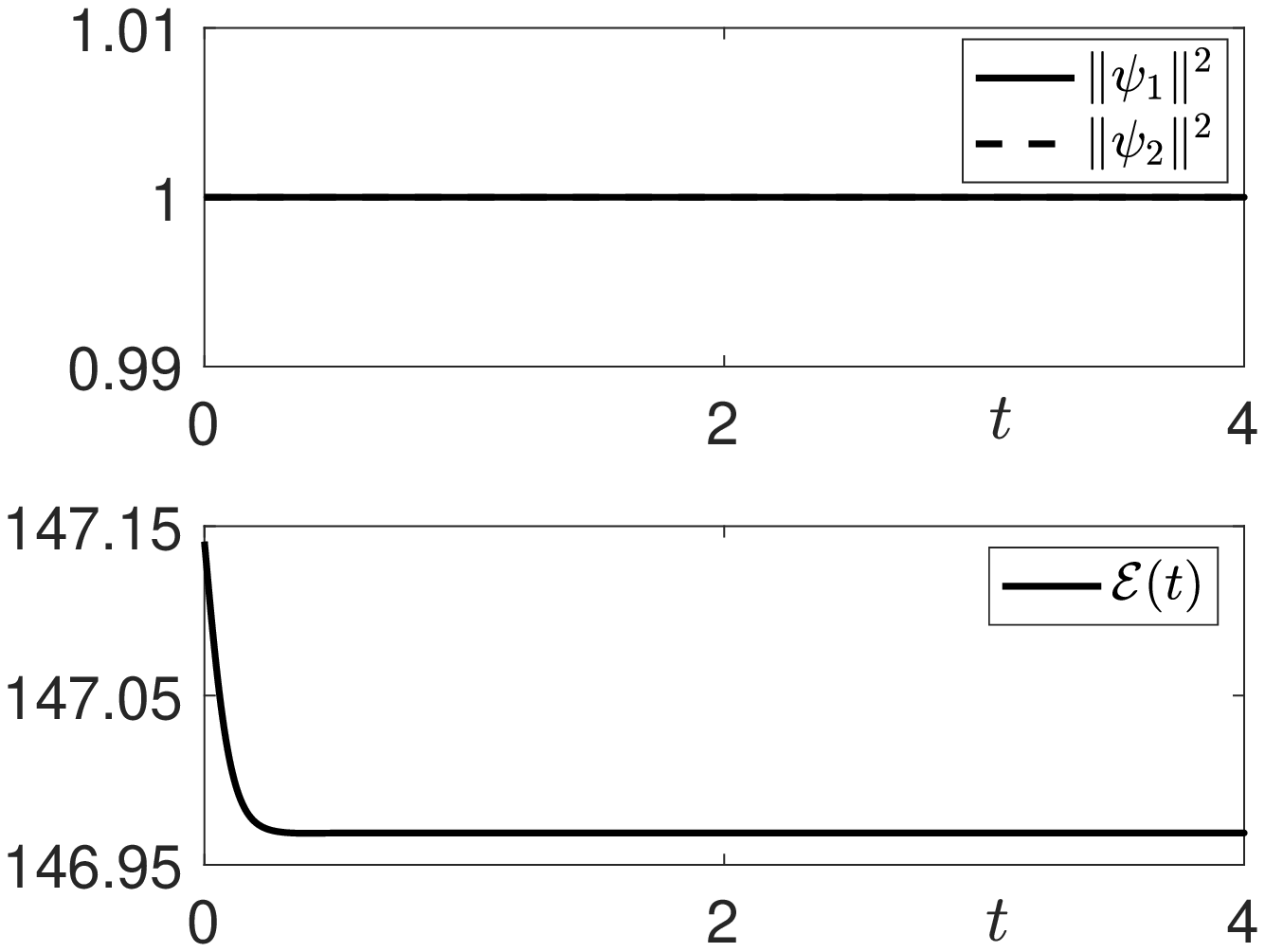}} }
\vspace{-0.3cm}
\caption{Time evolution of the quantity $1-R(t)$ (left), the center of mass  $x_c^j(t)$ (middle),
and  the component mass $\|\psi_j\|^2$ and the total energy  $\mathcal{E}(t)$ (right) for {\bf Case 1} in Example \ref{eg:1d-case}
for  \textcolor{black}{$\kp=0, 2, 20$} (top to bottom).}
\label{fig:quant_case1}
\end{figure}

\begin{figure}[h]
\centering{
\emph{(a)}
\subfigure{\;\,\includegraphics[width=0.3\textwidth]{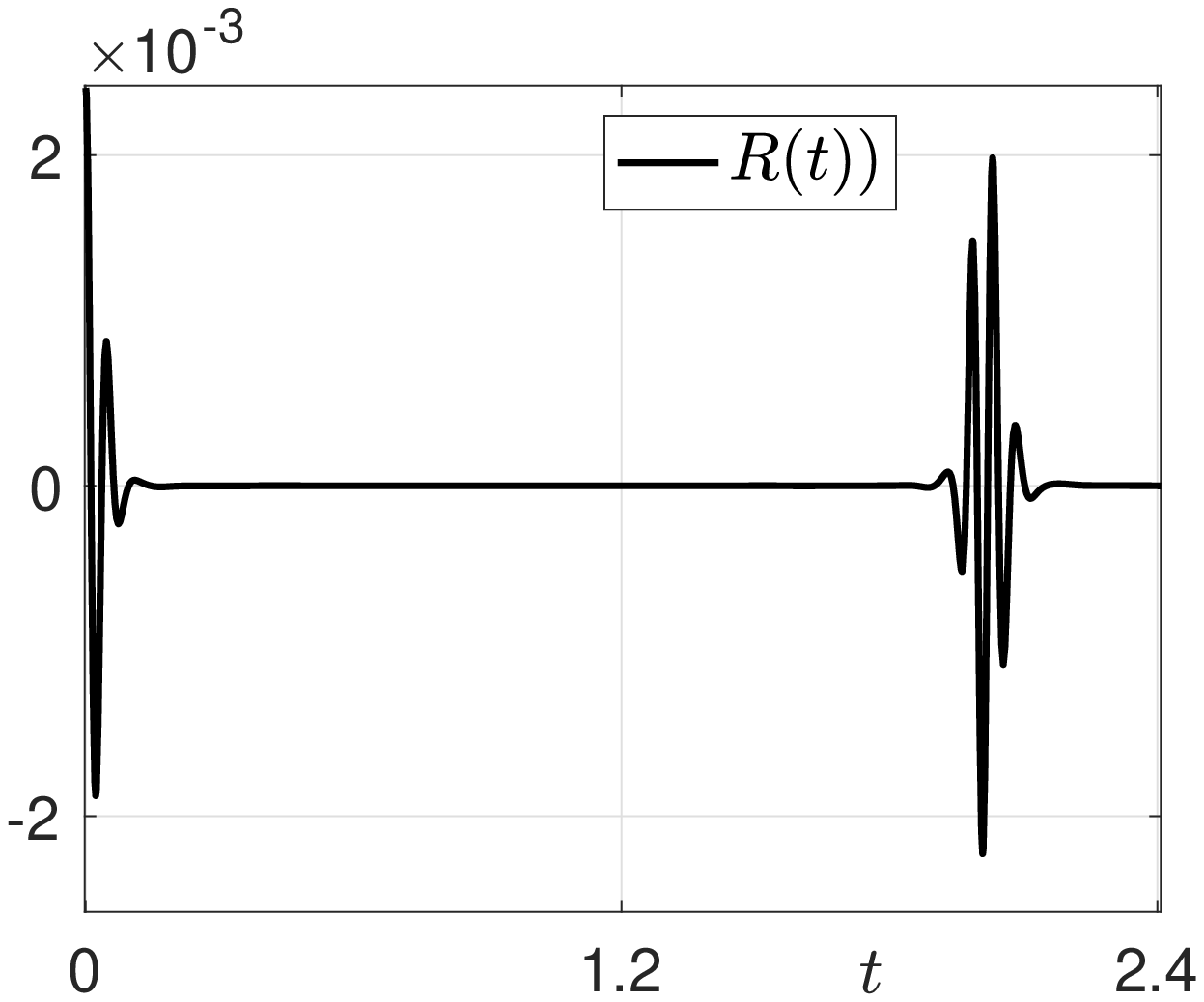}} 
\subfigure{\includegraphics[width=0.3\textwidth]{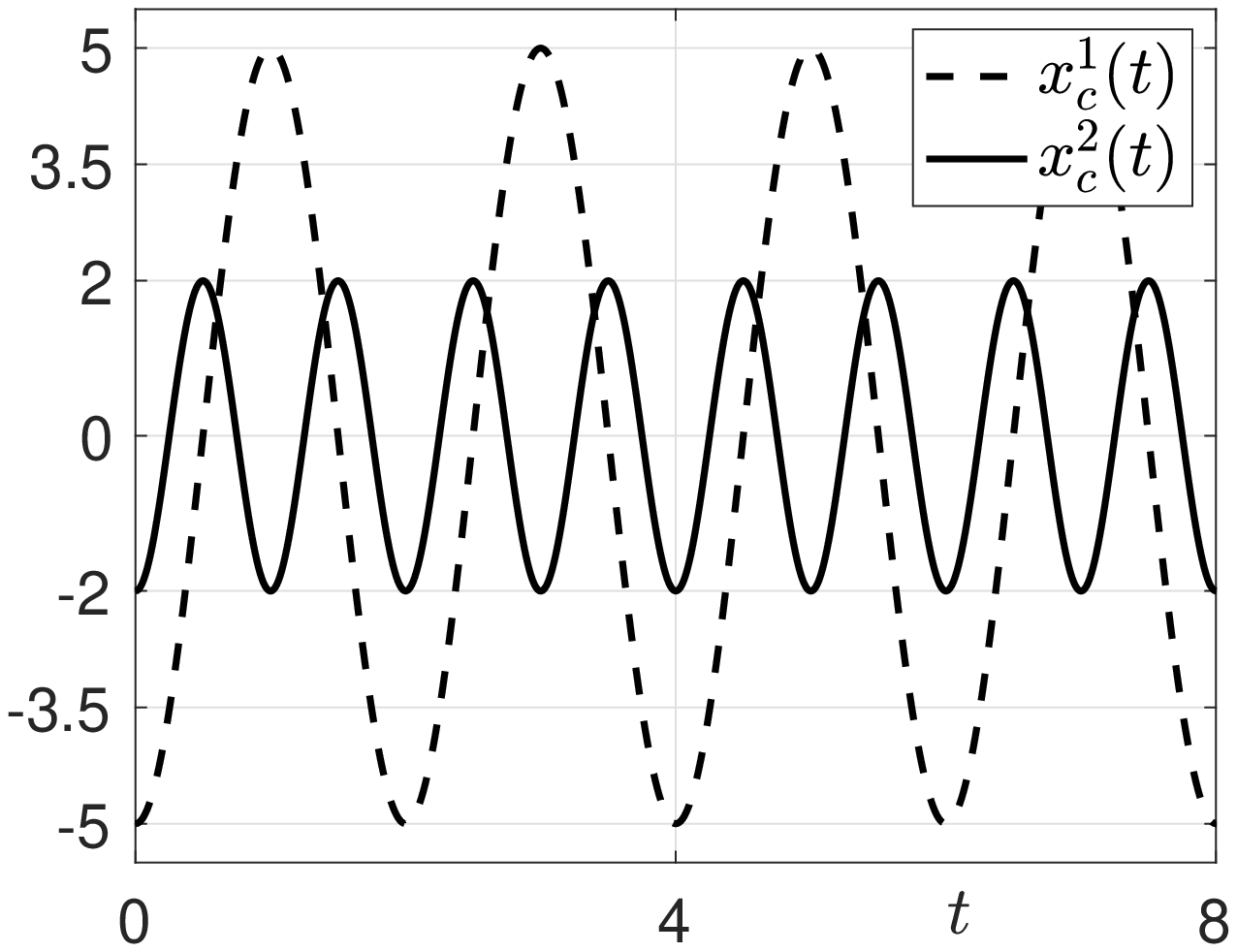}} 
\subfigure{\,\includegraphics[width=0.3\textwidth]{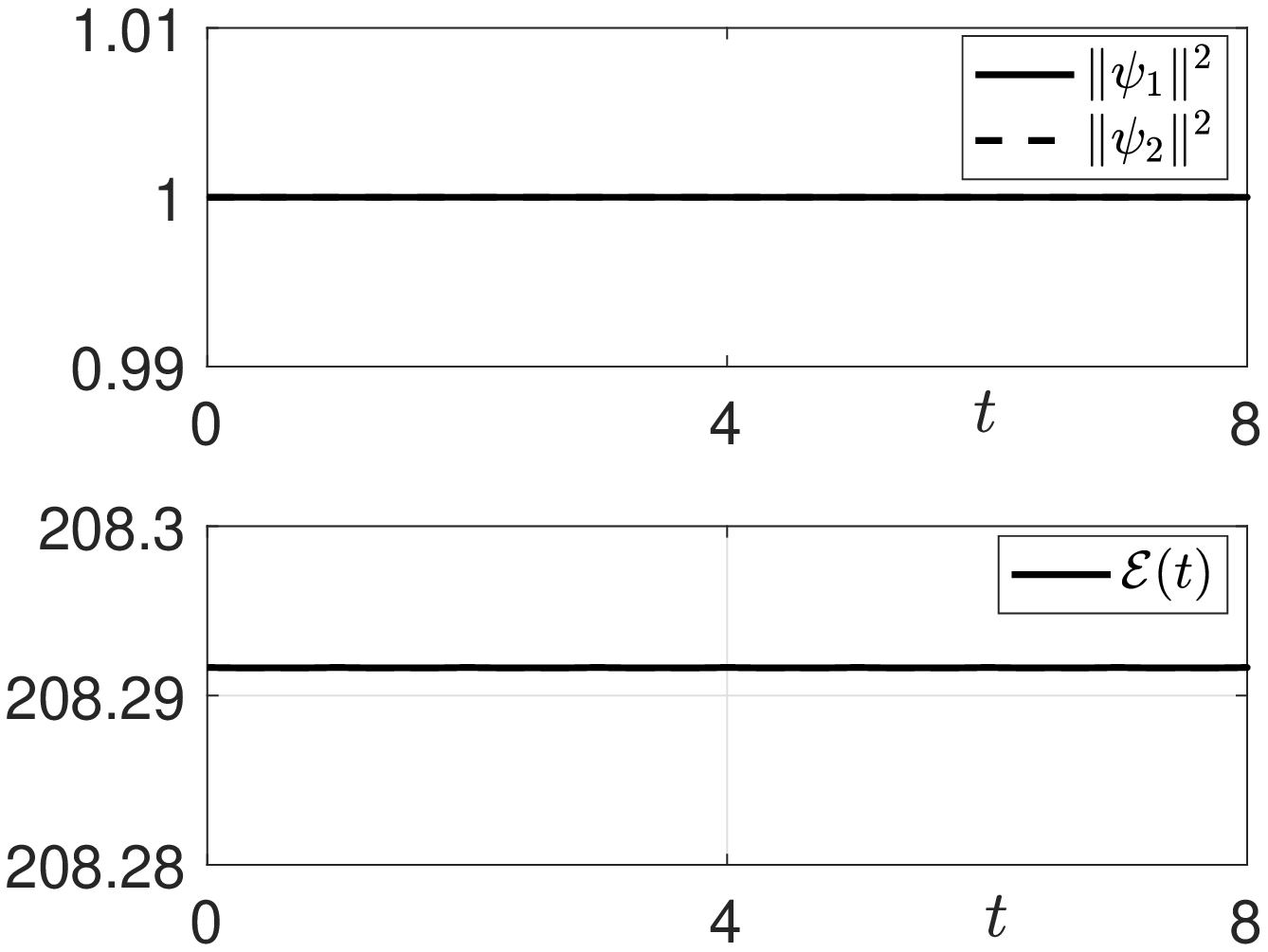}} }
\centering{
\emph{(b)}
\subfigure{\includegraphics[width=0.3\textwidth]{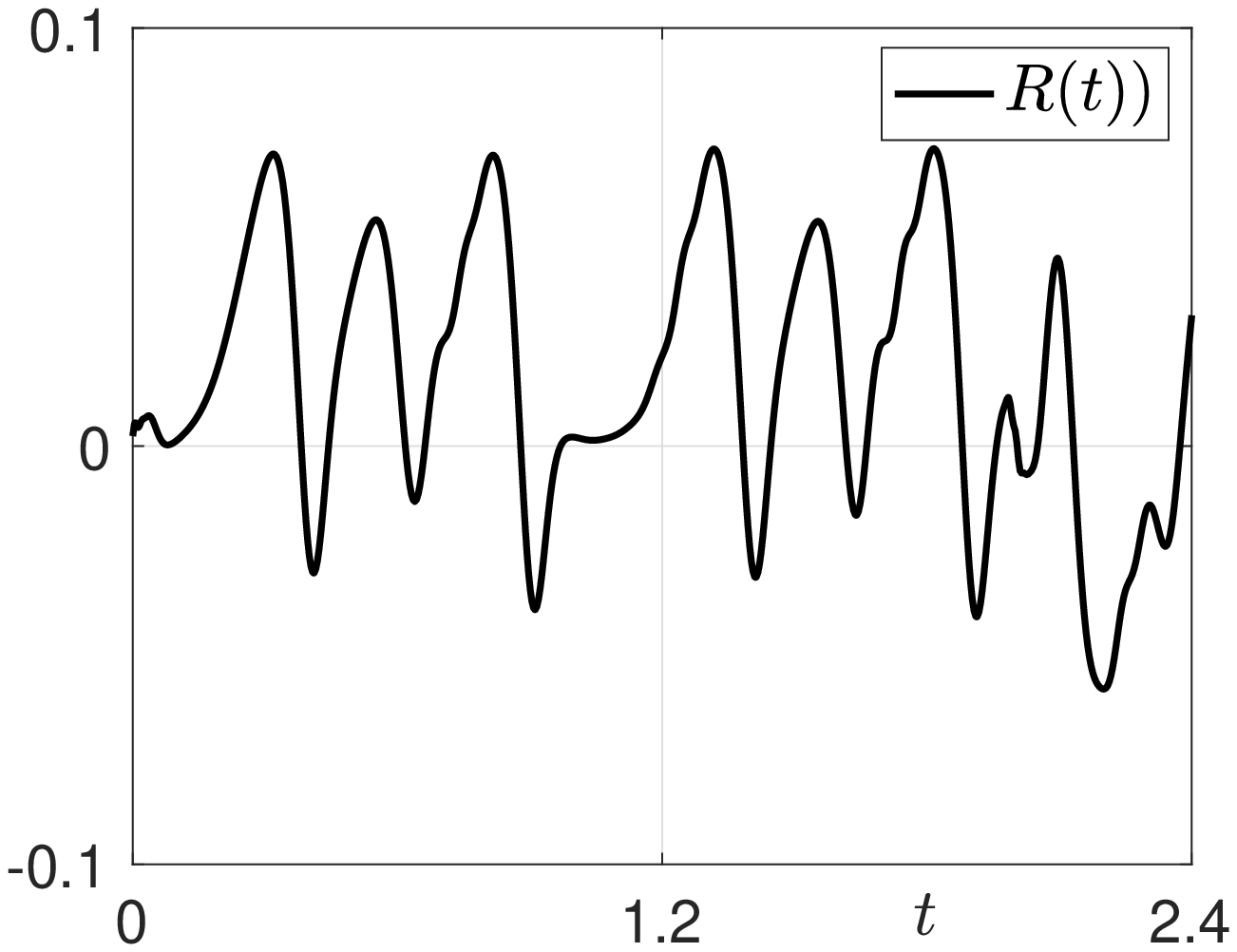}} 
\subfigure{\includegraphics[width=0.3\textwidth]{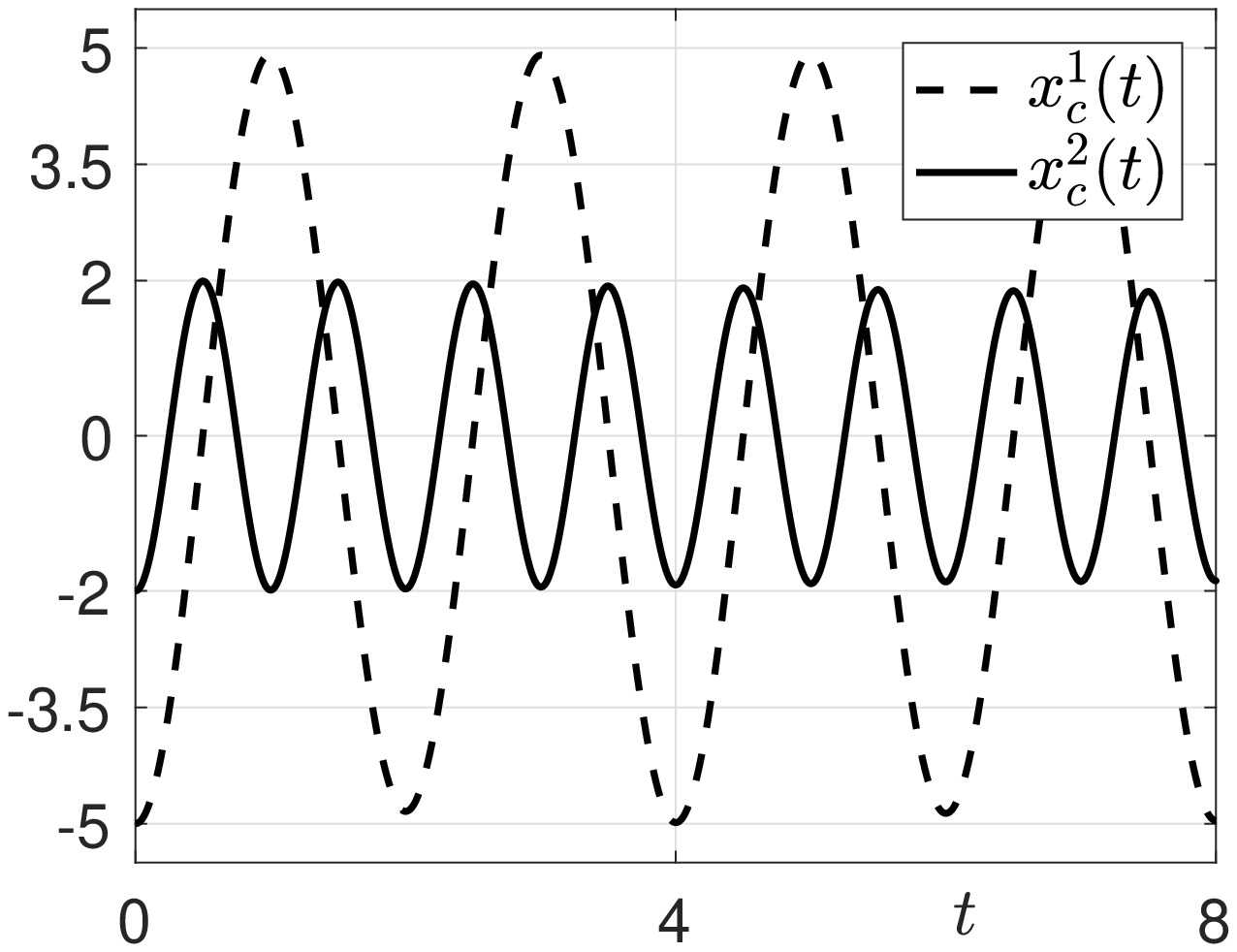}} 
\subfigure{\includegraphics[width=0.3\textwidth]{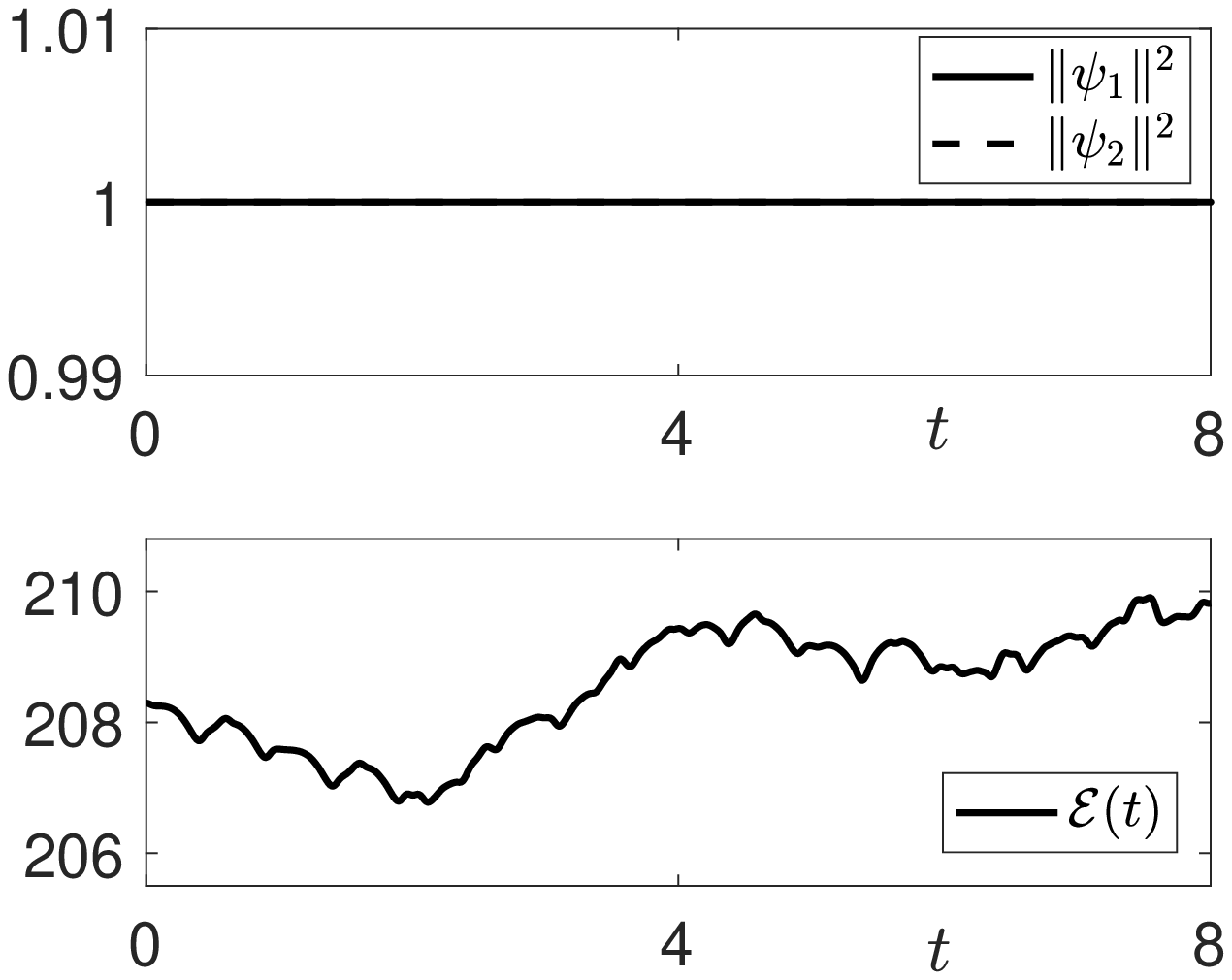}}
}
\centering{
\emph{(c)}
\subfigure{\includegraphics[width=0.3\textwidth]{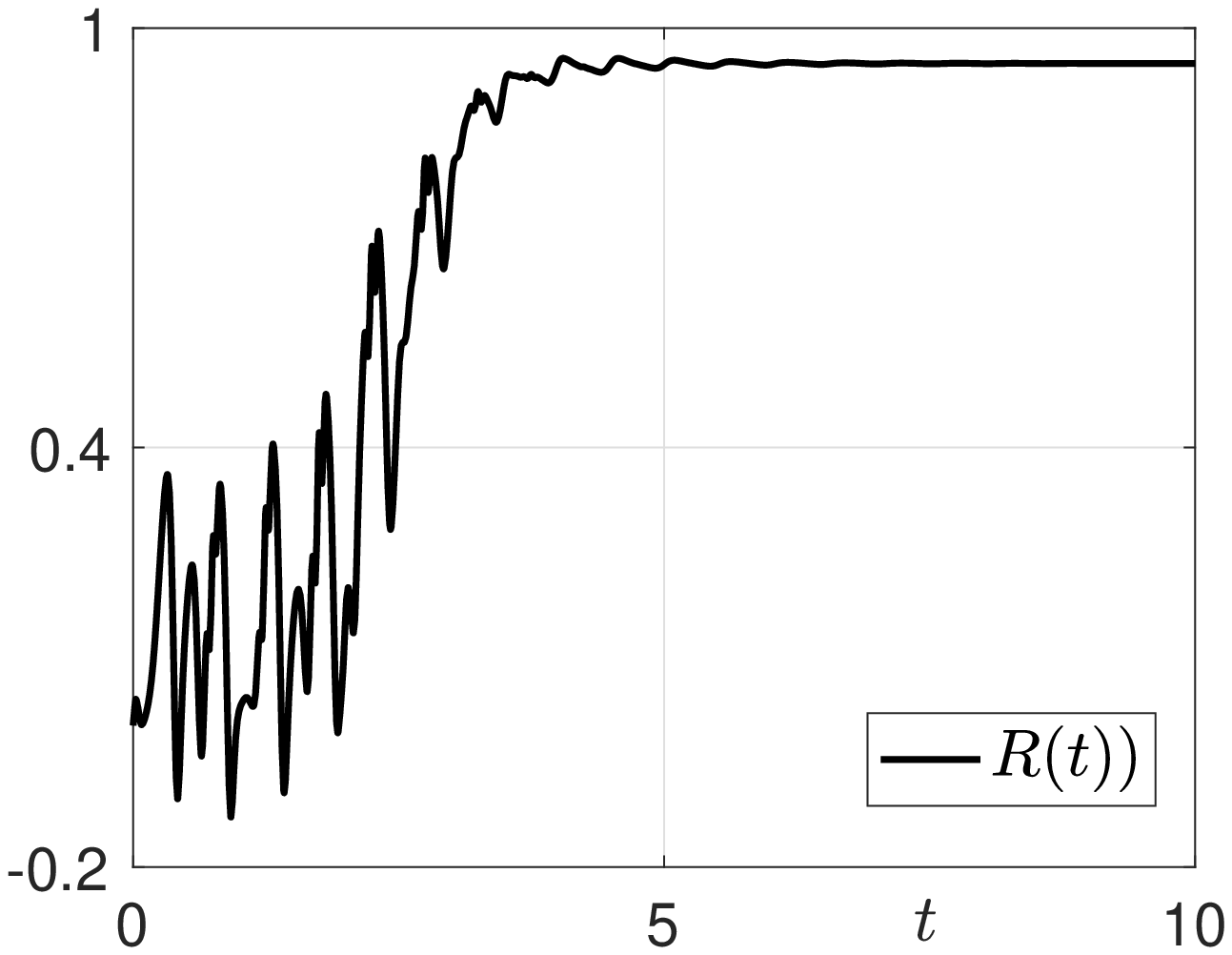}} 
\subfigure{\includegraphics[width=0.3\textwidth]{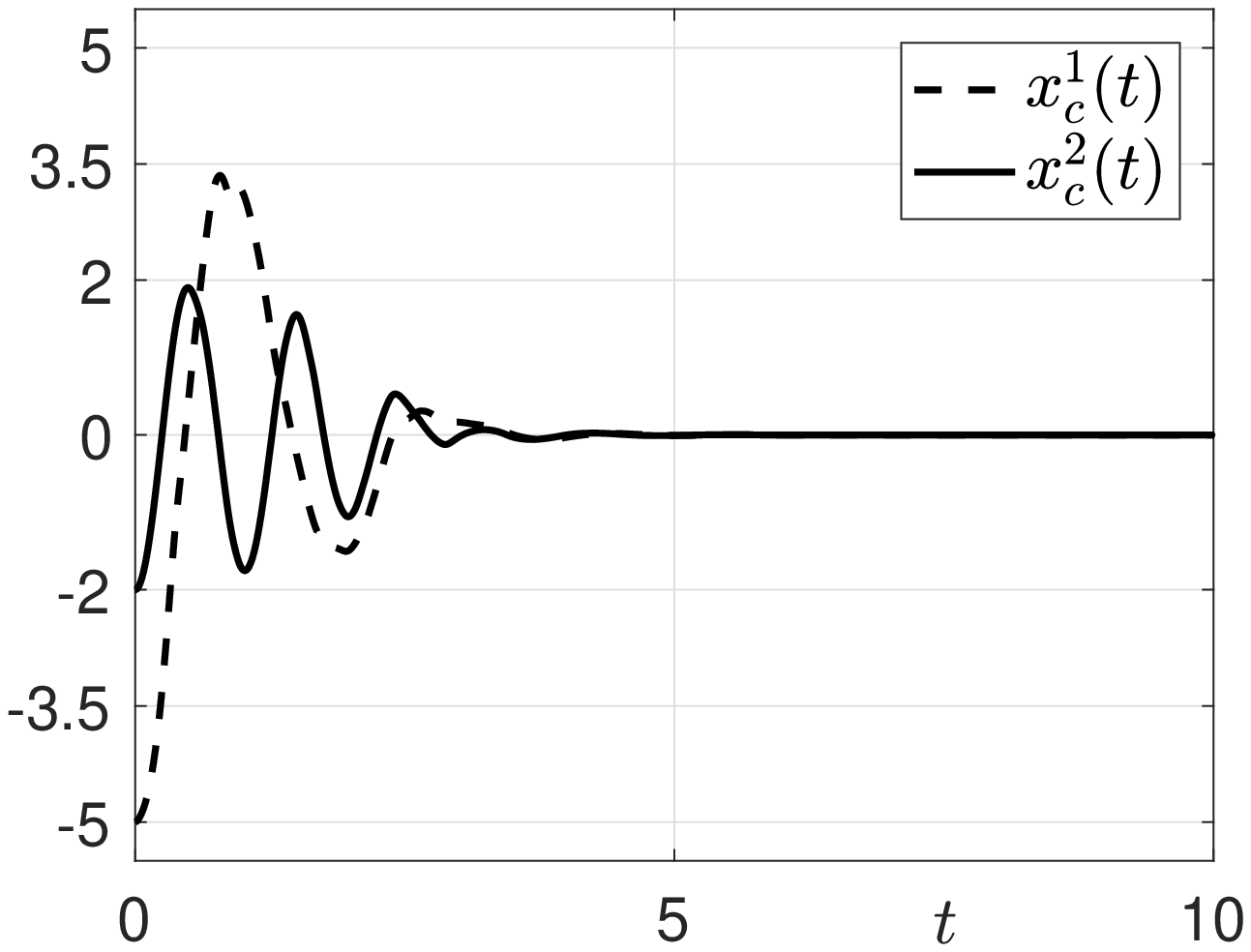}}
\subfigure{\,\includegraphics[width=0.3\textwidth]{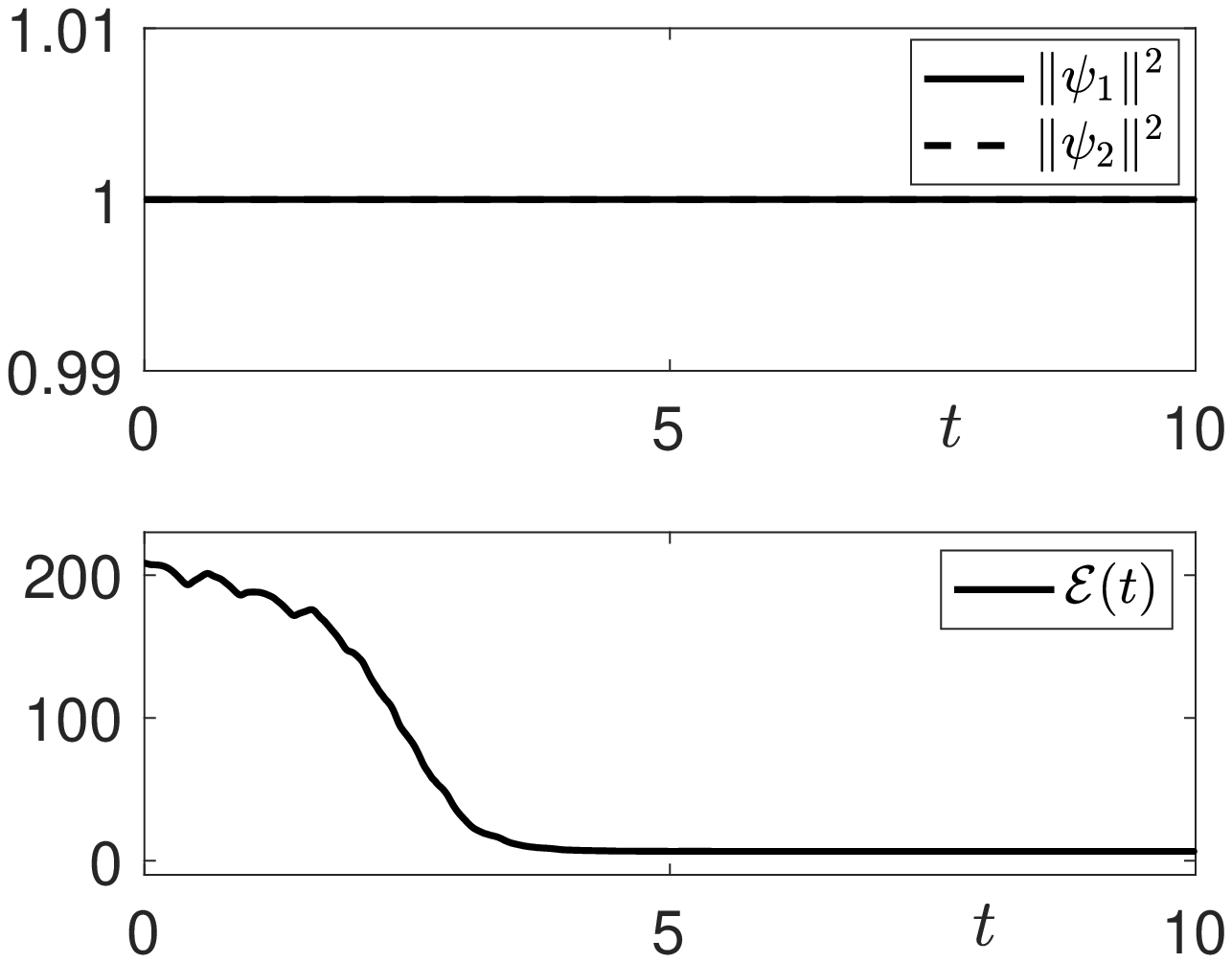}} 
}
\centering{
\emph{(d)}
\subfigure{\includegraphics[width=0.3\textwidth]{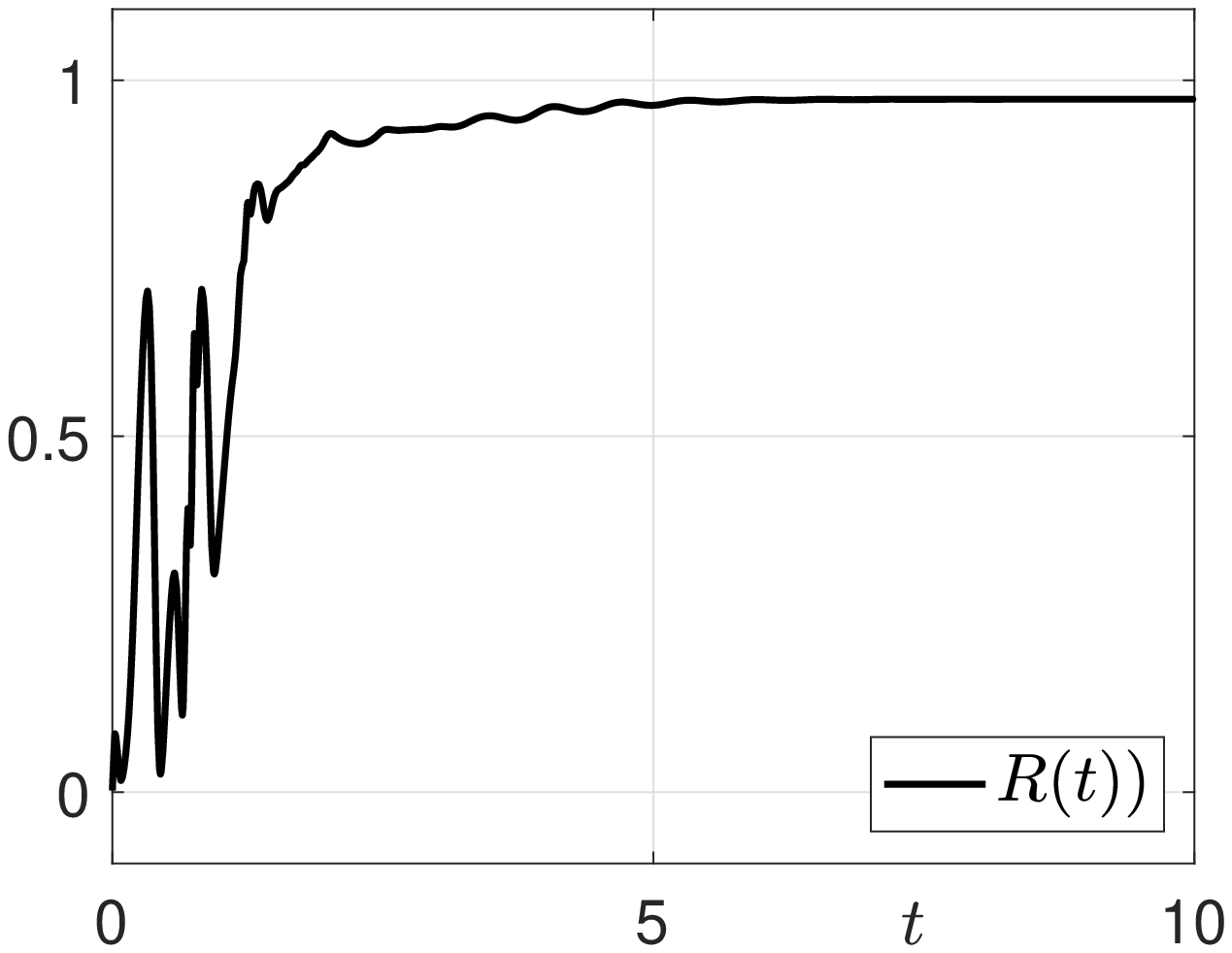}}
\subfigure{\includegraphics[width=0.3\textwidth]{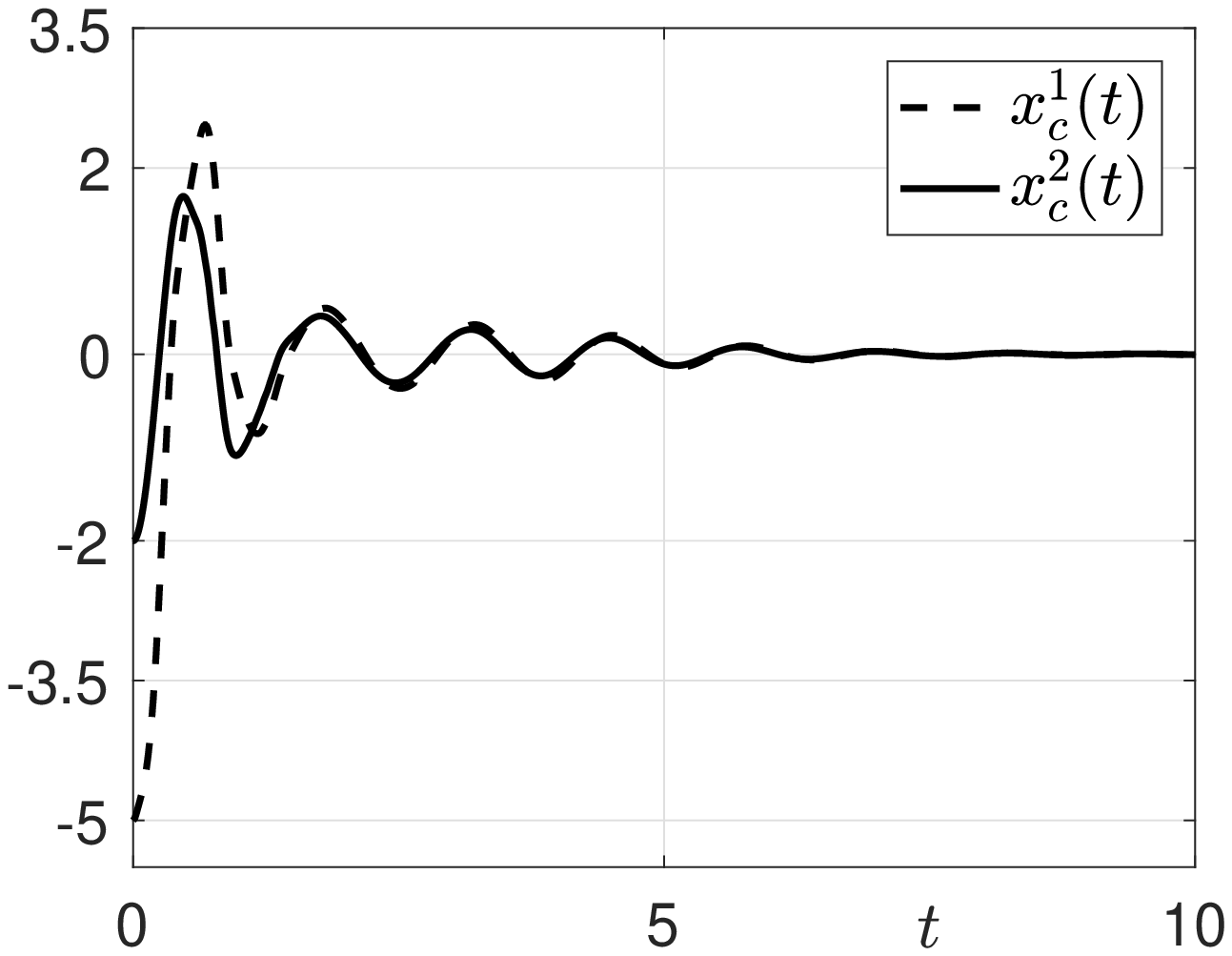}} 
\subfigure{\,\includegraphics[width=0.3\textwidth]{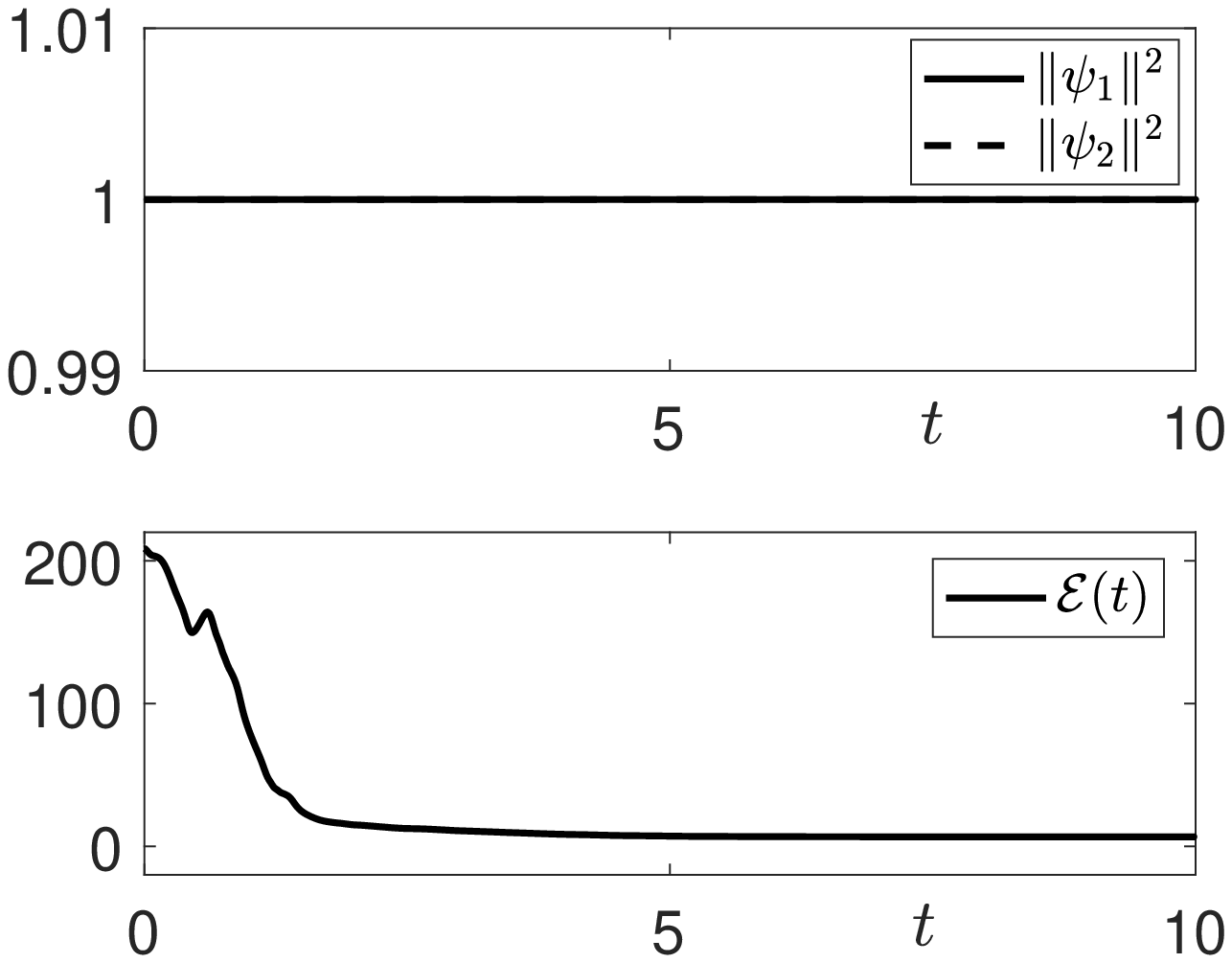}} 
}
\vspace{-0.3cm}
\caption{Time evolution of the quantity $1-R(t)$ (left), the center of mass  $x_c^j(t)$ (middle),
and  the component mass $\|\psi_j\|^2$ and the total energy  $\mathcal{E}(t)$ (right) for {\bf Case 2} in Example \ref{eg:1d-case}
for   \textcolor{black}{$\kp=0, 2, 10, 20$} (top to bottom). }
 \label{fig:quant_case2}
\end{figure}
 \end{example}

\begin{example}
\label{eg:2d-case}{\em
Here, we consider the  six-component system in 2D, i.e., we take $N=6$ and $d=2$ in \eqref{S-L}. To this end,  we here only consider the
identical case, i.e.,  we choose $\alpha_j=1=\beta_{j\ell}=1$ ($j,\ell=1,\cdots,6$ ). Let  \textcolor{black}{$\kappa=20$}, we
 consider   four cases of initial setups:
\begin{itemize}
\item[]{\bf Case 3.} $x_0^j=\big(6\, \cos((j-1)\pi/3), 6\, \sin((j-1)\pi/3)\big)$, \quad $j=1,\cdots,6$.
\item[]{\bf Case 4.} $x_0^j=\big(2+4\, \cos(j\pi/3-\pi/12), 2+4\, \sin(j\pi/3-\pi/12)\big)$, \quad  $j=1,\cdots,6$.
\item[]{\bf Case 5.} $x_0^j=\big(6\, \cos((j-1)\pi/5), 6\, \sin((j-1)\pi/5) \big)$, \quad  $j=1,\cdots,6$.
\item[]{\bf Case 6.} Random location:
\bea\nonumber
&&x_0^1=(3.4707, 2.7526 ), \quad x_0^2=(-0.8931, 1.9951), \\
\nonumber
&& x_0^3=(0.1809, -1.1538), \quad x_0^4=(0.0937, -5.8995 ), \\
\nonumber
&& x_0^5=(-2.9235, -2.4171),\quad  x_0^6=(-3.6423, 4.3714).
\eea
\end{itemize}
\textcolor{black}{For  {\bf Case 3-Case 6},  Figure \ref{fig:quant_case3to6} illustrates the trajectory and time evolution of the center of mass $x_c^j(t)=:(x_{c1}^j(t),  \,x_{c2}^j(t))$,  Figure \ref{fig:dyn_R_ijkl} depicts the time evolution of $|\mathcal{R}_{1256}(t)-\mathcal{R}_{1256}(0)|$, $|\mathcal{R}_{2456}(t)-\mathcal{R}_{2456}(0)|$ and  $|\mathcal{R}_{3456}(t)-\mathcal{R}_{3456}(0)|$, }and
 Figure \ref{fig:density_case3to6} shows the contour plots of  $|\psi_1(x,t)|^2$  at different times. 
From these figures and other numerical experiments not shown here for brevity, we can see the following observations. \newline

\noindent (i). Complete state aggregation occurs for all cases. 

\vspace{0.2cm} 
\noindent (ii). All the center of mass  $x_c^j(t)$ ($j=1,\cdots,6$) will converge to the same periodic function $\bar{x}_c(t)$,  which swings exactly along the line connecting the points ($-\bar{x}_{c1}^0, -\bar{x}_{c2}^0$) and ($\bar{x}_{c1}^0, \bar{x}_{c2}^0$ which are defined as the average of the initial center of mass of the six oscillators:
$$
\left(\bar{x}_{c1}^0, \bar{x}_{c2}^0\right):=\left(\frac16\sum_{j=1}^{6} x^j_{c1}(0), \frac16\sum_{j=1}^{6} x^j_{c2}(0)\right).
$$
Thus, when $\bar{x}_{c1}^0=\bar{x}_{c2}^0=0$, the center of mass will stay steady at
the origin (cf. Figure \ref{fig:quant_case3to6} (a)), which also agrees with the conclusion in Remark 4.2 of \cite{B-H-K-T}.  \newline

\noindent \textcolor{black}{(iii)
Before complete state aggregation, all density profiles $|\psi_j(x,t)|^2$ ($j=1,\cdots,6$) will evolve similarly, i.e., the same dynamical pattern  as those shown in Figure \ref{fig:density_case3to6} for $|\psi_1|^2$ (only differ from the `color', i.e., the more blurred humps imply the centers of the other five component, while the lighter one shows the one of the current component). While after complete state aggregation (around $t=0.4$,  which corresponds  to the moment the  center of mass $x_c^j$ ($j=1,\cdots,6$)  meet together in Figure \ref{fig:quant_case3to6}), all $\psi_j(x,t)$ (hence also for all density profiles)  will converge to the same function,  whose density changes periodically in time (as shown in columns 4--6 in Figure \ref{fig:density_case3to6}, which also indicate the periodic dynamics for the center of mass that illustrated in Figure \ref{fig:quant_case3to6}).
In addition, before complete state aggregation, although  numerical schemes cannot conserve  the cross-ratio like quantities $\mathcal{R}_{ijkl}(t) (1\le i, j, k, l \le 6)$  in discretized level, the difference of  those quantities from their initial ones are still small (cf. Figure \ref{fig:dyn_R_ijkl}).}
}
\end{example}

\begin{figure}[h]
\centering{
\mbox{(a)
\subfigure{\includegraphics[width=0.23\textwidth]{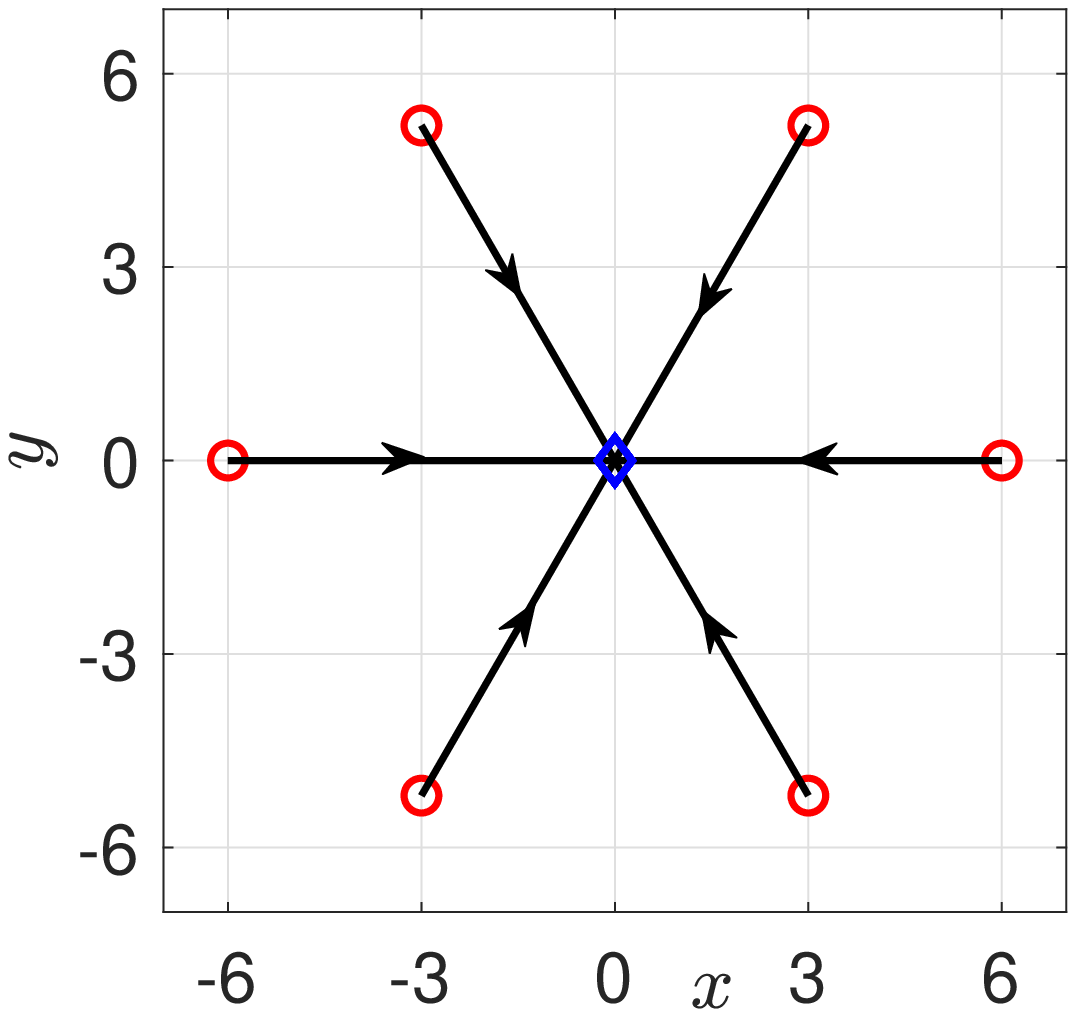}}
\subfigure{\includegraphics[width=0.23\textwidth]{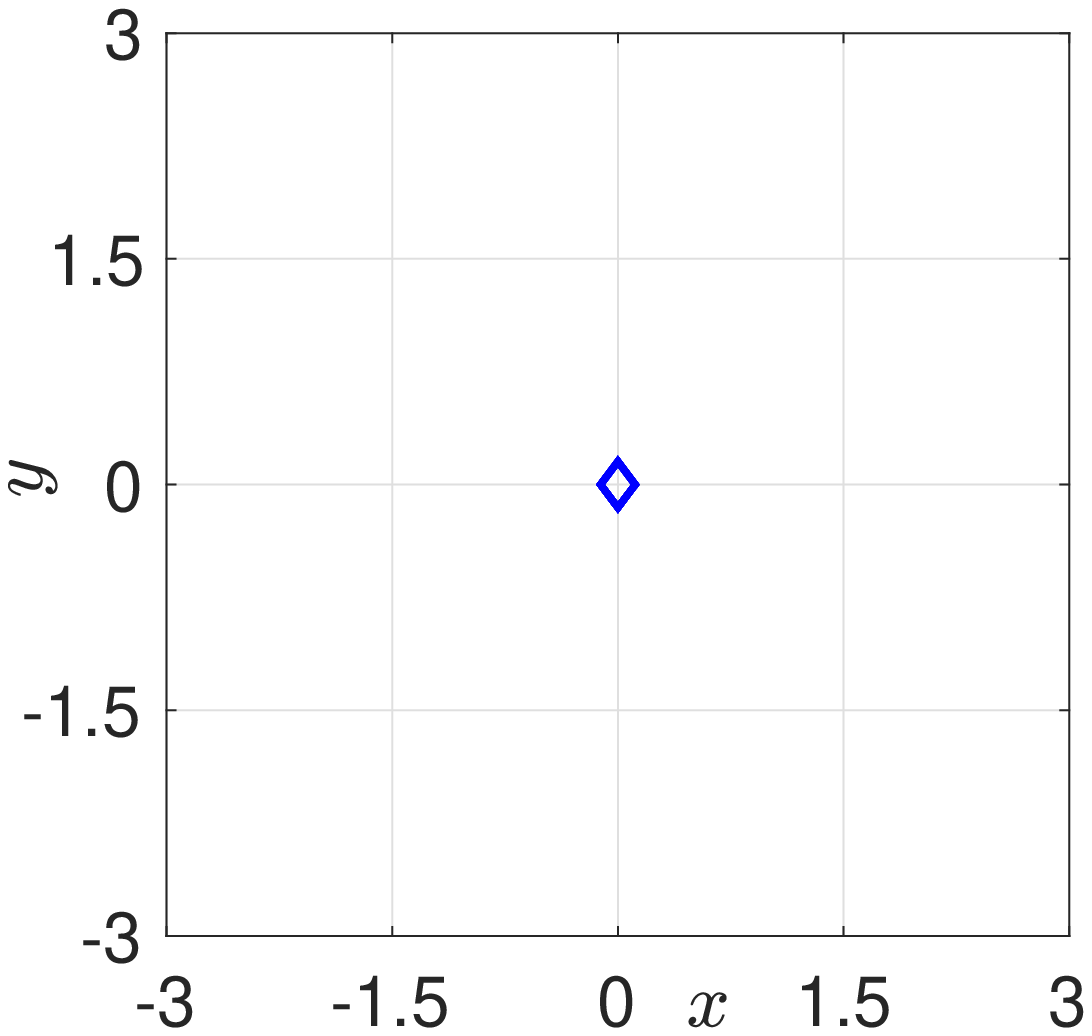}}
\subfigure{\includegraphics[width=0.45\textwidth]{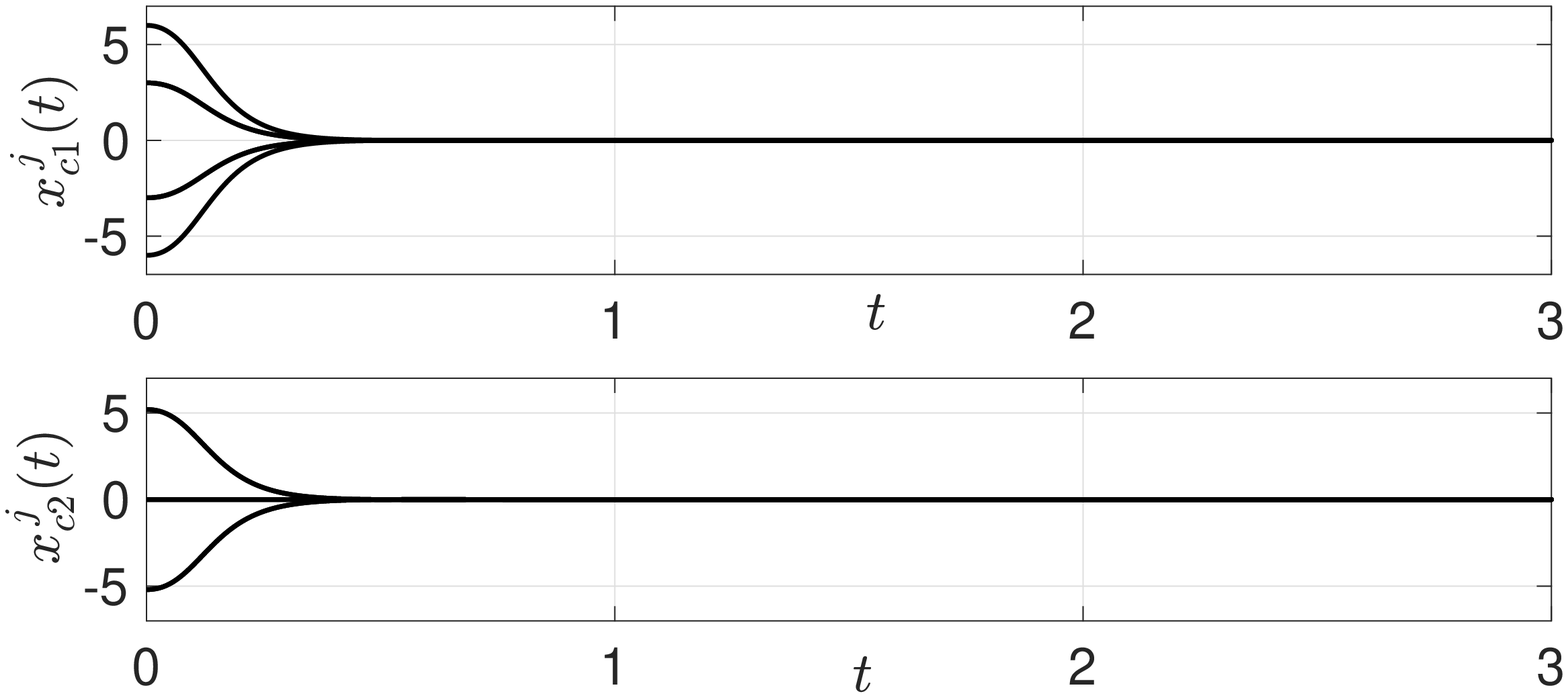}}
}
\mbox{(b)
\subfigure{\includegraphics[width=0.23\textwidth]{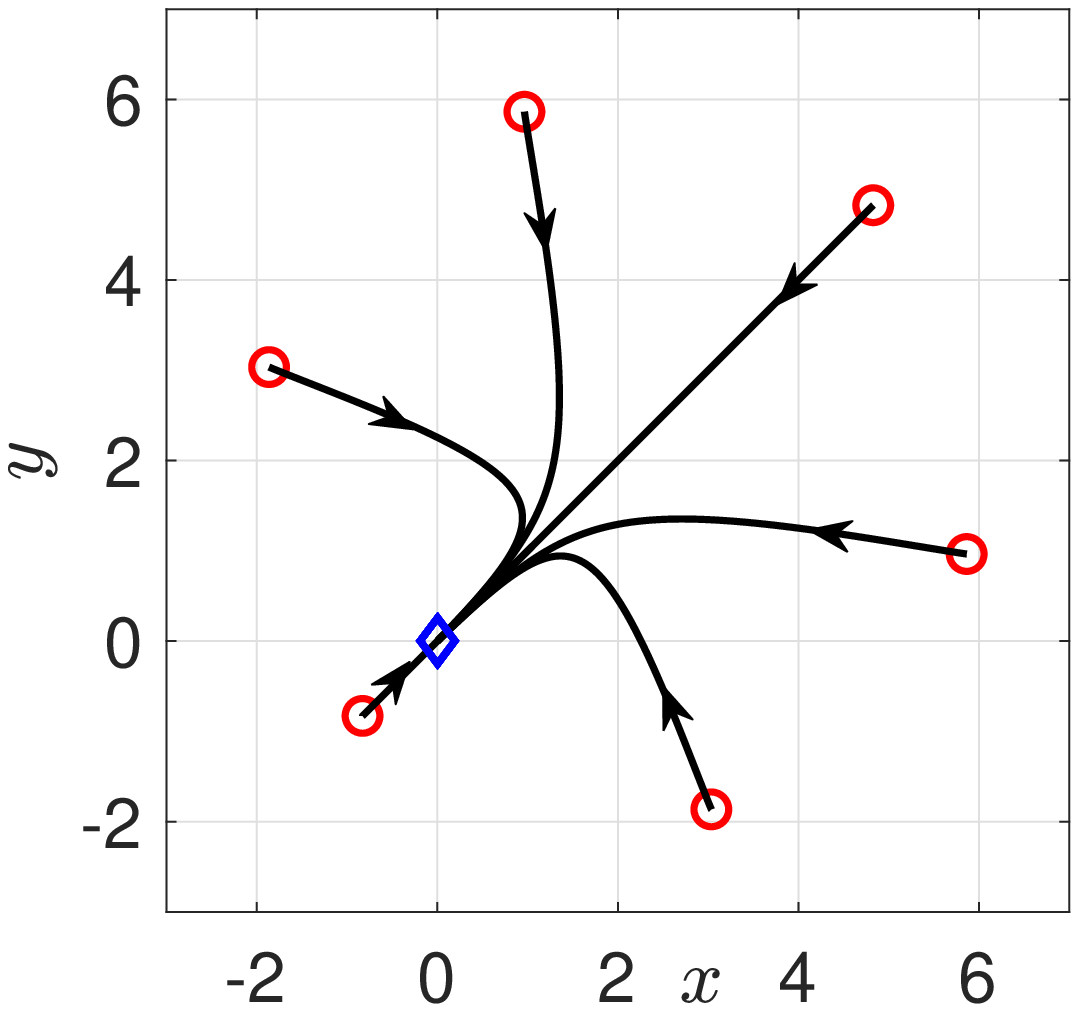}}
\subfigure{\includegraphics[width=0.23\textwidth]{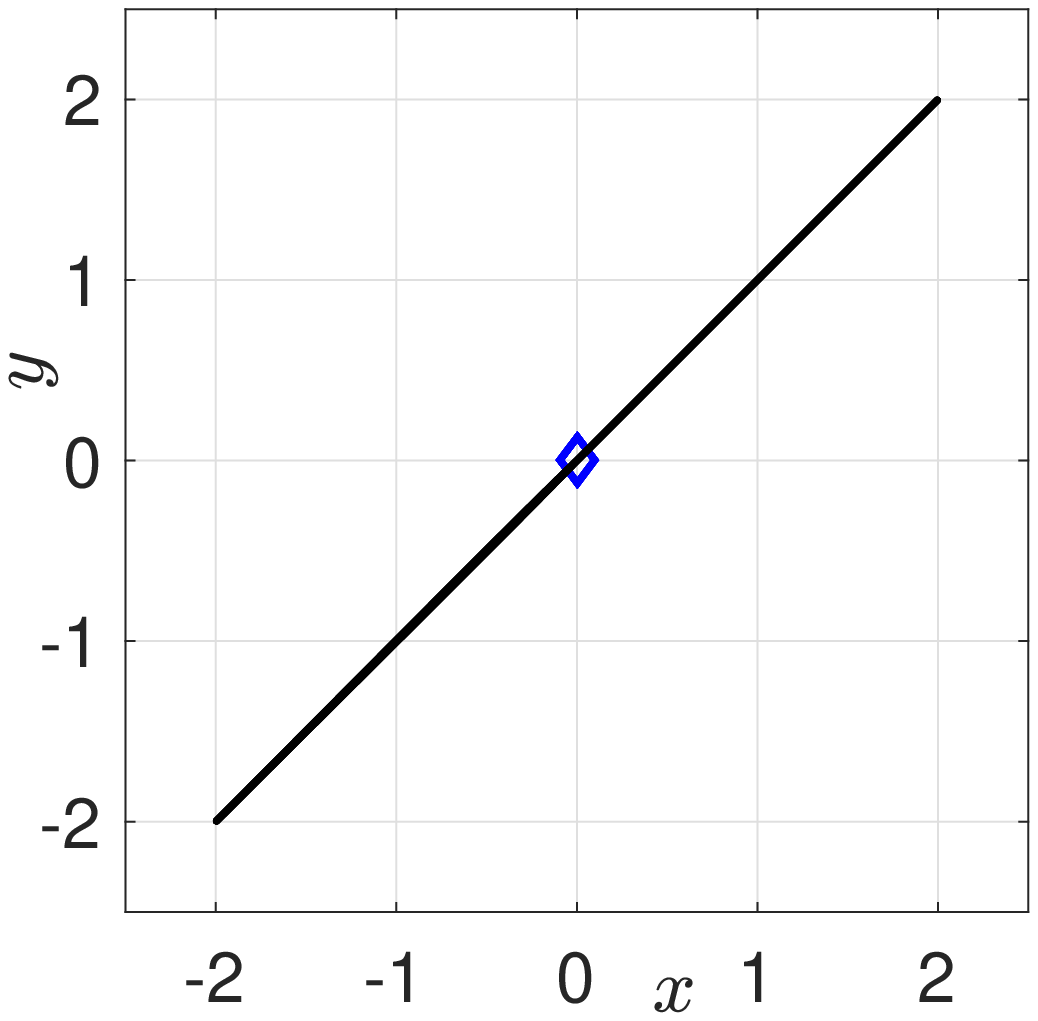}}
\subfigure{\includegraphics[width=0.45\textwidth]{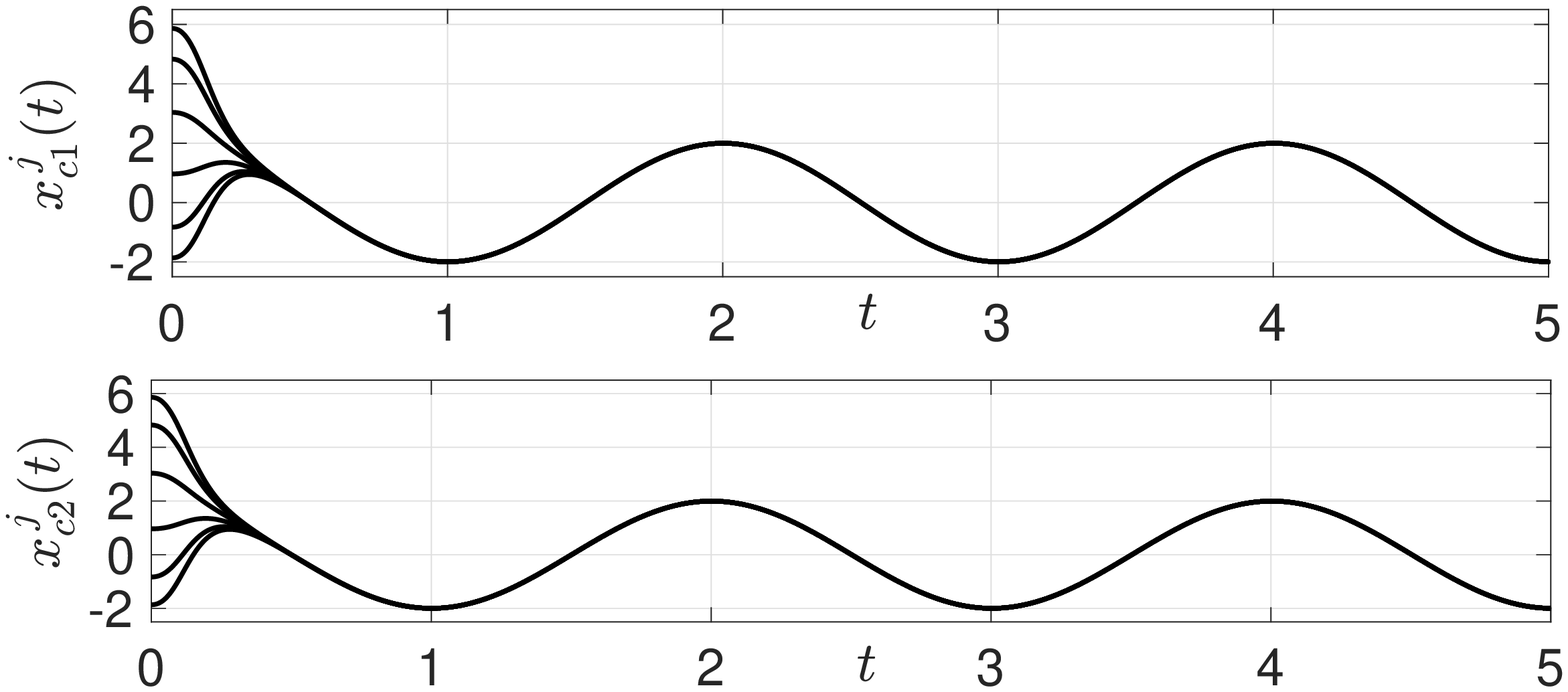}}
}
\mbox{(c)
\subfigure{\includegraphics[width=0.23\textwidth]{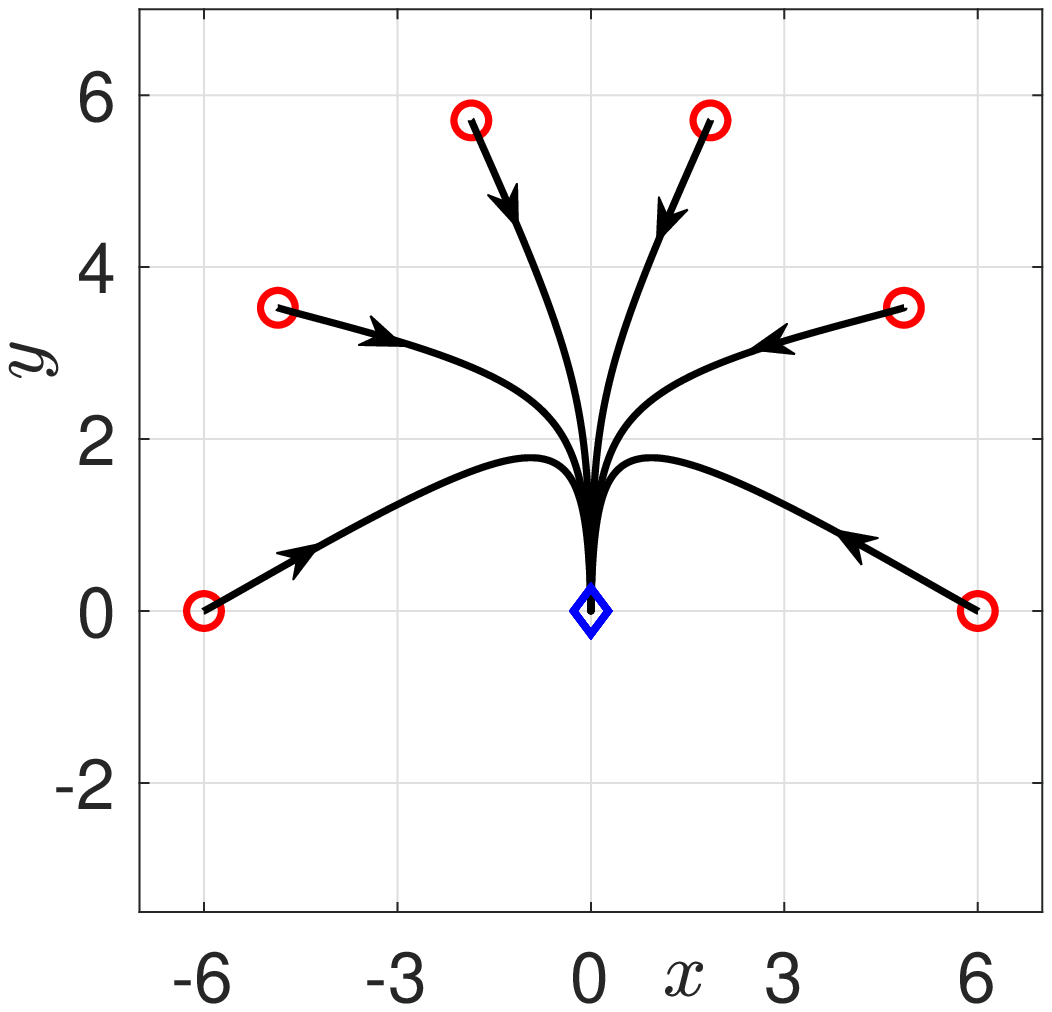}}
\subfigure{\includegraphics[width=0.23\textwidth]{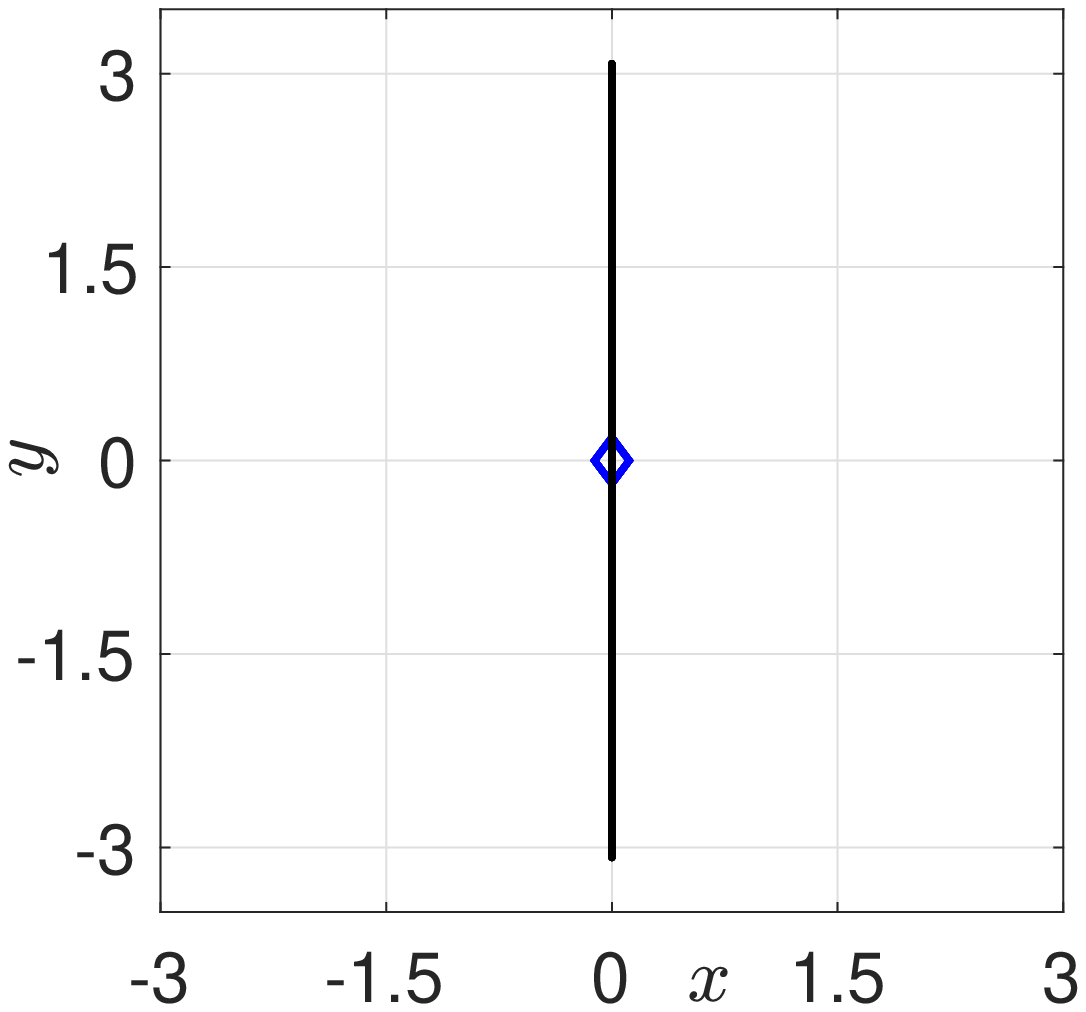}}
\subfigure{\includegraphics[width=0.45\textwidth]{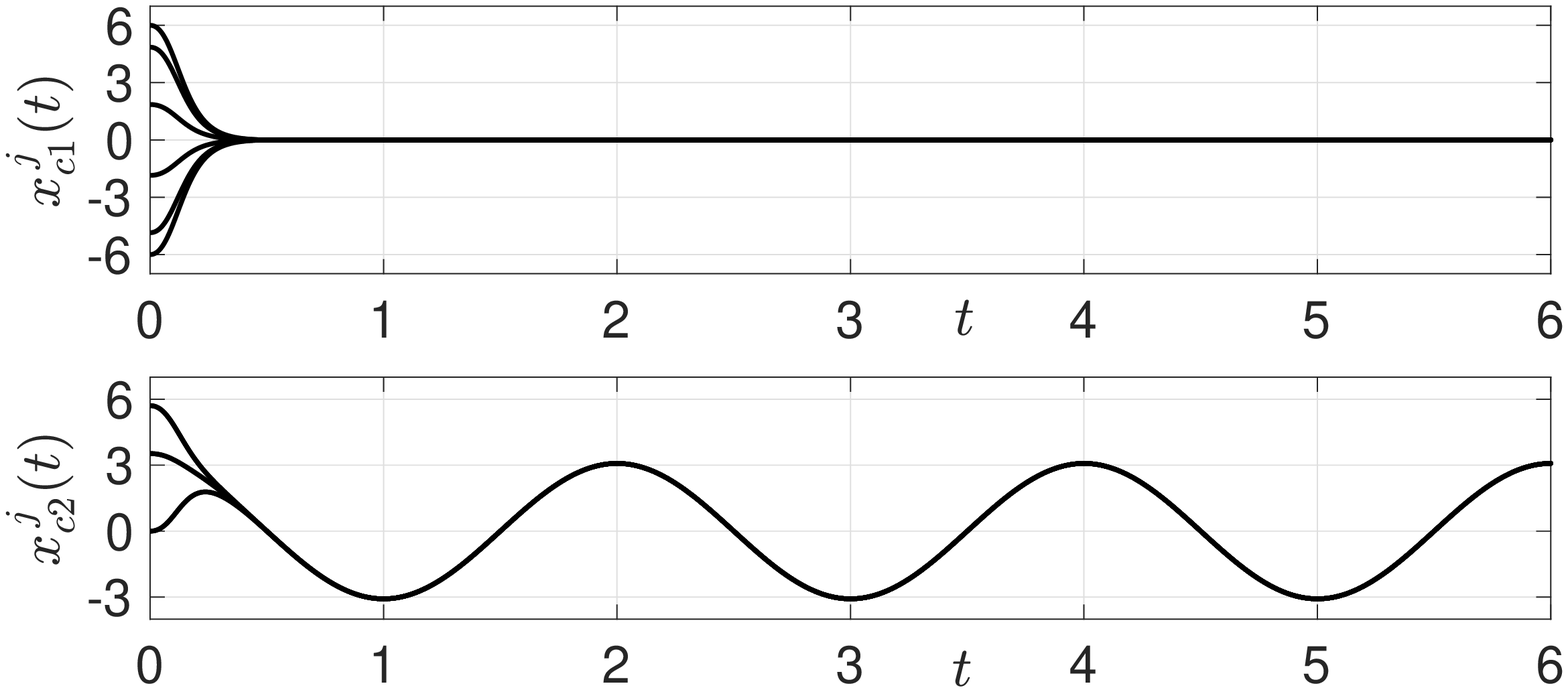}}
}
\mbox{(d)
\subfigure{\includegraphics[width=0.23\textwidth]{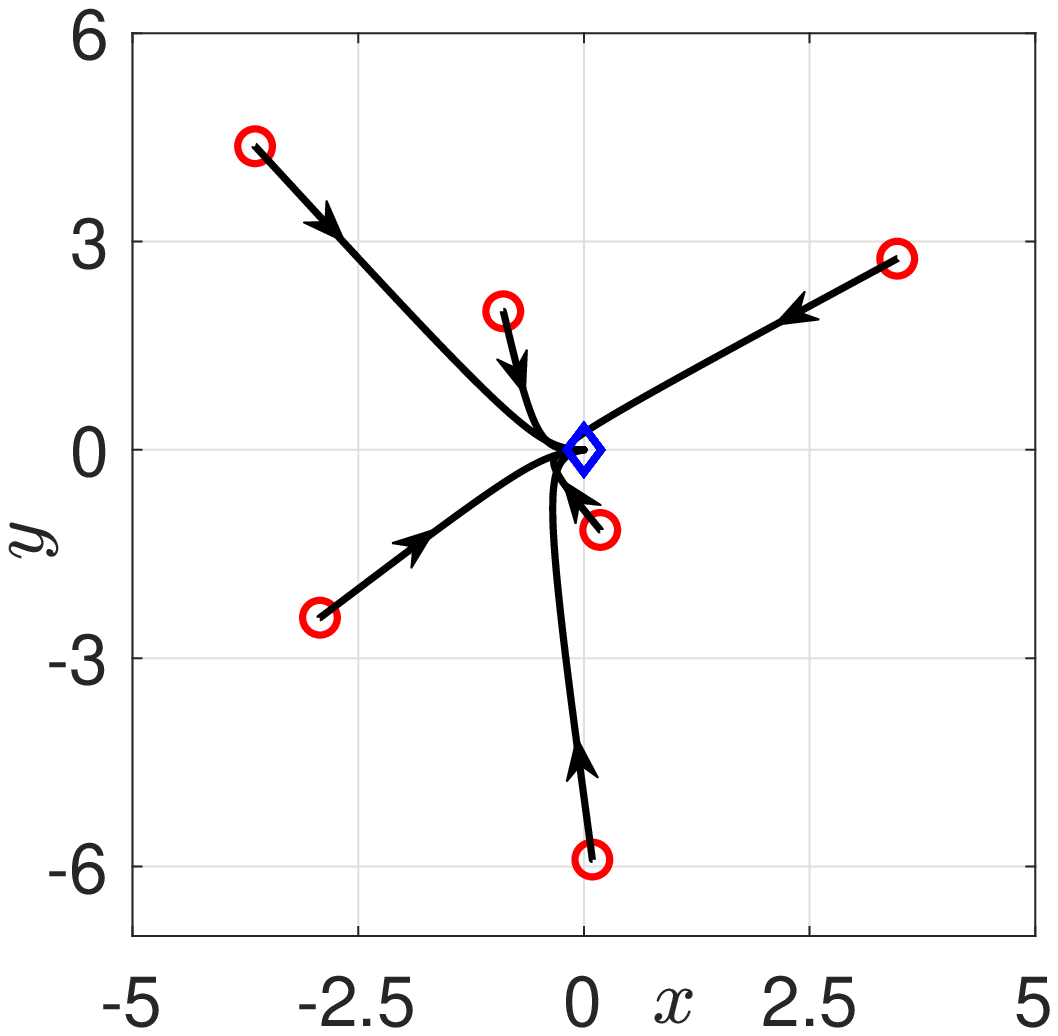}}
\subfigure{\includegraphics[width=0.23\textwidth]{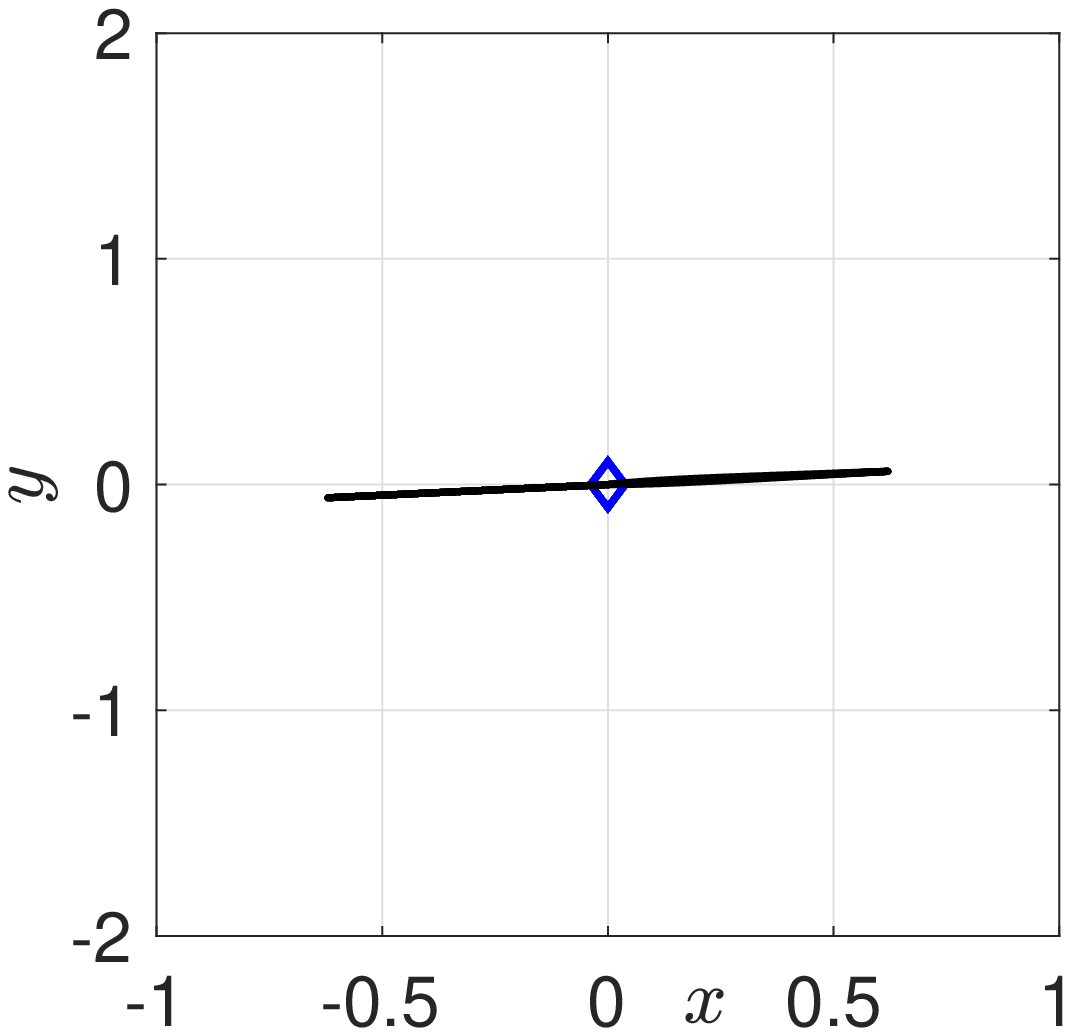}}
\subfigure{\includegraphics[width=0.45\textwidth]{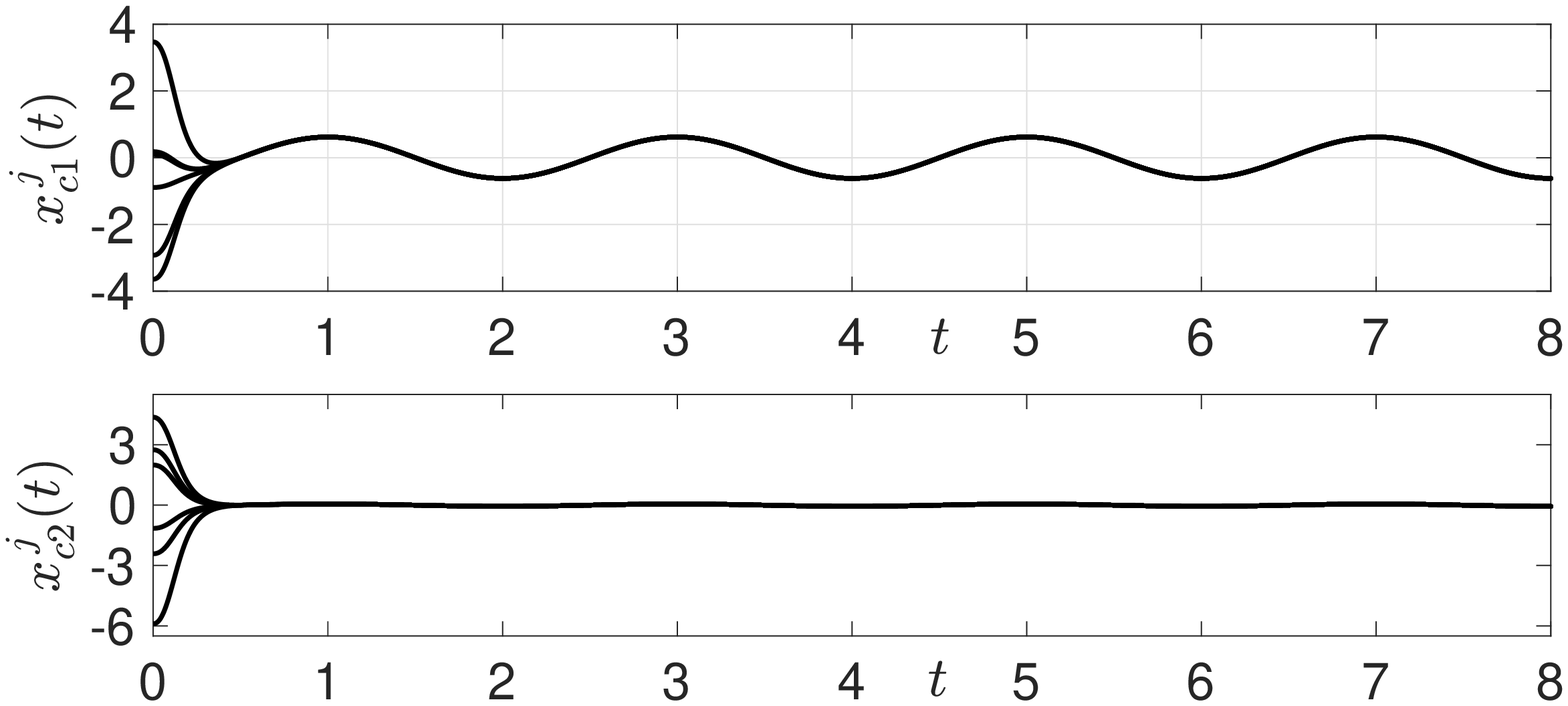}}
}
}
\vspace{-0.3cm}
\caption{First two columns:  trajectory of  center of mass $x_c^j(t)$ in $t\in[0, t_c]$ and  $t\in[t_c, 10]$ ($t_c=1.5$ for first row while 0.5 for the others).
The third column:  time evolution of $x_{c1}^j(t)$ and  $x_{c2}^j(t))$ (right).  $\circ$ denotes location of $x_c^j(0)$,
while $\Diamond$ denotes the one of $x_c^j(t_c)$.}
 \label{fig:quant_case3to6}
\end{figure}

\begin{figure}[h]
\centering{
\subfigure{\includegraphics[width=0.23\textwidth]{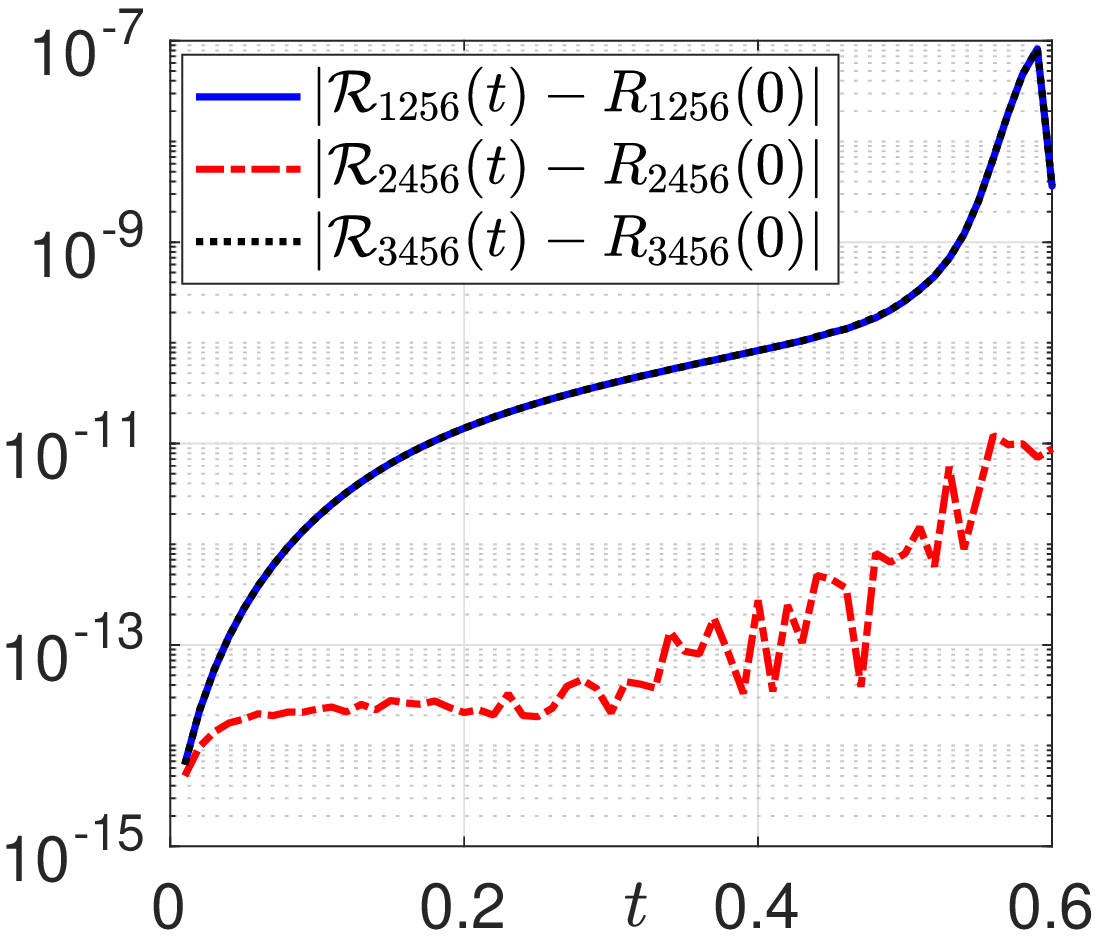}}
\subfigure{\includegraphics[width=0.23\textwidth]{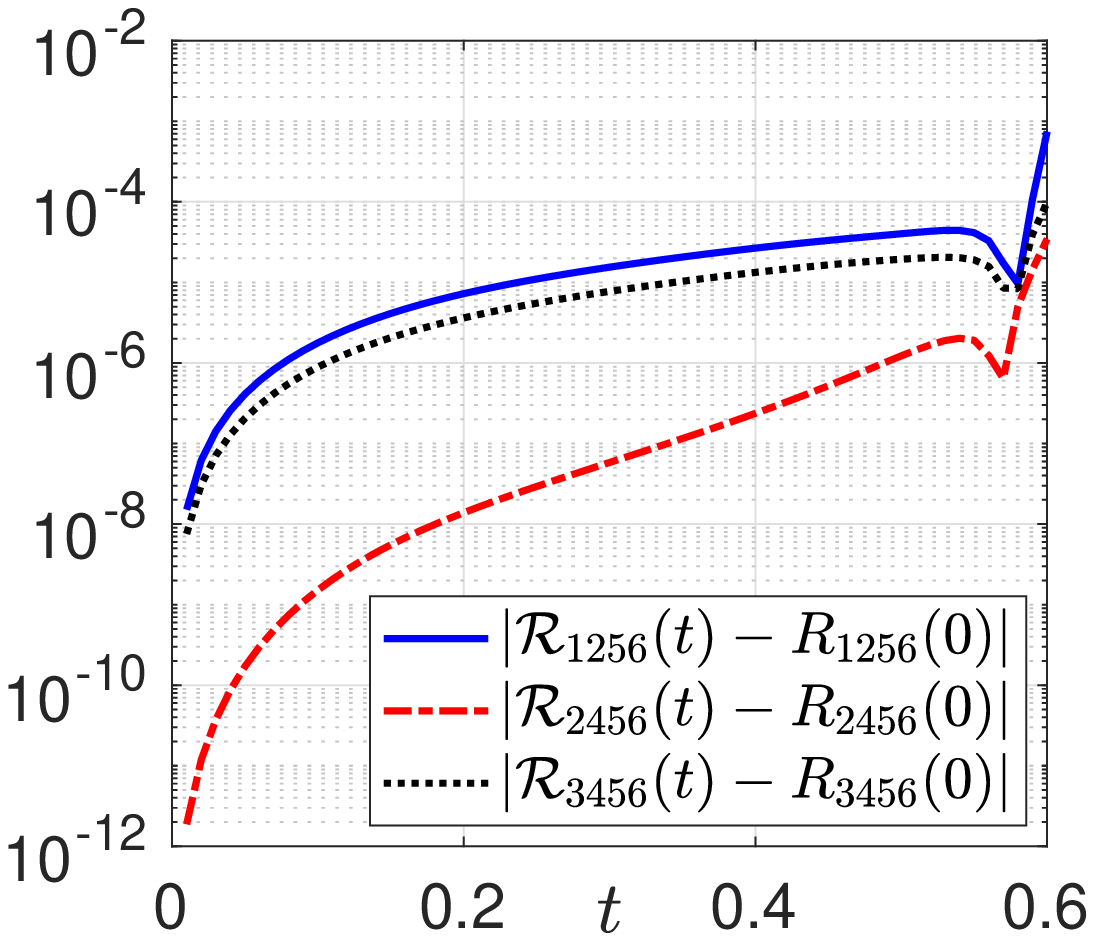}}
\subfigure{\includegraphics[width=0.23\textwidth]{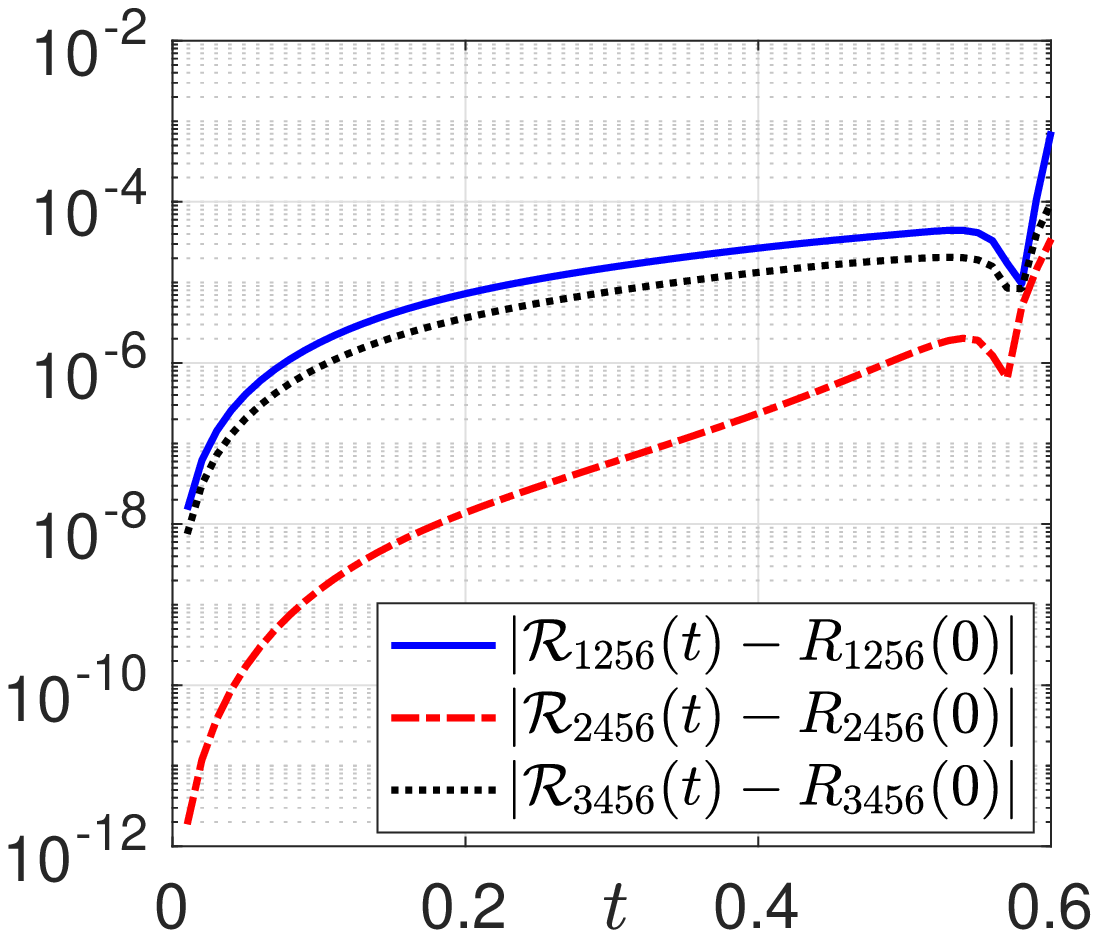}}
\subfigure{\includegraphics[width=0.23\textwidth]{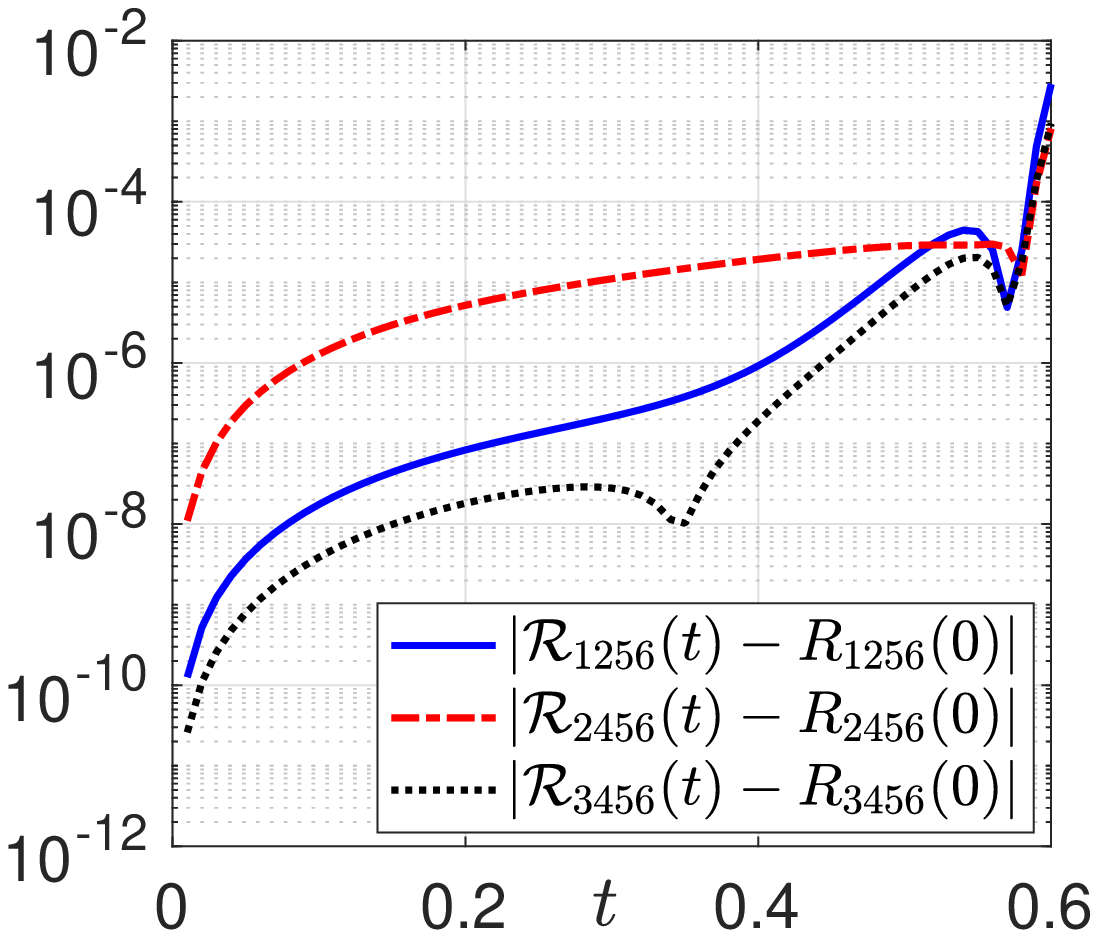}}
}
\vspace{-0.3cm}
\caption{Time evolution of $|\mathcal{R}_{1256}(t)-\mathcal{R}_{1256}(0)|$, $|\mathcal{R}_{2456}(t)-\mathcal{R}_{2456}(0)|$ and  $|\mathcal{R}_{3456}(t)-\mathcal{R}_{3456}(0)|$ for {\bf Case} 3-6 (Left to right). }
 \label{fig:dyn_R_ijkl}
\end{figure}

\begin{figure}[h]
\centering{
\mbox{(a)
\subfigure{\includegraphics[width=0.15\textwidth]{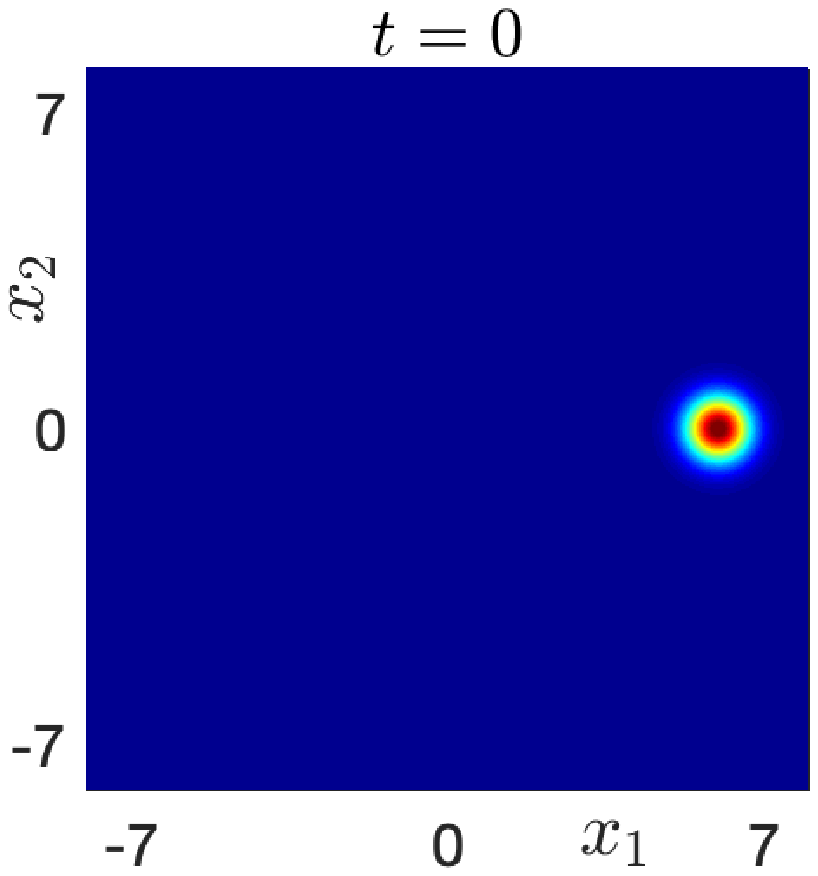}}
\subfigure{\includegraphics[width=0.15\textwidth]{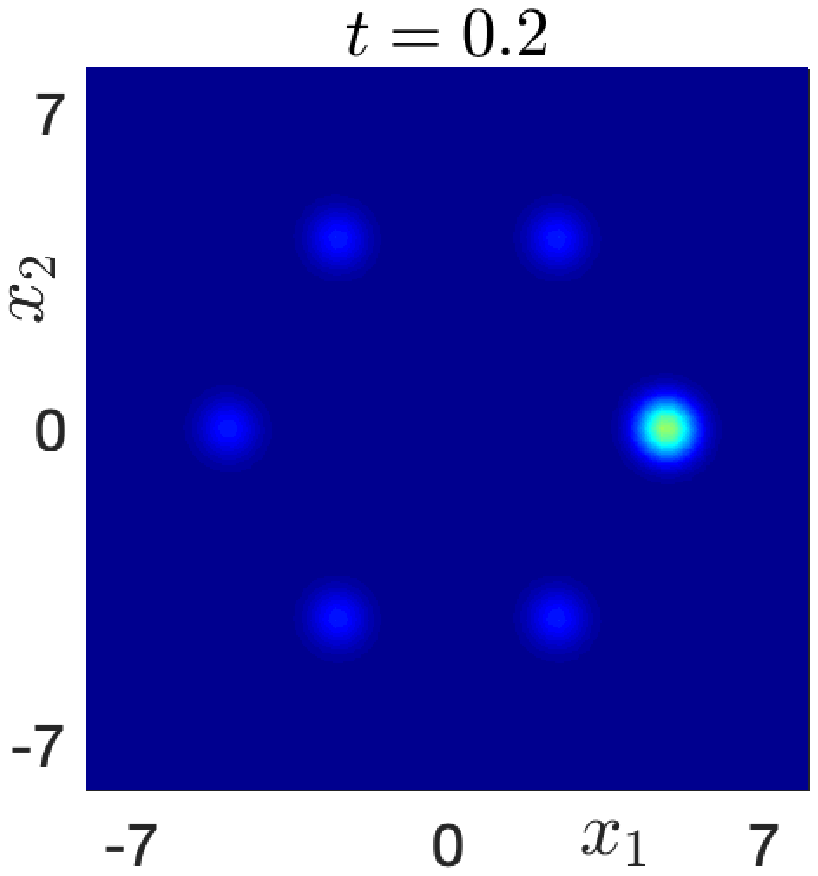}}
\subfigure{\includegraphics[width=0.15\textwidth]{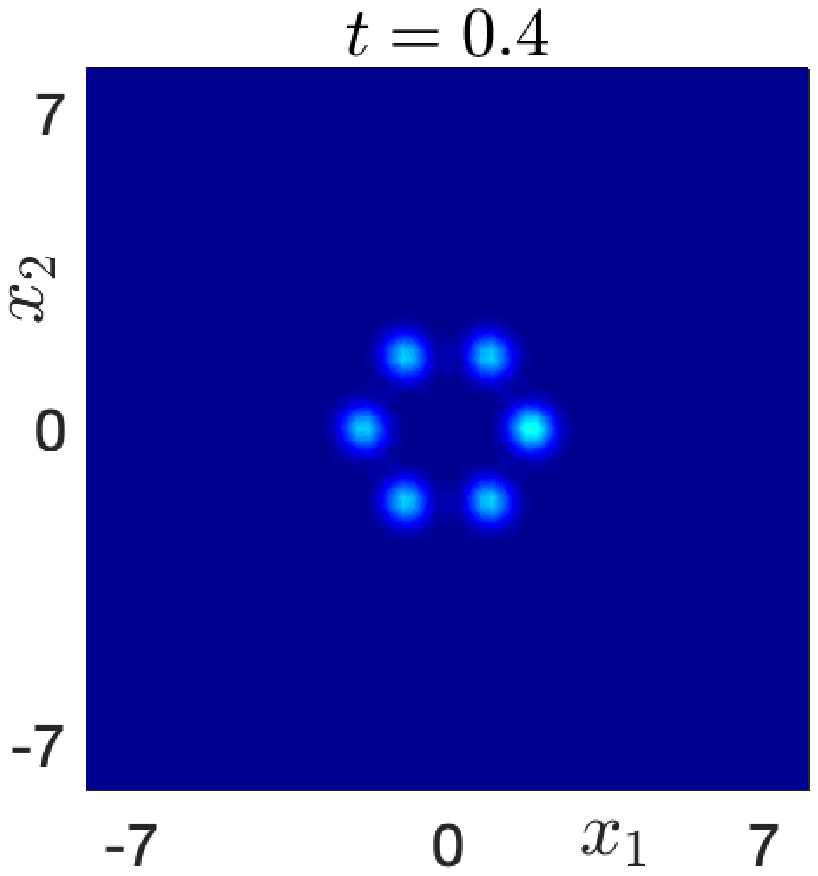}}
\subfigure{\includegraphics[width=0.15\textwidth]{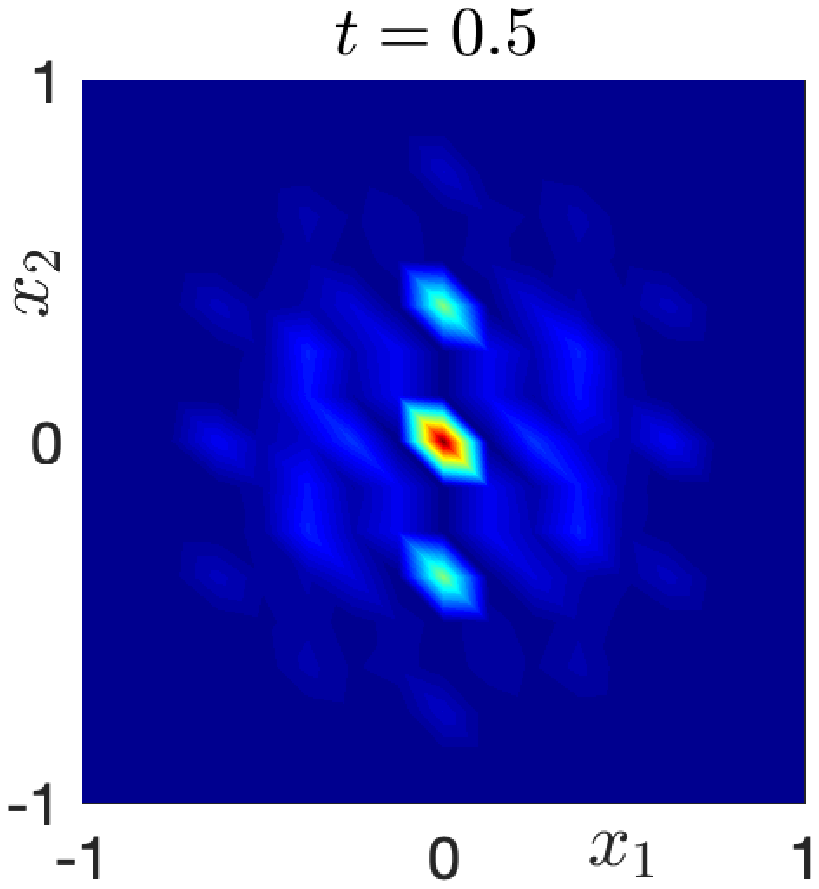}}
\subfigure{\includegraphics[width=0.15\textwidth]{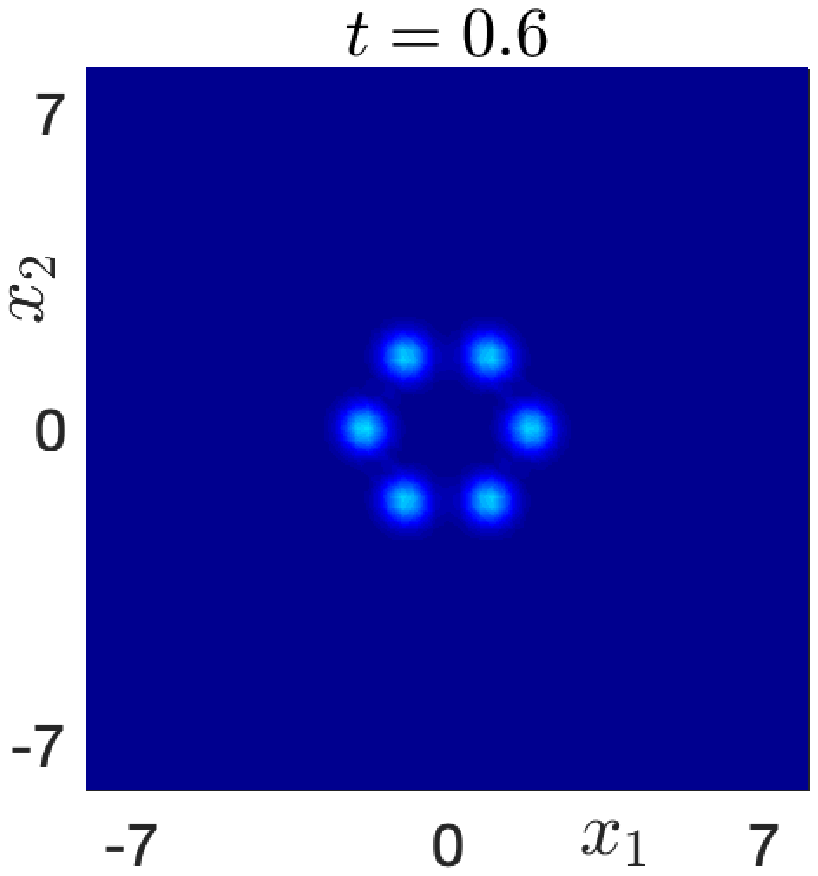}}
\subfigure{\includegraphics[width=0.15\textwidth]{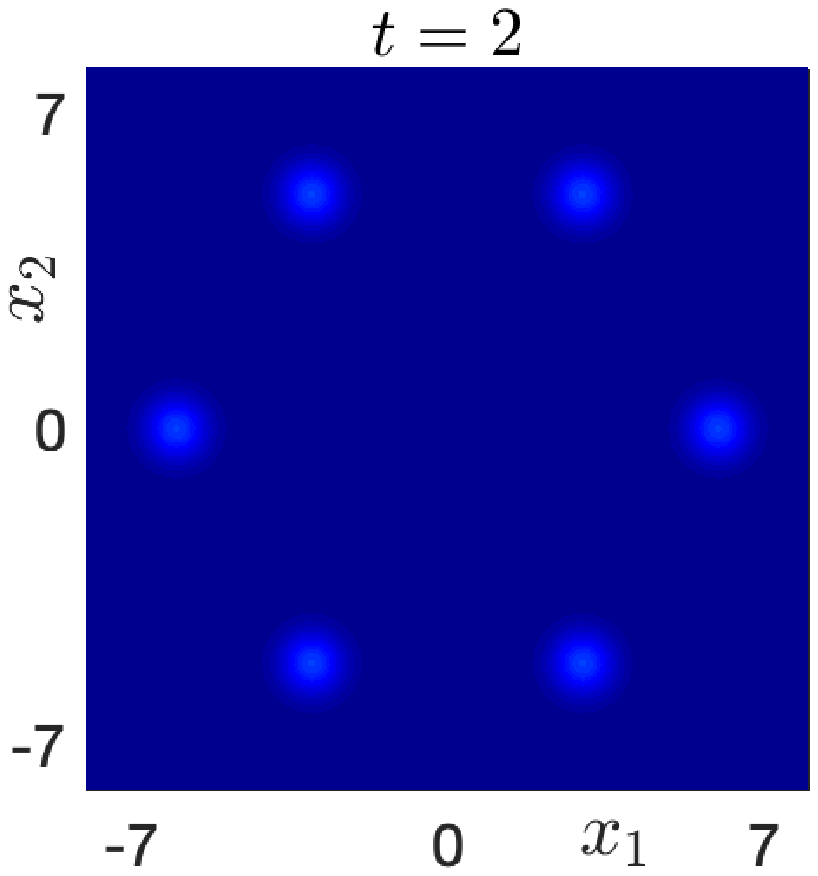}}
}
\mbox{(b)
\subfigure{\includegraphics[width=0.15\textwidth]{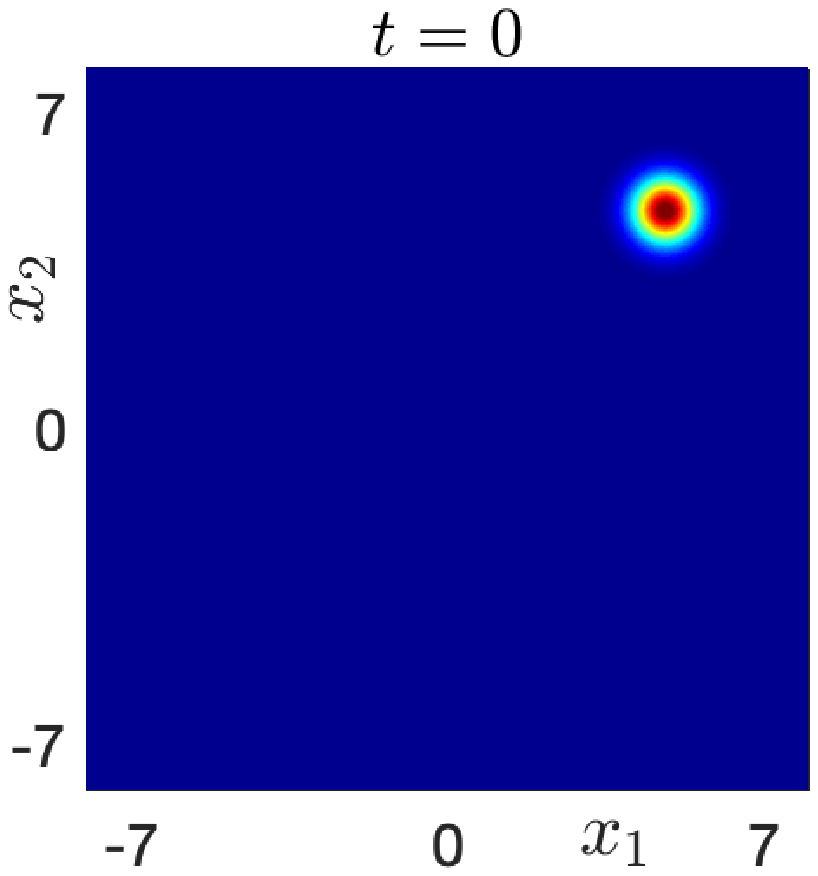}}
\subfigure{\includegraphics[width=0.15\textwidth]{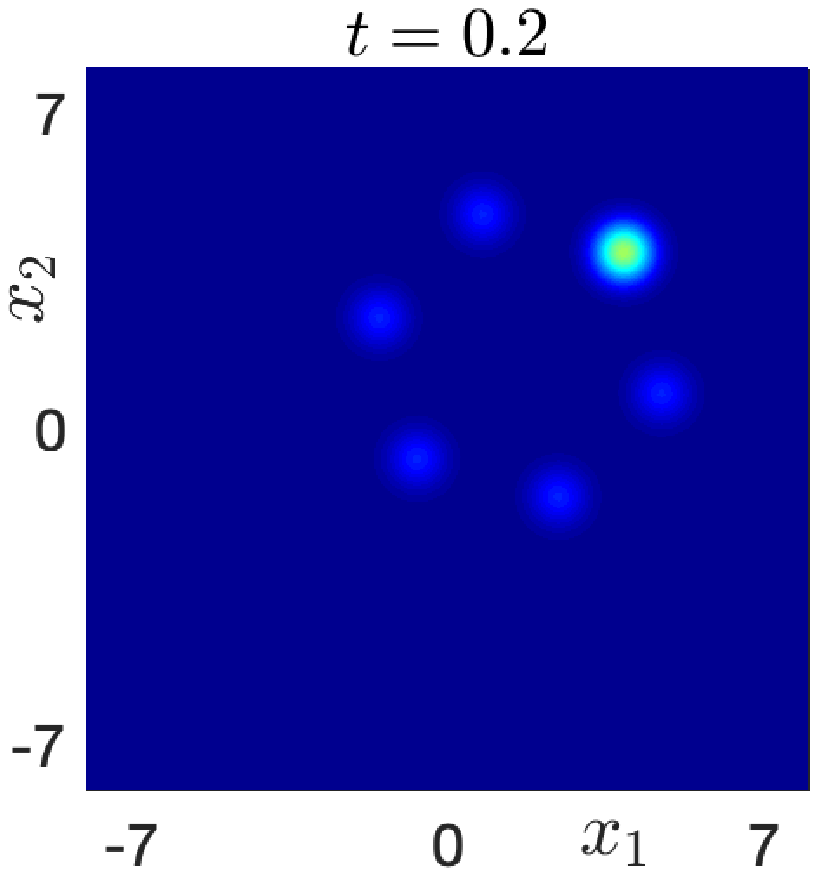}}
\subfigure{\includegraphics[width=0.15\textwidth]{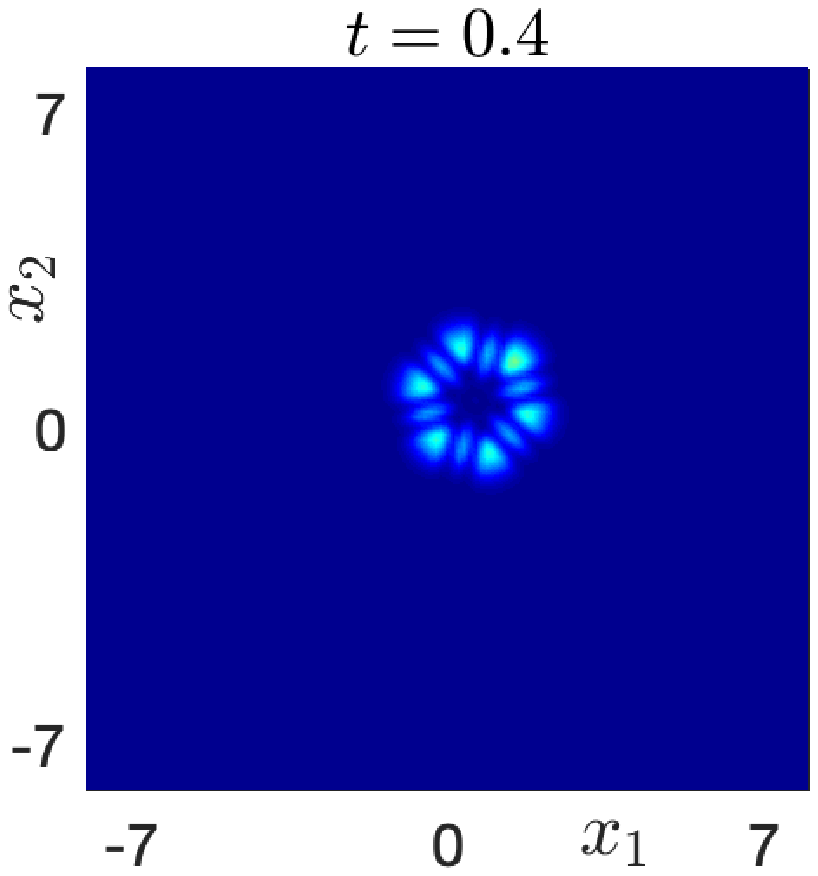}}
\subfigure{\includegraphics[width=0.15\textwidth]{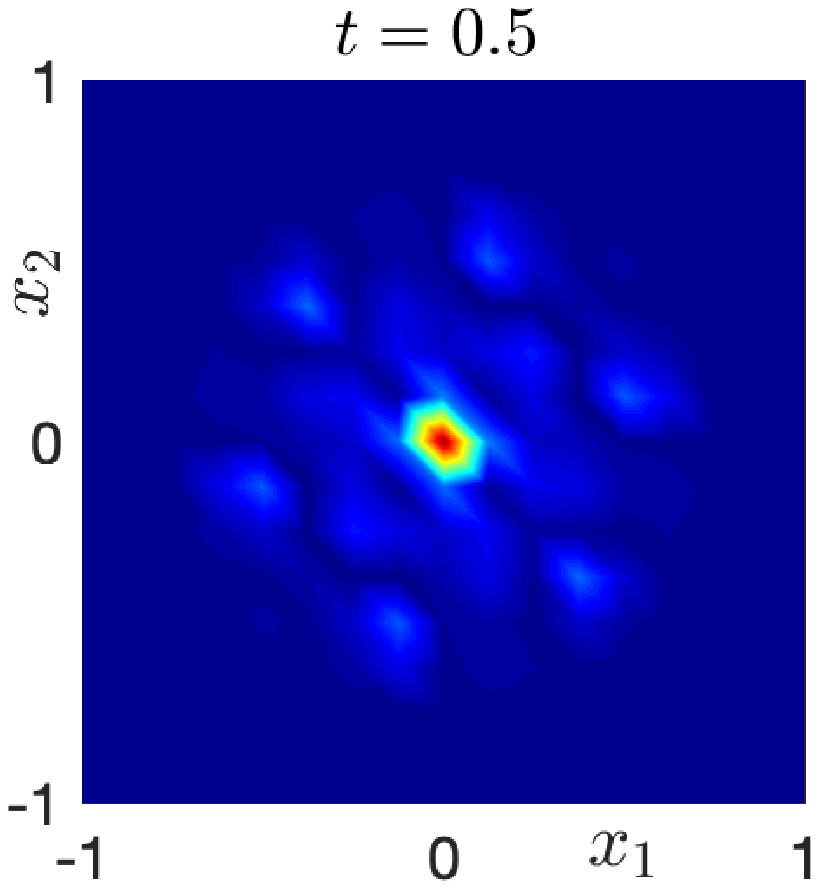}}
\subfigure{\includegraphics[width=0.15\textwidth]{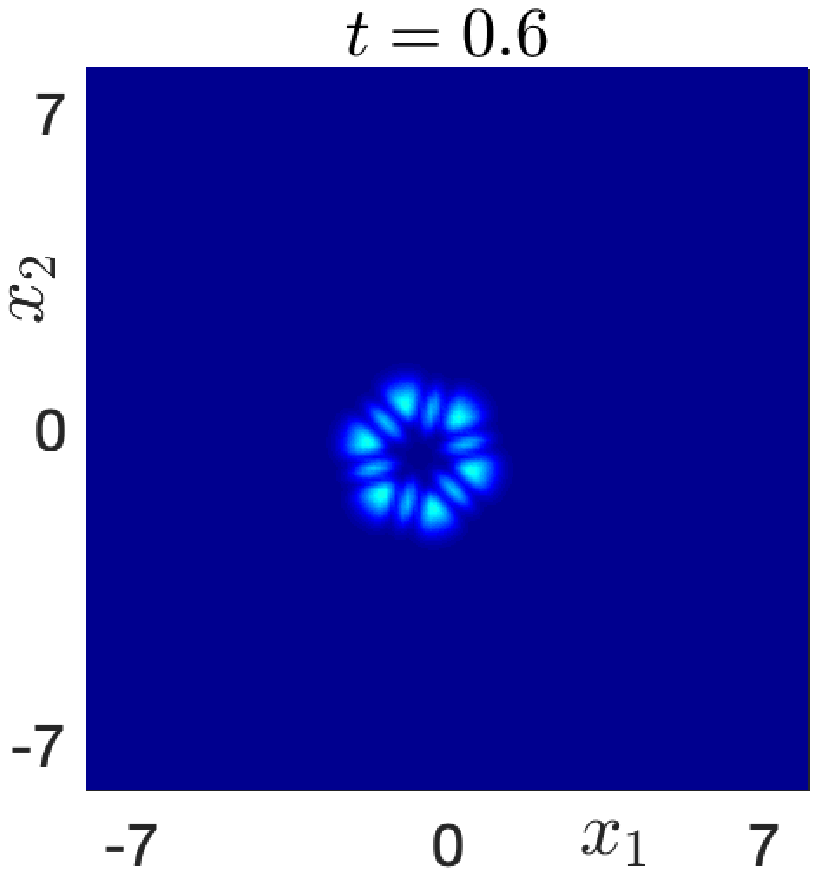}}
\subfigure{\includegraphics[width=0.15\textwidth]{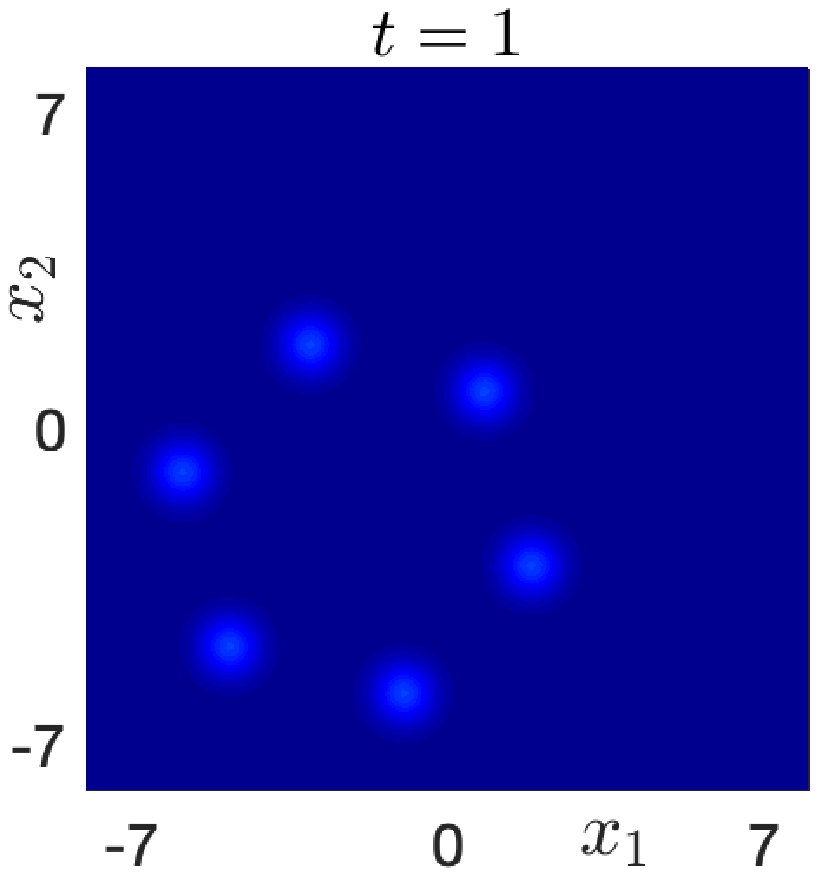}}
}
\mbox{(c)
\subfigure{\includegraphics[width=0.15\textwidth]{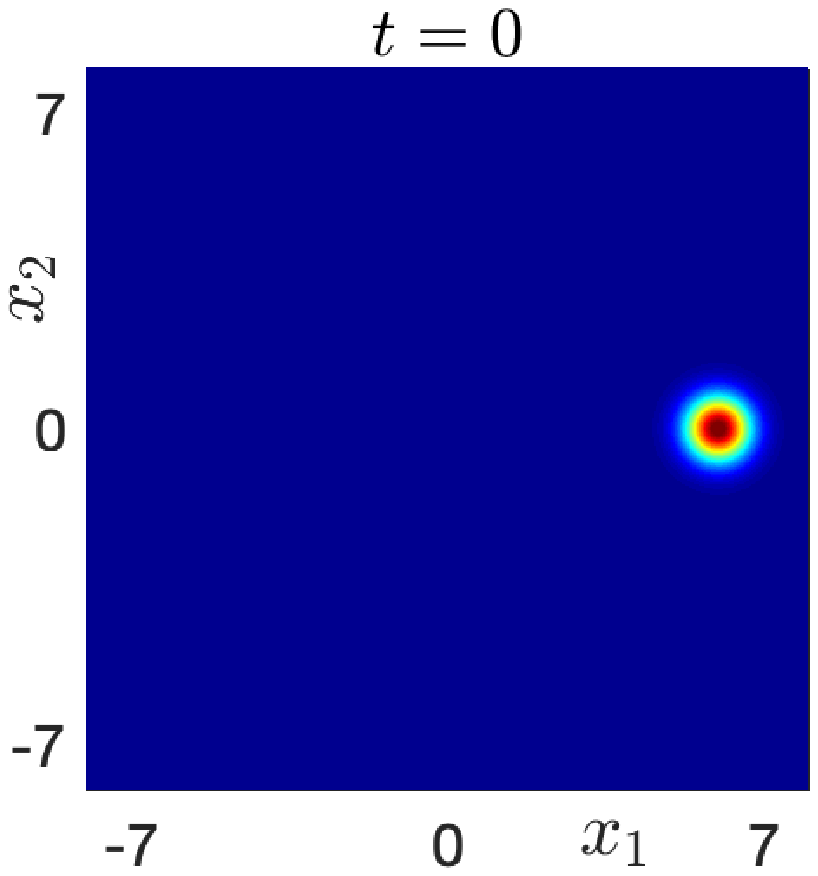}}
\subfigure{\includegraphics[width=0.15\textwidth]{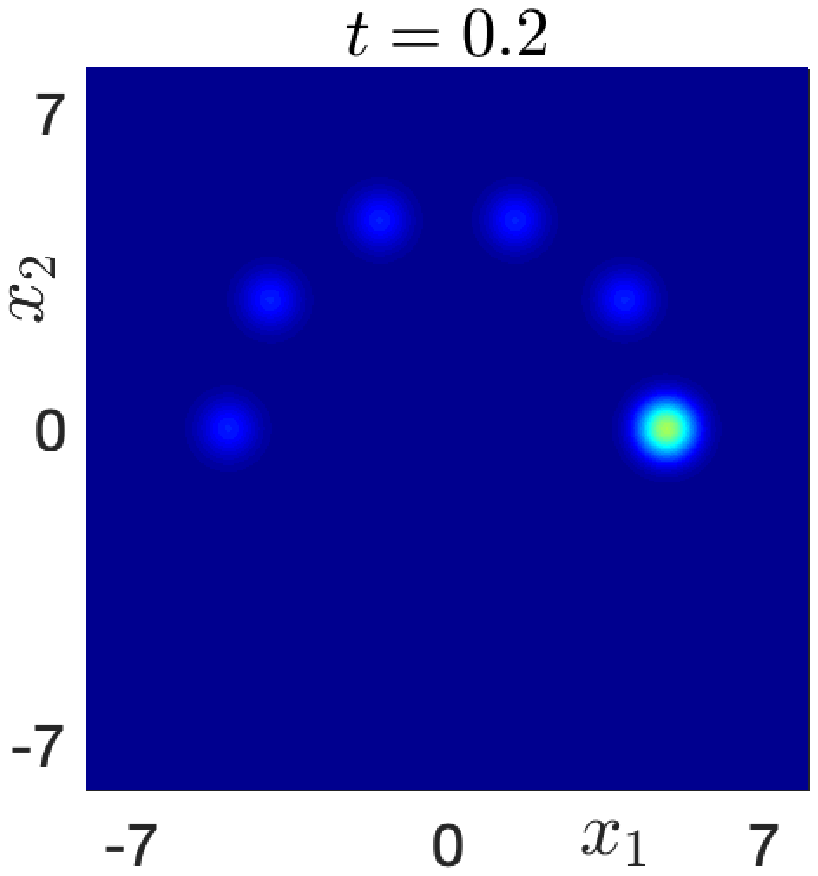}}
\subfigure{\includegraphics[width=0.15\textwidth]{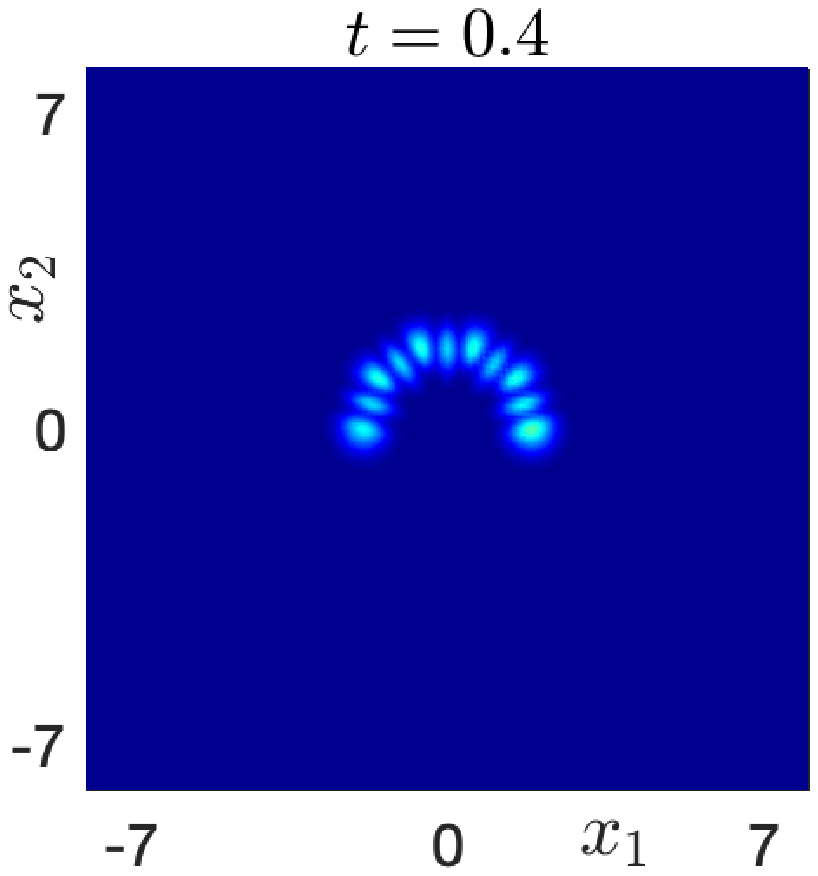}}
\subfigure{\includegraphics[width=0.15\textwidth]{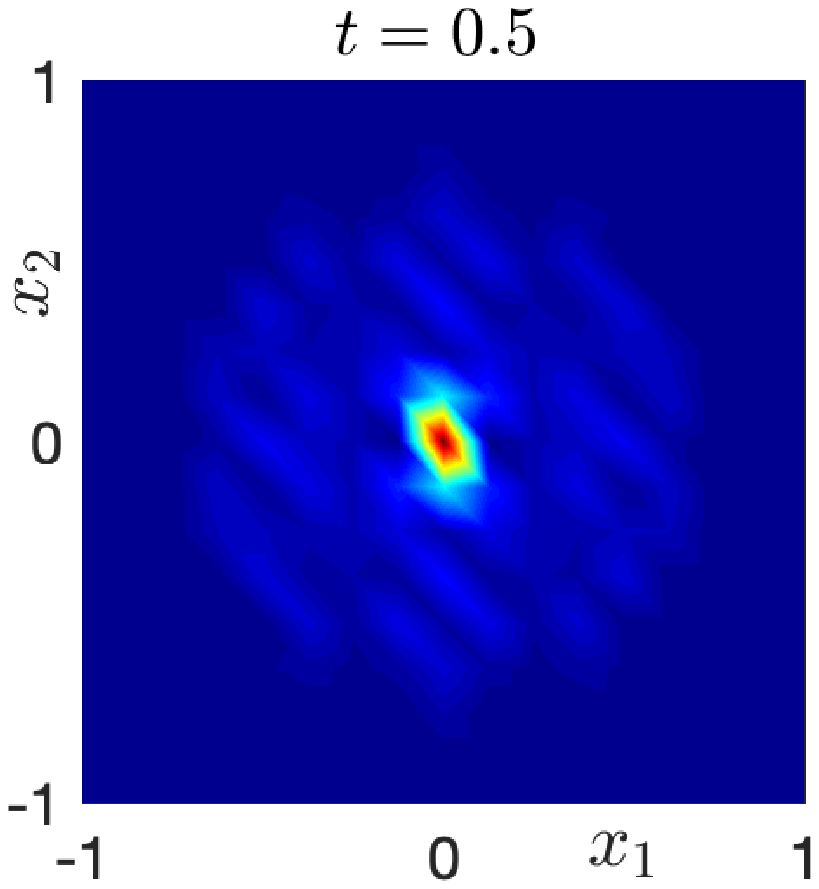}}
\subfigure{\includegraphics[width=0.15\textwidth]{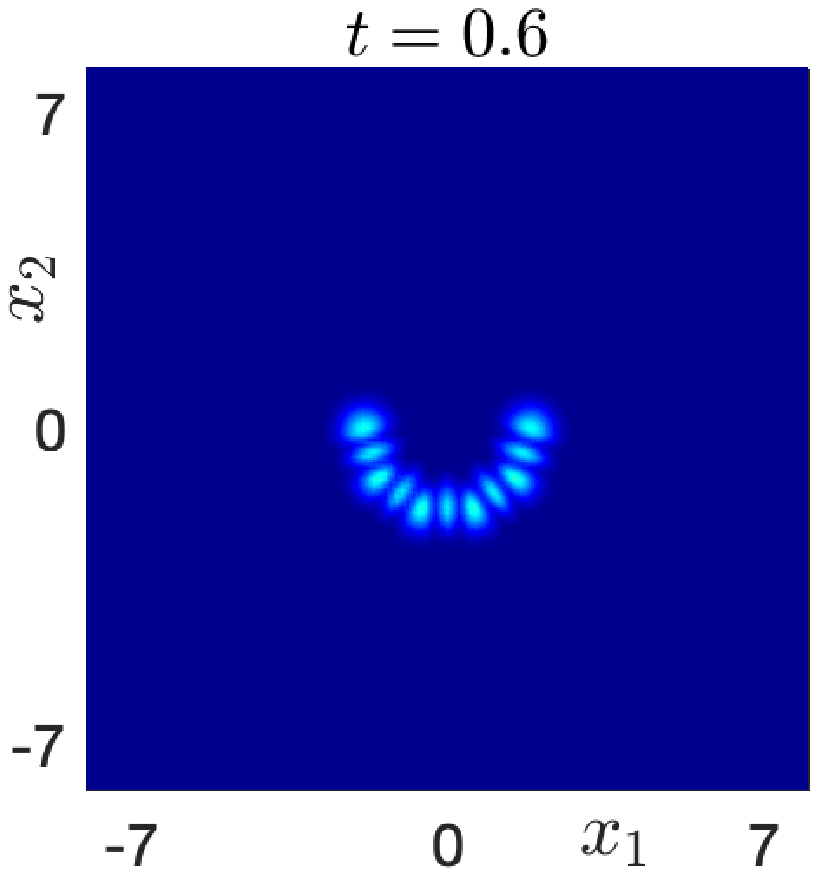}}
\subfigure{\includegraphics[width=0.15\textwidth]{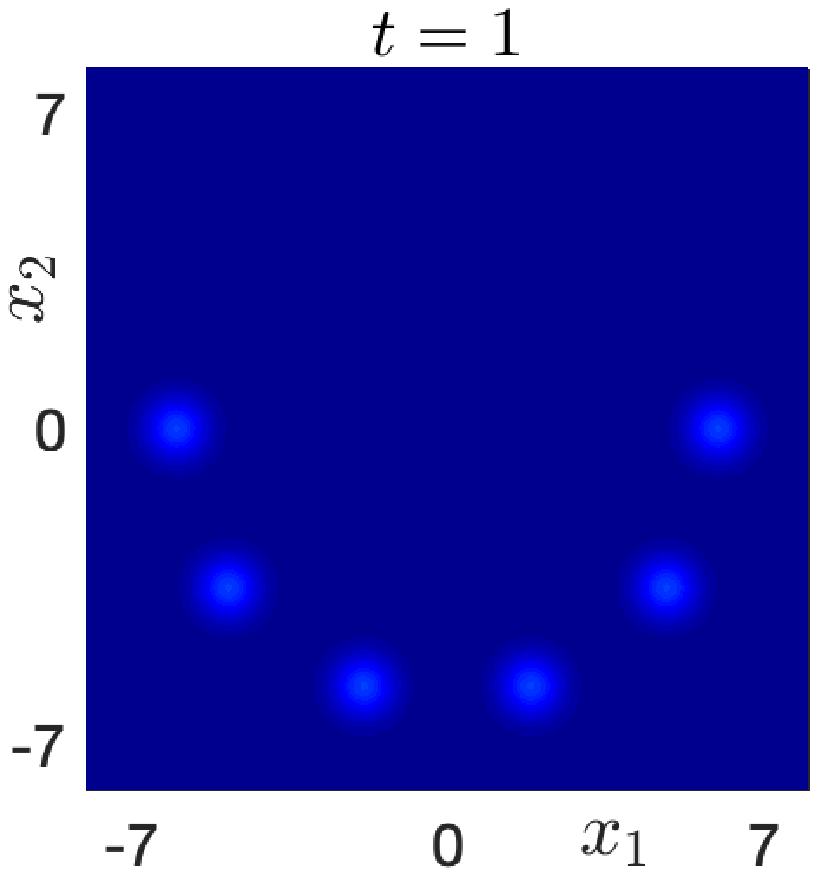}}
}
\mbox{(d)
\subfigure{\includegraphics[width=0.15\textwidth]{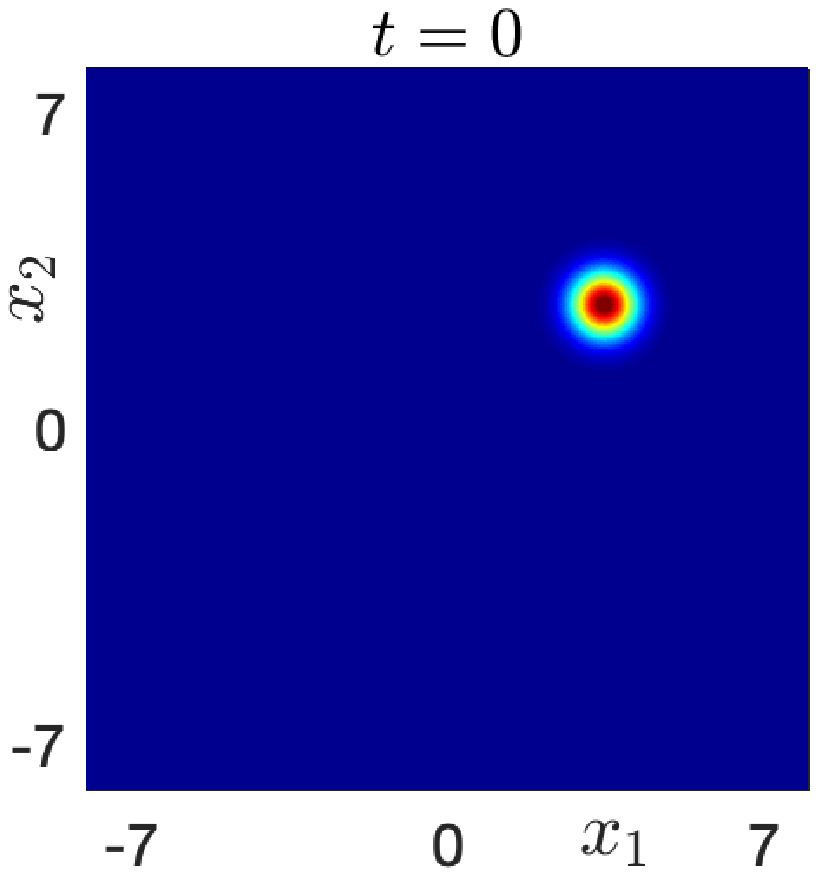}}
\subfigure{\includegraphics[width=0.15\textwidth]{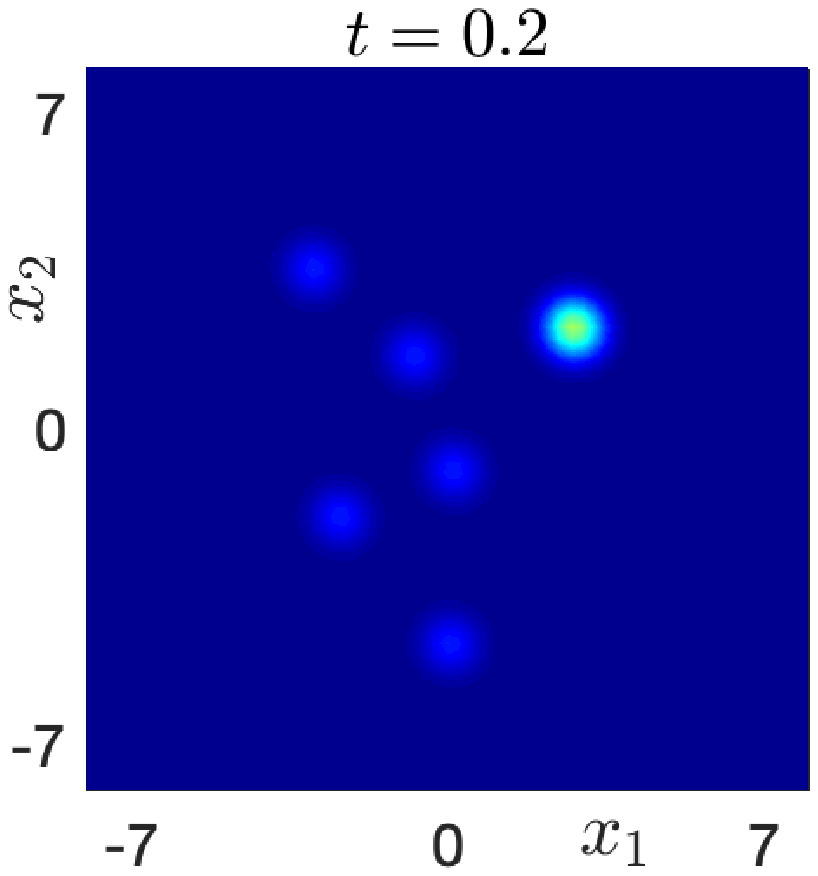}}
\subfigure{\includegraphics[width=0.15\textwidth]{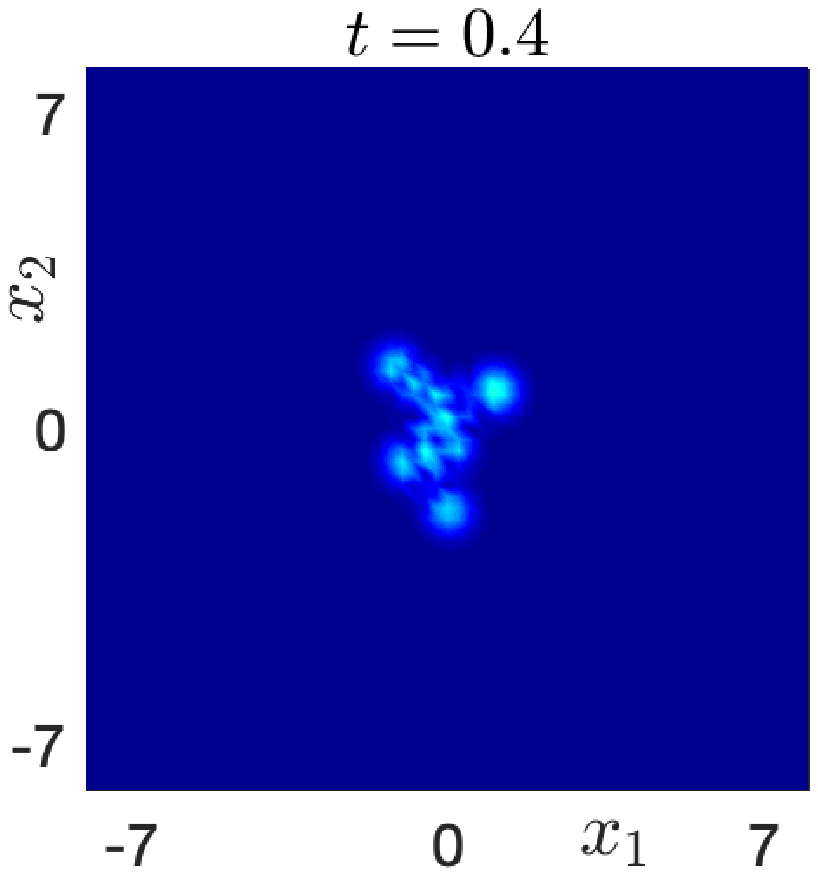}}
\subfigure{\includegraphics[width=0.15\textwidth]{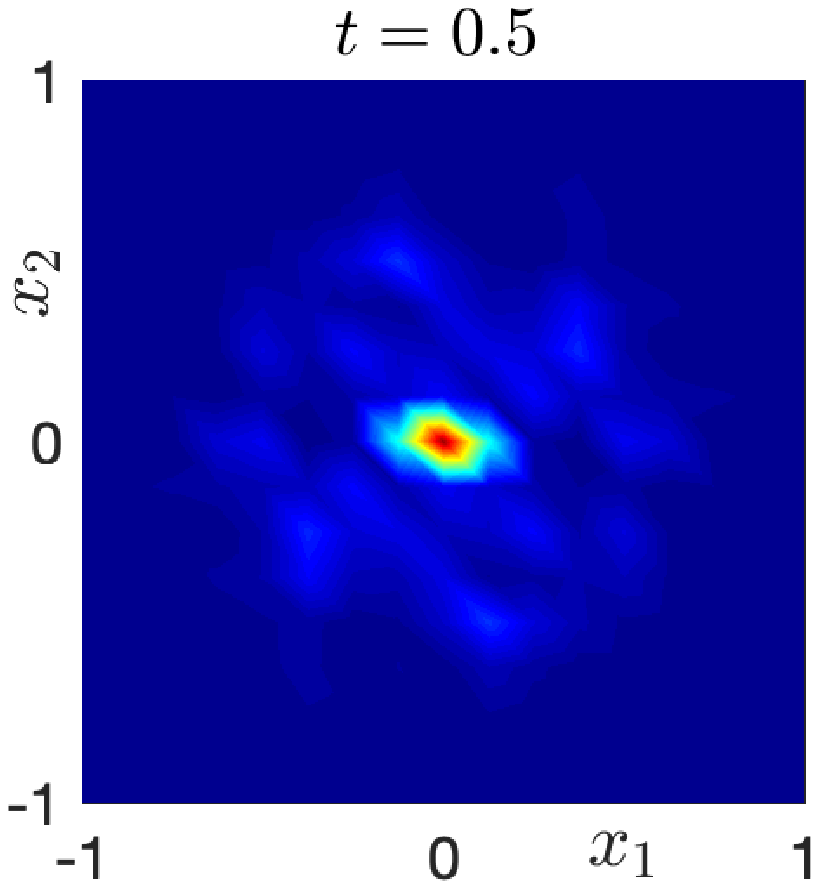}}
\subfigure{\includegraphics[width=0.15\textwidth]{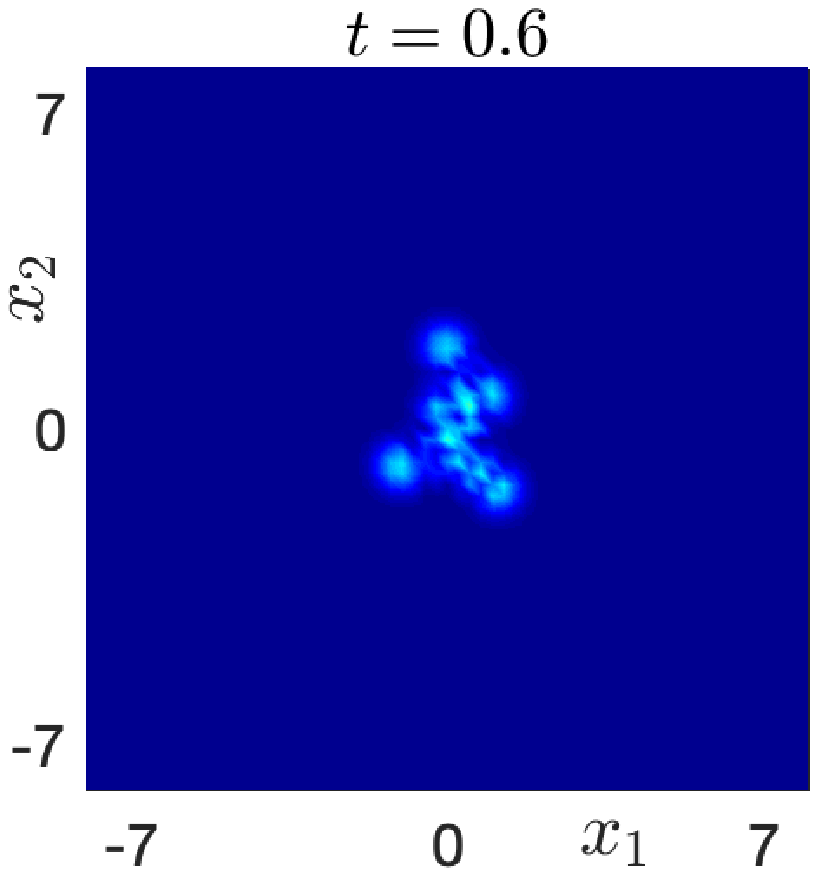}}
\subfigure{\includegraphics[width=0.15\textwidth]{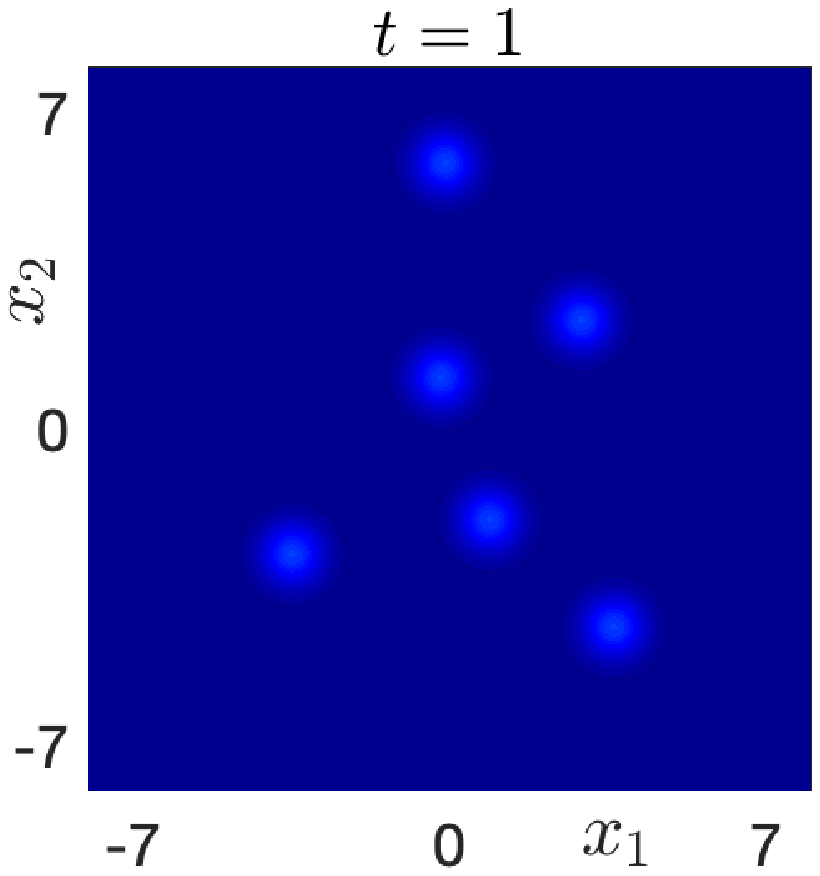}}
}
\mbox{
\subfigure{\includegraphics[width=0.5\textwidth,height=1.3cm]{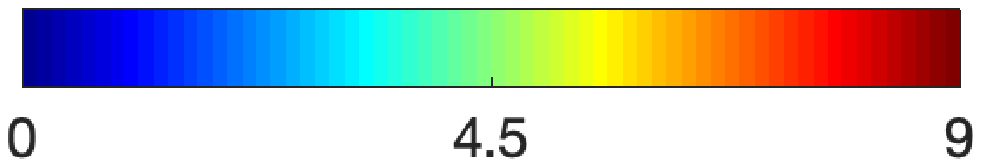}}
\subfigure{\includegraphics[width=0.5\textwidth,height=1.3cm]{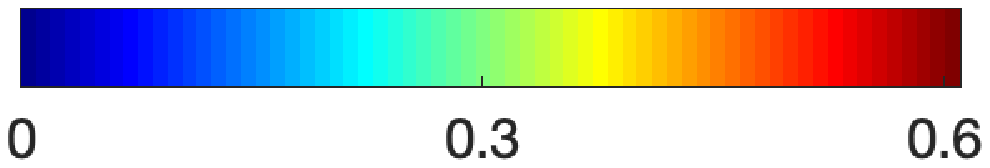}}
}}
\vspace{-0.8cm}
\caption{Contour plots of \textcolor{black}{$|\psi_1(x,t)|^2$} at different time $t$ for {\bf Cases 3-5} in Example \ref{eg:2d-case} (the top 4 rows)
and color bars of the contour plots at  $t=0.5$ (bottom left) and other time $t$ (bottom right).
 } \label{fig:density_case3to6}
\end{figure}

\section{Conclusion}  \label{sec:5} 
\setcounter{equation}{0}
In this survey paper, we have reviewed state-of-the-art results on the collective behaviors for two Lohe type aggregation models (the Lohe tensor model and the Schr\"{o}dinger-Lohe model). The former deals with the aggregation dynamics of tensors with the same rank and size, and it turns out to be a generalized model for   previously known first-order aggregation models such as the Kuramoto model, the swarm sphere model and the Lohe matrix model. Of course, the Lohe tensor model can be reduced to aggregation models on the hermitian sphere and nonsquare matrix group which are not known in previous literature. For the collective dynamics, we adopt two concepts of aggregation (complete state aggregation and practical aggregation). When all state aggregates to the same state, we call it as   complete state aggregation. This phenomenon occurs for a homogeneous ensemble in which all particles follow the same free flow. In contrast, when states are governed by different free flows, complete state aggregation cannot occur. Of course, the rate of state change (we call it velocity) can aggregate to the same value. At present, this strong estimate is not available for the aforementioned models yet. 

The latter  describes the collective aggregation of the coupled Schr\"{o}dinger equations for wave functions. For a special case, it can be reduced to the swarm sphere and Kuramoto models. For this model, we can also adopt the same universal approaches (Lyapunov functional approach and dynamical systems theory approach for two-point correlation functions). Similar to the Lohe tensor model, this model exhibits the same aggregation phenomena for homogeneous and heterogeneous ensembles. As we have already mentioned, we do not have a complete theory for dealing with a heterogeneous ensemble except for a weak aggregation estimate (practical aggregation). Thus, we would say that our reviewed results for a heterogeneous ensemble are still far from  completeness. This will be an interesting research direction for those who are interested in collective dynamics.

\section*{Acknowledgments}
The content of this paper is based on a lecture given by the first author at the Institute of Mathematical Sciences  of
National University of Singapore in December 2019. The authors acknowledge the Institute of Mathematical Sciences for generous support and especially thank Prof. Weizhu Bao for the invitation and his warm hospitality. The work of S.-Y. Ha is supported by the NRF grant (2020R1A2C3A01003881) and the work of   D. Kim was supported by the National Research Foundation of Korea (NRF) grant funded by the Korea government (MSIT) (No.2021R1F1A1055929).

\end{document}